\documentclass[12pt]{amsart}

\usepackage{graphicx} 
\usepackage{scrextend}
\usepackage{amsfonts, amsmath, amssymb,amsthm,comment}  
\usepackage{times,enumerate}
\usepackage{mathdots}%
\usepackage{nccmath}
\usepackage{mathtools}
\usepackage{ulem}
\usepackage{multicol}
    {\end{pmatrix}\end{medsize}}%
\usepackage{dsfont,bbm}
\usepackage{rotating}
\usepackage{lscape}
\usepackage{url}
\usepackage{multicol}
\usepackage{thmtools, thm-restate}
\usepackage{blkarray}
\usepackage{multicol}
\usepackage[margin=2.5cm]{geometry}
 \usepackage{hyperref}
\usepackage{cancel}
\usepackage{multirow}
\usepackage[baseline]{euflag}

\usepackage{pifont}
\usepackage{tikz}
\usetikzlibrary{topaths}
\tikzstyle{every picture}+=[remember picture,inner xsep=0,inner ysep=0.25ex]
\usetikzlibrary{calc}
\usepackage{kbordermatrix}
\renewcommand{\kbldelim}{(}
\renewcommand{\kbrdelim}{)}
\renewcommand{\kbrowstyle}{\displaystyle}
\renewcommand{\kbcolstyle}{\displaystyle}
\def\VR{\kern-\arraycolsep\strut\vrule &\kern-\arraycolsep}
\def\vr{\kern-\arraycolsep & \kern-\arraycolsep}
\makeatletter
\newcommand*{\sublabel}[1]{%
    \let\old@currentlabel\@currentlabel%
    \renewcommand{\@currentlabel}{\theenumii}%
    \label{#1}%
    \let\@currentlabel\old@currentlabel%
}
\makeatother
\allowdisplaybreaks

\graphicspath{ {./figs/} }

\newcommand*\interior[1]{\mathring{#1}}

\hypersetup{
    colorlinks,
    linkcolor={blue},
    citecolor={blue},
    urlcolor={blue}
}

\DeclareMathOperator{\Span}{span}

\DeclareMathOperator{\Rank}{rank}

\DeclareMathOperator{\sign}{sign}

\makeatletter
\def\widebreve{\mathpalette\wide@breve}
\def\wide@breve#1#2{\sbox\z@{$#1#2$}%
     \mathop{\vbox{\m@th\ialign{##\crcr
\kern0.08em\brevefill#1{0.8\wd\z@}\crcr\noalign{\nointerlineskip}%
                    $\hss#1#2\hss$\crcr}}}\limits}
\def\brevefill#1#2{$\m@th\sbox\tw@{$#1($}%
  \hss\resizebox{#2}{\wd\tw@}{\rotatebox[origin=c]{90}{\upshape(}}\hss$}
\makeatletter

\newcommand{\RR}{\mathbb R}

\newcommand{\NN}{\mathbb N}
\newcommand{\ZZ}{\mathbb Z}

\newcommand{\cT}{\mathcal T}

\newcommand{\cF}{\mathcal F}
\newcommand{\cB}{\mathcal B}

\newcommand{\cI}{\mathcal I}
\newcommand{\cL}{\mathcal L}
\newcommand{\cM}{\mathcal M}

\newcommand{\cA}{\mathcal A}

\newcommand{\cH}{\mathcal H}

\newcommand{\cZ}{\mathcal Z}
\newcommand{\cC}{\mathcal C}

\newcommand{\benu}{\begin{enumerate}}
\newcommand{\eenu}{\end{enumerate}}
\newcommand{\bop}{\begin{opomba}}
\newcommand{\eop}{\end{opomba}}

\newcommand{\supp}{\mathrm{supp}}

\newtheorem{theorem}{Theorem}[section]

\newtheorem{lemma}[theorem]{Lemma}
\newtheorem{proposition}[theorem]{Proposition}

\theoremstyle{definition}

\newtheorem{example}[theorem]{Example}

\newcommand{\mc}{\mathcal}
\newcommand{\mbb}{\mathbb}
\newcommand{\mbf}{\mathbf}

\definecolor{green-new}{rgb}{0.0, 0.5, 0.0}

\definecolor{cyan}{rgb}{0.0, 0.8, 1.0}

\theoremstyle{remark}
\newtheorem{remark}[theorem]{Remark}

\numberwithin{equation}{section}

\begin{document}

\title[TMP on reducible cubic curves I]{The truncated moment problem on reducible cubic curves I: 
Parabolic and Circular type relations}

\author[S. Yoo]{Seonguk Yoo}
\address{Department of Mathematics Education and RINS, Gyeongsang National University, Jinju, 52828, Korea. }
\email{seyoo@gnu.ac.kr}

\author[A. Zalar]{Alja\v z Zalar${}^{2,Q}$}
\address{Alja\v z Zalar, 
Faculty of Computer and Information Science, University of Ljubljana  \& 
Faculty of Mathematics and Physics, University of Ljubljana  \&
Institute of Mathematics, Physics and Mechanics, Ljubljana, Slovenia.}
\email{aljaz.zalar@fri.uni-lj.si}
\thanks{${}^2$Supported by the Slovenian Research Agency program P1-0288 and grants J1-50002, J1-2453, J1-3004.}
\thanks{${}^Q$This work was performed within the project COMPUTE, funded within the QuantERA II Programme that has received funding from the EU's H2020 research and innovation programme under the GA No 101017733 {\normalsize\euflag}}


\date{\today}
\maketitle

\begin{abstract}
    In this article we study the bivariate truncated moment problem (TMP) of degree $2k$ on reducible cubic curves. 
    First we show that every such TMP is equivalent after applying an affine linear transformation to one of 8 canonical forms of the curve.
    The case of the union of three parallel lines was solved in \cite{Zal22a}, while the degree 6 cases in \cite{Yoo17b}. 
    Second we characterize in terms of concrete numerical conditions the existence of the solution to the TMP on two of the remaining cases concretely, i.e., 
    a union of a line and a circle $y(ay+x^2+y^2)=0, a\in \RR\setminus \{0\}$, and 
    a union of a line and a parabola $y(x-y^2)=0$. 
    In both cases we also determine the number of atoms in a minimal representing measure.
 \end{abstract}


\section{
    Introduction
    }

Given a real $2$--dimensional sequence
	$$\beta\equiv\beta^{(2k)}=\{\beta_{0,0},\beta_{1,0},\beta_{0,1},\ldots,\beta_{2k,0},\beta_{2k-1,1},\ldots,
		\beta_{1,2k-1},\beta_{0,2k}\}$$
of degree $2k$ and a closed subset $K$ of $\RR^2$, the \textbf{truncated moment problem ($K$--TMP)} supported on $K$ for $\beta^{(2k)}$
asks to characterize the existence of a positive Borel measure $\mu$ on $\RR^2$ with support in $K$, such that
	\begin{equation}
		\label{moment-measure-cond}
			\beta_{i,j}=\int_{K}x^iy^j d\mu\quad \text{for}\quad i,j\in \ZZ_+,\;0\leq i+j\leq 2k.
	\end{equation}
If such a measure exists, we say that $\beta^{(2k)}$ has a representing measure 
	supported on $K$ and $\mu$ is its $K$--\textbf{representing measure ($K$--rm).}

In the degree-lexicographic order 
    $$\mathit{1},X,Y,X^2,XY,Y^2,\ldots,X^k,X^{k-1}Y,\ldots,Y^k$$ 
of rows and columns, the corresponding moment matrix to $\beta$
is equal to 
	\begin{equation}	
	\label{281021-1448}
            \mc M(k)\equiv
		\mc M(k;\beta):=
		\left(\begin{array}{cccc}
		\mc M[0,0](\beta) & \mc M[0,1](\beta) & \cdots & \mc M[0,k](\beta)\\
		\mc M[1,0](\beta) & \mc M[1,1](\beta) & \cdots & \mc M[1,k](\beta)\\
		\vdots & \vdots & \ddots & \vdots\\
		\mc M[k,0](\beta) & \mc M[k,1](\beta) & \cdots & \mc M[k,k](\beta)
		\end{array}\right),
	\end{equation}
where
	$$\mc M[i,j](\beta):=
		\left(\begin{array}{ccccc}
		\beta_{i+j,0} & \beta_{i+j-1,1} & \beta_{i+j-2,2} & \cdots & \beta_{i,j}\\
		\beta_{i+j-1,1} & \beta_{i+j-2,2} & \beta_{i+j-3,3} & \cdots & \beta_{i-1,j+1}\\
		\beta_{i+j-2,2} & \beta_{i+j-3,3} & \beta_{i+j-4,4} & \cdots & \beta_{i-2,j+2}\\
		\vdots & \vdots & \vdots & \ddots &\vdots\\
		\beta_{j,i} & \beta_{j-1,i+1} & \beta_{j-2,i+2} & \cdots & \beta_{0,i+j}\\
		\end{array}\right).$$
Let 
$\RR[x,y]_{\leq k}:=\{p\in \RR[x,y]\colon \deg p\leq k\}$ 
stand for the set of real polynomials in variables $x,y$ of total degree at most $k$.
For every $p(x,y)=\sum_{i,j} a_{i,j}x^iy^j\in \RR[x,y]_{\leq k}$ we define
its \textbf{evaluation} $p(X,Y)$ on the columns of the matrix $\mc M(k)$ by replacing each capitalized monomial $X^iY^j$
in $p(X,Y)=\sum_{i,j} a_{i,j}X^iY^j$ by the column of $\mc M(k)$, indexed by this monomial.
Then $p(X,Y)$ is a vector from the linear span of the columns of $\mc M(k)$. If this vector is the zero one, i.e., all coordinates are equal to 0, 
then we say $p$ is a \textbf{column relation} of $\mc M(k)$.
A column relation $p$ is \textbf{nontrivial}, if 
$p\not\equiv 0$.
We denote by
$\cZ(p):=\{(x,y)\in \RR^2\colon p(x,y)=0\}$,
the zero set of $p$.
We say that the matrix $\mc M(k)$ is \textbf{recursively generated (rg)} if for $p,q,pq\in \RR[x,y]_{\leq k}$ such that $p$ is a column relation of $\mc M(k)$, 
it follows that $pq$ is also a column relation of $\mc M(k)$.
The matrix $\mc M(k)$ is \textbf{$p$--pure}, if the only column relations of $\mc M(k)$ are those determined recursively by $p$. We call a sequence $\beta$ $p$--pure, if $\mc M(k)$ is $p$--pure.

A \textbf{concrete solution} to the TMP is a set of necessary and sufficient conditions for the existence of a $K$--representing measure $\mu$, 
that can be tested in numerical examples. 
Among necessary conditions, $\mc M(k)$ must be positive semidefinite (psd) and rg \cite{CF04,Fia95}, and by \cite{CF96}, if the support $\supp(\mu)$ of $\mu$ is a subset of 
$\cZ(p)$ for a polynomial $p\in \RR[x,y]_{\leq k}$, 
then $p$ is a column relation of $\mc M(k)$.
The bivariate $K$--TMP  is concretely solved in the following cases:
\begin{enumerate}
\item 
$K=\cZ(p)$ for a polynomial $p$ with $1\leq \deg p\leq 2$.

Assume that $\deg p=2$.
By applying an affine linear transformation it suffices to consider one of the canonical cases:  
    $x^2+y^2=1$,
    $y=x^2$,
    $xy=1$,
    $xy=0$,
    $y^2=y$.
    The case $x^2+y^2=1$ is equivalent to the univariate trigonometric moment problem, solved in \cite{CF02}.
    The other four cases were tackled
    in \cite{CF02, CF04, CF05,Fia15} by applying the far-reaching \textbf{flat extension theorem (FET)} \cite[Theorem 7.10]{CF96} (see also \cite[Theorem 2.19]{CF05b} and \cite{Lau05} for an alternative proof),
    which states that $\beta^{(2k)}$ admits a $(\Rank \mc M(k))$--atomic 
rm
if and only if $\mc M(k)$ is psd and admits a rank--preserving extension to a moment matrix $\mc M(k+1)$.
    For an alternative approach with shorter proofs compared to the original ones by reducing the problem to the univariate setting see
    \cite[Section 6]{BZ21} (for $xy=0$), \cite{Zal22a} (for $y^2=y$),
    \cite{Zal22b} (for $xy=1$) and \cite{Zal23} (for $y=x^2$).
    
    For $\deg p=1$ the solution is \cite[Proposition 3.11]{CF08} and uses the FET, 
    but can be also derived in the univariate setting (see \cite[Remark 3.3.(4)]{Zal23})
    \smallskip
\item 
	$K=\RR^2$, $k=2$ and $\mc M(2)$ is invertible.
    This case was first solved nonconstructively using convex geometry techniques in \cite{FN10} and later on constructively in \cite{CY16} by a novel rank reduction technique.
\item 
$K$ is one of $\cZ(y-x^3)$ \cite{Fia11,Zal21}, $\cZ(y^2-x^3)$ \cite{Zal21}, $\cZ(y(y-a)(y-b))$ \cite{Yoo17a,Zal22a}, $a,b\in \RR\setminus\{0\}$, $a\neq b$, or $\cZ(xy^2-1)$ \cite{Zal22b}. The main technique in \cite{Fia11} is the FET,
while in \cite{Zal21,Zal22a,Zal22b} the reduction to the univariate TMP is applied.\smallskip
\item $\mc M(k)$ has a special feature called \textit{recursive determinateness} \cite{CF13} or \textit{extremality} \cite{CFM08}.\smallskip
\item $\mc M(3)$ satisfies symmetric cubic column relations which can only  cause extremal moment problems. In order to satisfy the variety condition, another symmetric column relation must exist, and the solution  was obtained by checking consistency \cite{CY14}.\smallskip
\item  Non-extremal sextic TMP{\it s} 
with $\Rank \mc M(3)\leq 8$ and with finite or infinite algebraic varieties \cite{CY15}.
\smallskip
\item $\mc M(3)$ with reducible cubic column relations \cite{Yoo17b}.
\end{enumerate}

The solutions to the $K$--TMP, which are not concrete in the sense of definition from the previous paragraph, are known in the cases
    $K=\cZ(y-q(x))$ and $K=\cZ(yq(x)-1)$, where $q\in \RR[x]$. \cite[Section 6]{Fia11} gives a solution in terms of the bound on the degree $m$
for which the existence of a positive extension $\mc M(k+m)$ of $\mc M(k)$ is equivalent to the existence of a rm. In \cite{Zal23} the bound on $m$ 
is improved to $m=\deg q-1$ 
for curves of the form $y=q(x)$,
$\deg q\geq 3$,
and to $m=\ell+1$ for curves of the form $yx^\ell=1$, $\ell\in \NN\setminus\{1\}$.

References to some classical work on the TMP are monographs \cite{Akh65,AhK62,KN77}, while for a recent development in the area we refer a reader to \cite{Sch17}.
Special cases of the TMP have also been considered in \cite{Kim14,Ble15,Fia17,DS18,BF20,Kim21},
while \cite{Nie14} considers subspaces of the polynomial algebra and \cite{CGIK+} the TMP for commutative $\RR$--algebras.\\

The motivation for this paper was to solve the TMP concretely on some reducible cubic curves, other than the case of three parallel lines solved in \cite{Zal22a}. Applying an affine linear transformation we show that every such TMP is equivalent to the TMP on one of 8 canonical cases of reducible cubics of the form $yc(x,y)=0$, where $c\in \RR[x,y]$, $\deg c=2$.
In this article we solve the TMP for the cases 
$c(x,y)=ay+x^2+y^2$, $a\in \RR\setminus \{0\}$, 
and $c(x,y)=x-y^2$, which we call the \textit{circular} and the \textit{parabolic type}, respectively.
The main idea is to characterize the existence of a decomposition of $\beta$ into the sum $\beta^{(\ell)}+\beta^{(c)}$,
where $\beta^{(\ell)}=\{\beta_{i,j}^{(\ell)}\}_{i,j\in \ZZ_+,\; 0\leq i+j\leq 2k}$ and $\beta^{(c)}=\{\beta_{i,j}^{(c)}\}_{i,j\in \ZZ_+,\; 0\leq i+j\leq 2k}$
admit a $\RR$--rm and a $\cZ(c)$--rm, respectively.
Due to the form of the cubic $yc(x,y)=0$, it turns out that all but two moments of $\beta^{(\ell)}$ and $\beta^{(c)}$ are not already fixed by the original sequence, i.e.,                     
        $\beta_{0,0}^{(\ell)}$, 
        $\beta_{1,0}^{(\ell)}$, 
        $\beta_{0,0}^{(c)}$, 
        $\beta_{1,0}^{(c)}$
in the circular type case 
and
        $\beta_{0,0}^{(\ell)}$, 
        $\beta_{2k,0}^{(\ell)}$, 
        $\beta_{0,0}^{(c)}$, 
        $\beta_{2k,0}^{(c)}$
in the parabolic type case.
Then, by an involved analysis, the characterization of the existence of a decomposition $\beta=\beta^{(\ell)}+\beta^{(c)}$ can be done in both cases. We also characterize the number of atoms in a minimal representing measure, i.e., a measure with the minimal number of atoms in the support.


\subsection{ Readers Guide}
The paper is organized as follows. 
In Section \ref{preliminiaries} we present 
some preliminary results needed to establish the main results of the paper.
In Section \ref{case-reduction} we show that to
solve the TMP on every reducible cubic curve it is enough to consider 8 canonical type relations (see Proposition \ref{cases}).
In Section \ref{Section-common-approach}
we present the general procedure for solving the TMP on all but one of the canonical types and prove some results  
that apply to them.
Then in Sections \ref{circular} and \ref{parabolic}
we specialize to the circular and the parabolic type relations and solve them concretely (see Theorems \ref{221023-1854} and \ref{131023-0847}). In both cases we show, by numerical examples, that there are  pure sequences $\beta^{(6)}$ with a psd $\mc M(3)$ but without a rm (see Examples \ref{ex-line-plus-circle-no-rm} and \ref{ex-line-plus-parabola-no-rm}).


\section{Preliminaries}
\label{preliminiaries}

We write $\RR^{n\times m}$ for the set of $n\times m$ real matrices. For a matrix $M$ 
we call the linear span of its columns a \textbf{column space} and denote it by $\cC(M)$.
The set of real symmetric matrices of size $n$ will be denoted by $S_n$. 
For a matrix $A\in S_n$ the notation $A\succ 0$ (resp.\ $A\succeq 0$) means $A$ is positive definite (pd) (resp.\ positive semidefinite (psd)).
We write $\mathbf{0}_{t_1,t_2}$ for a $t_1\times t_2$ matrix with only zero entries and 
$\mathbf{0}_{t}=\mathbf{0}_{t,t}$ for short, where $t_1,t_2,t\in \NN$.
The notation
    $E^{(\ell)}_{i,j}$, 
    $\ell\in \NN$,
    stands for the usual $\ell\times \ell$ coordinate matrix with the only nonzero entry at the position $(i,j)$, which is equal to 1.

In the rest of this section let $k\in \NN$ and $\beta\equiv \beta^{ (2k)}=\{\beta_{i,j}\}_{i,j\in \ZZ_+,\; 0\leq i+j\leq 2k}$ be a bivariate sequence of degree $2k$.

\subsection{Moment matrix}
Let $\mc M(k)$ be the moment matrix of $\beta$ (see \eqref{281021-1448}).
Let $Q_1, Q_2$ be subsets of the set $\{X^iY^j\colon  i,j \in \ZZ_+,\; 0\leq i+j\leq k\}$.
We denote by 
$(\mc M(k))_{Q_1,Q_2}$ 
the submatrix of $\mc M(k)$ consisting of the rows indexed by the elements of $Q_1$
and the columns indexed by the elements of $Q_2$. In case $Q:=Q_1=Q_2$, we write 
$(\mc M(k))_{Q}:=(\mc M(k))_{Q,Q}$
for short. 

\subsection{Affine linear transformations} \label{affine linear-trans}

The existence of representing measures is invariant under invertible affine linear transformations of the form 
\begin{equation}
\label{alt}
    \phi(x,y)=(\phi_1(x,y),\phi_2(x,y)):=(a+bx+cy,d+ex+fy),\; (x,y)\in \RR^{2},
\end{equation}
$a,b,c,d,e,f\in \RR$ with $bf-ce \neq 0$.
Namely, let  $L_{\beta}:\mbb{R}[x,y]_{\leq 2k}\to \RR$ be a \textbf{Riesz functional} of the sequence $\beta$ defined by 
$$
	L_{\beta}(p):=\sum_{\substack{i,j\in \ZZ_+,\\ 0\leq i+j\leq 2k}} a_{i,j}\beta_{i,j},\qquad \text{where}\quad p=
	\sum_{\substack{i,j\in \ZZ_+,\\ 0\leq i+j\leq 2k}} a_{i,j}x^iy^j.
$$
We define $\widetilde \beta=\{\widetilde \beta_{i,j}\}_{i,j\in \ZZ_+,\; 0\leq i+j\leq 2k}$ by
	$$\widetilde \beta_{i,j}=L_{\beta}(\phi_1(x,y)^i \cdot \phi_2(x,y)^j).$$
By \cite[Proposition 1.9]{CF04}, $\beta$ admits a ($r$--atomic) $K$--rm if and only if $\widetilde \beta$ admits a ($r$--atomic) $\phi(K)$--rm.
We write $\widetilde \beta=\phi(\beta)$ and $\mc M(k;\widetilde \beta)=\phi(\mc M(k;\beta))$.

\subsection{Generalized Schur complements}\label{SubS2.1}
Let 
	\begin{equation*}
		M=\left( \begin{array}{cc} A & B \\ C & D \end{array}\right)\in \RR^{(n+m)\times (n+m)}
	\end{equation*}
be a real matrix where $A\in \RR^{n\times n}$, $B\in \RR^{n\times m}$, $C\in \RR^{m\times n}$  and $D\in \RR^{m\times m}$.
The \textbf{generalized Schur complement} \cite{Zha05} of $A$ (resp.\ $D$) in $M$ is defined by
	$$M/A=D-CA^\dagger B\quad(\text{resp.}\; M/D=A-BD^\dagger C),$$
where $A^\dagger $ (resp.\ $D^\dagger $) stands for the Moore--Penrose inverse of $A$ (resp.\ $D$). 

The following lemma will be frequently used in the proofs.

\begin{lemma}
	\label{140722-1055} 
	Let $n,m\in \NN$ and
		\begin{equation*}
			M=\left( \begin{array}{cc} A & B \\ B^{T} & C\end{array}\right)\in S_{n+m},
		\end{equation*} 
	where $A\in S_n$, $B\in \RR^{n\times m}$ and $C\in S_m$.
	If $\Rank M=\Rank A$, then the matrix equation
	\begin{equation}\label{140722-1055-eq}
		\begin{pmatrix}
			A\\
			B^T
		\end{pmatrix}
		W
		=
		\begin{pmatrix}
			B\\
			C
		\end{pmatrix},
	\end{equation}
	where $W\in \RR^{n\times m}$, is solvable and 
	the solutions are precisely the solutions of the matrix equation $AW=B$.
	In particular, $W=A^{\dagger}B$ satisfies \eqref{140722-1055-eq}.
\end{lemma}

\begin{proof}
	The assumption $\Rank M=\Rank A$ implies that 
	\begin{equation}\label{140722-1056}
		\begin{pmatrix}
			A\\
			B^T
		\end{pmatrix}W	
		=
		\begin{pmatrix}
			AW\\
			B^TW
		\end{pmatrix}
		=
		\begin{pmatrix}
			B\\
			C
		\end{pmatrix}
	\end{equation}
	for some $W\in \RR^{n\times m}$. So the equation \eqref{140722-1055-eq} is solvable.
	In particular, $AW=B$. It remains to prove that any solution $W$ to $AW=B$ is also a solution to $\eqref{140722-1056}$.
	Note that all the solutions of the equation $A\widetilde W=B$ are 
	\begin{equation}\label{140722-1130}
		\widetilde W=A^\dagger B+Z,
	 \end{equation}
	where each column of $Z\in \RR^{n\times m}$ is an arbitrary vector from $\ker A$.
	So $W$ satisfiying $\eqref{140722-1056}$ is also of the form $A^\dagger B+Z_0$ for some $Z_0\in \RR^{n\times m}$ with columns belonging to $\ker A$.
	We have that
	\begin{equation}\label{140722-1129}
		C=B^TW=B^T(A^\dagger B+Z_0)=B^TA^\dagger B+B^TZ_0=B^TA^\dagger B,
	\end{equation}
 	where we used the fact that each column of $B$ belongs to $\cC(A)$ and $\ker(A)^\perp = \cC(A).$
	Replacing $W$ with any $\widetilde W$ of the form \eqref{140722-1130} in the calculation \eqref{140722-1129}
	gives the same result, which proves the statement of the proposition.
\end{proof}

The following theorem is a characterization of psd $2\times 2$ block matrices. 

\begin{theorem}[{\cite{Alb69}}]
    \label{block-psd} 
	Let 
		\begin{equation*}
			M=\left( \begin{array}{cc} A & B \\ B^{T} & C\end{array}\right)\in S_{n+m}
		\end{equation*} 
	be a real symmetric matrix where $A\in S_n$, $B\in \RR^{n\times m}$ and $C\in S_m$.
	Then: 
	\begin{enumerate}
		\item 
                \label{021123-1702}
                The following conditions are equivalent:
			\begin{enumerate}
			     \item 
                    \label{pt1-281021-2128} 
                        $M\succeq 0$.
                    \smallskip
			     \item 
                    \label{pt2-281021-2128} 
                        $C\succeq 0$,         
                    $\cC(B^T)\subseteq\cC(C)$ and $M/C\succeq 0$.
                    \smallskip
				 \item 
                    \label{pt3-281021-2128}
                        $A\succeq 0$,    
                        $\cC(B)\subseteq\cC(A)$ and $M/A\succeq 0$.
			\end{enumerate}
                \smallskip
		\item 
                \label{prop-2604-1140-eq2}
                If $M\succeq 0$, then 
		      $$\Rank M= \Rank A+\Rank M/A=\Rank C+\Rank M/C.$$
	\end{enumerate}
\end{theorem}


\subsection{Extension principle}

\begin{proposition}
	\label{extension-principle}
	Let $\cA\in S_n$ be positive semidefinite, 
	$Q$ a subset of the set $\{1,\ldots,n\}$
	and	
	$\cA|_Q$ the restriction of $\cA$ to the rows and columns from the set $Q$. 
	If $\cA|_Qv=0$ for a nonzero vector $v$,
	then $\cA\widehat v=0$, where $\widehat{v}$ is a vector with the only nonzero entries in the rows from $Q$ and such that the restriction 
	$\widehat{v}|_Q$ to the rows from $Q$ equals to $v$. 
\end{proposition}

\begin{proof}
See \cite[Proposition 2.4]{Fia95} or \cite[Lemma 2.4]{Zal22a} for an alternative proof.
\end{proof}


\subsection{Partially positive semidefinite matrices and their completions}\label{SubS2.4}

A \textbf{partial matrix} $A=(a_{i,j})_{i,j=1}^n$ is a matrix of real numbers $a_{i,j}\in \RR$, where some of the entries are not specified. 

A partial symmetric matrix $A=(a_{i,j})_{i,j=1}^n$ is 
\textbf{partially positive semidefinite (ppsd)} 
(resp.\ \textbf{partially positive definite (ppd)}) 
if the following two conditions hold:
\begin{enumerate} 
  \item $a_{i,j}$ is specified if and only if $a_{j,i}$ is specified and $a_{i,j}=a_{j,i}$.
  \item All fully specified principal minors of $A$ are psd (resp.\ pd). 
\end{enumerate}

For $n\in \NN$ write $[n]:=\{1,2,\ldots,n\}$.
We denote by 
$A_{Q_1,Q_2}$ 
the submatrix of 
$A\in \RR^{n\times n}$ 
consisting of the rows indexed by the elements of $Q_1\subseteq [n]$
and the columns indexed by the elements of $Q_2\subseteq [n]$. 
In case $Q:=Q_1=Q_2$, we write 
$A_{Q}:=A_{Q,Q}$
for short. 

It is well-known that a ppsd matrix 
$A(\mathbf{x})$ of the form as in Lemma \ref{psd-completion} below admits a psd completion
(This follows from the fact that the corresponding graph is chordal, see e.g., \cite{GJSW84,Dan92,BW11}). 
Since we will need an additional information about the rank of the completion $A(x_0)$ and 
the explicit interval of all possible $x_0$ for our results, we give a proof of Lemma \ref{psd-completion} based on 
the use of generalized Schur complements.

\begin{lemma}
	\label{psd-completion}
	Let 
		$A(\mathbf{x})$	
	be a partially positive semidefinite symmetric matrix of size $n\times n$ with the missing entries 
        in the positions $(i,j)$ and $(j,i)$, $1\leq i<j\leq n$.
	Let 
	\begin{align*}
		A_1 &= (A(\mathbf{x}))_{[n]\setminus \{i,j\}},\; 
            a=(A(\mathbf{x}))_{[n]\setminus \{i,j\},\{i\}},\; 
            b=(A(\mathbf{x}))_{[n]\setminus \{i,j\},\{j\}},\;
		\alpha=(A(\mathbf{x}))_{i,i},\;
		\gamma=(A(\mathbf{x}))_{j,j}.
	\end{align*}
	Let
		$$A_2=(A(\mathbf{x}))_{[n]\setminus \{j\}}	
			=\begin{pmatrix}
				A_1 & a \\
				a^T & \alpha
			\end{pmatrix}\in S_{n-1},\qquad
		A_3=(A(\mathbf{x}))_{[n]\setminus \{i\}}
			=\begin{pmatrix}
				A_1 & b \\
				b^T & \gamma
			\end{pmatrix}\in S_{n-1},$$
	and
		$$x_{\pm}:=b^TA_1^{\dagger}a\pm \sqrt{(A_2/A_1)(A_3/A_1)}\in \RR.$$
	Then: 
	\begin{enumerate}[(i)]
		\item\label{psd-comp-pt1} $A(x_{0})$ is positive semidefinite if and only if $x_0\in [x_-,x_+]$. 
	 	\item\label{psd-comp-pt2} 
			$$\Rank A(x_0)=
			\left\{\begin{array}{rl}
			\max\big\{\Rank A_2, \Rank A_3\big\},& \text{for}\;x_0\in \{x_-,x_+\},\\[0.5em]
			\max\big\{\Rank A_2, \Rank A_3\big\}+1,& \text{for}\;x_0\in (x_-,x_+).
			\end{array}\right.$$
            \item 
            The following statements are equivalent:
	\begin{enumerate}
		\item $x_-=x_+$.\smallskip
		\item $A_2/A_1=0$ or $A_3/A_1=0$.\smallskip
		\item $\Rank A_2=\Rank A_1$ or $\Rank A_3=\Rank A_1$.
	\end{enumerate}
	\end{enumerate}
\end{lemma}

\begin{proof}
We write
\begin{align*}
A(\mathbf{x})
    &=
\begin{pmatrix}
    A_{11} & a_{12} & A_{13} & a_{14} & A_{15} \\[0.2em]
    (a_{12})^T & \alpha & (a_{23})^T & \mathbf{x} & (a_{25})^T \\[0.2em]
    (A_{13})^T & a_{23} & A_{33} & a_{34} & a_{35}\\[0.2em]
    (a_{14})^T & \mathbf{x} & (a_{34})^T & \gamma & (a_{45})^T \\[0.2em]
    (A_{15})^T & a_{25} & (A_{35})^T & a_{45} & A_{55}
\end{pmatrix}\\[0.5em]
    &\in
    \begin{pmatrix}
    S_{i-1} & \RR^{(i-1)\times 1} & \RR^{(i-1)\times (j-i-1)} & \RR^{(i-1)\times 1} & \RR^{(i-1)\times (n-j)}\\
    \RR^{1\times (i-1)} & \RR & \RR^{1\times (j-i-1)} & \RR & \RR^{1\times (n-j)}\\
    \RR^{(j-1-1)\times (i-1)} & \RR^{(j-i-1)\times 1} & S_{j-i-1} & \RR^{(j-i-1)\times 1} & \RR^{(j-i-1)\times (n-j)}\\
    \RR^{1\times (i-1)} & \RR & \RR^{1\times (j-i-1)} & \RR & \RR^{1\times (n-j)}\\
    \RR^{(n-j)\times (i-1)} & \RR^{(n-j)\times 1} & \RR^{(n-j)\times (j-i-1)} & \RR^{(n-j)\times 1} & S_{n-j}
    \end{pmatrix}
\end{align*}
Let $P$ be a permutation matrix, which changes the order of columns to 
    $$1,2,\ldots,i-1,i+1,\ldots,j-1,j+1,\ldots,n,i,j.$$
Then 
    $$
    P^TA(\mathbf{x})P
    =
    \begin{pmatrix}
    A_{11} & A_{13} & A_{15} & a_{12} & a_{14} \\[0.2em]
    (A_{13})^T & A_{33} & A_{35} & a_{23} & a_{34} \\[0.2em]
    (A_{15})^T & (A_{35})^T & A_{55} & a_{25} & a_{45}\\[0.2em]
    (a_{12})^T & (a_{23})^T & (a_{25})^T & \alpha & \mathbf{x} \\[0.2em]
    (a_{14})^T & (a_{34})^T & (a_{45})^T & \mathbf{x} & \gamma
    \end{pmatrix}
    $$
Note that
\begin{equation}
    \label{021123-1626}
    P^TA(\mathbf{x})P
    =
    \begin{pmatrix}
        A_{1} & a & b\\[0.2em]
        a^T & \alpha & \mathbf{x} \\[0.2em]
        b^T & \mathbf{x} & \gamma
    \end{pmatrix}
    \qquad \text{and} \qquad
    P^TA(\mathbf{x})P\succeq 0\; \Leftrightarrow\; A(\mathbf{x})\succeq 0.
\end{equation}
Lemma \ref{psd-completion} with the missing entries in the positions $(n-1,n)$ and $(n,n-1)$
was proved in \cite[Lemma 2.11]{Zal21}
	using computations with generalized Schur complements under one additional assumption:
		\begin{equation}\label{assump-on-ranks-050822}
			A_1\;\text{is invertible}\quad \text{or} \quad \Rank A_1=\Rank A_2.
		\end{equation}
Here we explain why the assumption \eqref{assump-on-ranks-050822} can be removed from \cite[Lemma 2.11]{Zal21}.
The proof of \cite[Lemma 2.11]{Zal21} is separated into two
cases: $A_2/A_1>0$ and $A_2/A_1=0$. 
The case $A_2/A_1=0$ does not use \eqref{assump-on-ranks-050822}.
Assume now that $A_2/A_1>0$ or equivalently $\Rank A_2>\Rank A_1$.
Invertibility of $A_1$ (and by $A_2/A_1>0$ also $A_2$ is invertible)
is used in the proof of \cite[Lemma 2.11]{Zal21} for the application of the quotient formula (\cite{CH69})
\begin{equation}\label{050822-2137}
	(A(x)/A_2)=\big(A(x)/A_1\big)\big/\big(A_2/A_1\big),
\end{equation}
where
$$
A(x)/A_1=
\begin{pmatrix}		
	A_2/A_1 & \begin{pmatrix} A_1 & b \\ a^T & x \end{pmatrix}\big/A_1\\
	\begin{pmatrix} A_1 & a \\ b^T & x \end{pmatrix}\big/A_1 & A_3/A_1
\end{pmatrix}
$$
However, the formula \eqref{050822-2137} has been generalized \cite[Theorem 4]{CHM74}
to noninvertible $A_1$, $A_2$, where all Schur complements are the generalized ones, 
under the conditions: 
\begin{equation}
    \label{021123-1643}
    \begin{pmatrix}b & x\end{pmatrix}^T \in \cC(A_2)
    \qquad\text{and}\qquad
	a \in \cC(A_1).
\end{equation}
So if we show that the conditions \eqref{021123-1643} hold, 
the same proof as in  \cite[Lemma 2.11]{Zal21} can be applied in the case $A_1$ is singular.
From $A_2$ (resp.\ $A_3$) being psd, $a \in \cC(A_1)$ (resp.\ $b\in \cC(A_1)$) follows 
by Theorem \ref{block-psd}, used for $(M,A):=(A_2,A_1)$ (resp.\ $(M,A):=(A_3,A_1)$).
The assumption $A_2/A_1>0$ implies that 
	$\begin{pmatrix}a & \alpha\end{pmatrix}^T \notin \cC(\begin{pmatrix} A_1 & a^T \end{pmatrix}^T)$.
Since $a \in \cC(A_1)$, it follows that	$\begin{pmatrix}0 & 1\end{pmatrix}^T \in \cC(A_2)$.
Hence, $\begin{pmatrix}b & x\end{pmatrix}^T \in \cC(A_2)$ for every $x\in \RR$,
which concludes the proof of \eqref{021123-1643}.
\end{proof}


\subsection{Hamburger TMP}\label{SubS2.2}
Let $k\in \NN$.
For 
		$v=(v_0,\ldots,v_{2k} )\in \RR^{2k+1}$
we define the corresponding Hankel matrix as
	\begin{equation}\label{vector-v}
		A_{v}:=\left(v_{i+j} \right)_{i,j=0}^k
					=\left(\begin{array}{ccccc} 
							v_0 & v_1 &v_2 & \cdots &v_k\\
							v_1 & v_2 & \iddots & \iddots & v_{k+1}\\
							v_2 & \iddots & \iddots & \iddots & \vdots\\
							\vdots 	& \iddots & \iddots & \iddots & v_{2k-1}\\
							v_k & v_{k+1} & \cdots & v_{2k-1} & v_{2k}
						\end{array}\right)
					\in S_{k+1}.
	\end{equation}
We denote by 
	$\mbf{v_j}:=\left( v_{j+\ell} \right)_{\ell=0}^k$ the $(j+1)$--th column of $A_{v}$, $0\leq j\leq k$, i.e.,
		$$A_{v}=\left(\begin{array}{ccc} 
								\mbf{v_0} & \cdots & \mbf{v_k}
							\end{array}\right).$$
As in \cite{CF91}, the \textbf{rank} of $v$, denoted by $\Rank v$, is defined by
	$$\Rank v=
	\left\{\begin{array}{rl} 
		k+1,&	\text{if } A_{v} \text{ is nonsingular},\\[0.2em]
		\min\left\{\bf{i}\colon \bf{v_i}\in \Span\{\bf{v_0},\ldots,\bf{v_{i-1}}\}\right\},&		\text{if } A_{v} \text{ is singular}.
	 \end{array}\right.$$

For $m\leq k$ we denote
the upper left--hand corner $\left(v_{i+j} \right)_{i,j=0}^m\in S_{m+1}$ of $A_{v}$ of size $m+1$ by $A_{v}(m)$. 
A sequence $v$ is called \textbf{positively recursively generated (prg)} if for  $r=\Rank v$ the following two conditions hold:
		\begin{itemize}
			\item $A_v(r-1)\succ 0$.
                \smallskip
			\item If $r<k+1$, denoting 
				\begin{equation}\label{rg-coefficients-2809-2000}
					(\varphi_0,\ldots,\varphi_{r-1}):=A_{v}(r-1)^{-1}(v_r,\ldots,v_{2r-1})^{T},
				\end{equation}
				the equality
				\begin{equation}\label{recursive-generation}
					  v_j=\varphi_0v_{j-r}+\cdots+\varphi_{r-1}v_{j-1}
				\end{equation}
				holds for $j=r,\ldots,2k$.
		 \end{itemize}

The solution to the $\RR$--TMP is the following.

\begin{theorem} [{\cite[Theorems 3.9--3.10]{CF91}}]
    \label{Hamburger} 
	For $k\in \NN$ and $v=(v_0,\ldots,v_{2k})\in \RR^{2k+1}$ with $v_0>0$, the following statements are equivalent:
\begin{enumerate}	
	\item 
            There exists a $\RR$--representing measure for $\beta$.\smallskip
        \item 
            There exists a $(\Rank A_v)$--atomic $\RR$--representing measure for $\beta$.\smallskip
	\item\label{pt5-v2206} 
            $A_v$ is positive semidefinite
            and one of the following holds:
            \smallskip
            \begin{enumerate}
                \item 
                $A_v(k-1)$ is positive definite.
                \smallskip
                \item 
                $\Rank A_v(k-1)=\Rank A_v$.
            \end{enumerate}
                \smallskip
        \item
            $v$ is positively recursively generated.
\end{enumerate}
\end{theorem}


\subsection{TMP on the unit circle}

The solution to the $\cZ(x^2+y^2-1)$--TMP is the following.

\begin{theorem}[{\cite[Theorem 2.1]{CF02}}]
	\label{circle-TMP}
	Let $p(x,y)=x^2+y^2-1$ and 
	$\beta:=\beta^{(2k)}=(\beta_{i,j})_{i,j\in \ZZ_+,i+j\leq 2k}$, where $k\geq 2$. 
	Then the following statements are equivalent:
	\begin{enumerate}
		\item
		      $\beta$ has a $\mc Z(p)$--representing measure.
            \smallskip
		\item
				$\beta$ has a $(\Rank \mc M(k))$--atomic $\mc Z(p)$--representing measure.
            \smallskip
		\item 
			$\mc M(k)$ is positive semidefinite
   and the relations $\beta_{2+i,j}+\beta_{i,2+j}=\beta_{i,j}$ hold for every $i,j\in \ZZ_+$ with $i+j\leq 2k-2$.
	\end{enumerate}
\end{theorem}


\subsection{Parabolic TMP}

We will need the following solution to the parabolic TMP (see \cite[Theorem 3.7]{Zal23}).

\begin{theorem}
    \label{131023-1211} 
    Let $p(x,y)=x-y^2$ and 
	$\beta:=\beta^{(2k)}=(\beta_{i,j})_{i,j\in \ZZ_+,i+j\leq 2k}$, where $k\geq 2$. 
    Let
        \begin{equation*}
			\cB=\{\mathit{1},Y,X,XY,X^2,X^2Y,\ldots,
                    X^i,X^iY,\ldots,X^{k-1},X^{k-1}Y,X^k\}.
	\end{equation*}
    
	Then the following statements are equivalent:
	\begin{enumerate}
		\item\label{131023-1346-pt0} 
		      $\beta$ has a $\mc Z(p)$--representing measure.
            \smallskip
		\item
				$\beta$ has a $(\Rank \mc M(k))$--atomic $\mc Z(p)$--representing measure.
            \smallskip
		\item 
            \label{parabolic--pt3}
			$\mc M(k)$ is positive semidefinite,
                the relations $\beta_{1+i,j}=\beta_{i,2+j}$ hold for every $i,j\in \ZZ_+$ with $i+j\leq 2k-2$
                and one of the following statements holds:
            \begin{enumerate}
            \smallskip
		\item\label{131023-1346-pt1} 
			$\big(\mc M(k)\big)_{\cB\setminus \{X^k\}}$ 
                is positive definite.
            \smallskip
		\item\label{131023-1346-pt2} 
			$\Rank \big(\mc M(k)\big)_{\cB\setminus \{X^k\}}
                =
                \Rank \mc M(k).$
		\end{enumerate}
            \smallskip
            \item\label{131023-1346-pt4} 
                The relations
                $\beta_{1+i,j}=\beta_{i,2+j}$ 
                hold for every $i,j\in \ZZ_+$ with $i+j\leq 2k-2$
                and 
                $\gamma=(\gamma_0,\ldots,\gamma_{4k})$,
                defined by
                $\gamma_i=\beta_{\lfloor \frac{i}{2}\rfloor,i\; \mathrm{mod}\; 2}$,
                admits a $\RR$--representing measure.
        \end{enumerate}
\end{theorem}

\begin{remark}
The equivalence $\eqref{parabolic--pt3}\Leftrightarrow\eqref{131023-1346-pt4}$
is part of the proof of \cite[Theorem 3.7]{Zal23}.
\end{remark}


\section{
    TMP on reducible cubics - case reduction
    }
\label{case-reduction}

In this section we show that to solve the TMP
on reducible cubic curves it suffices, after applying an affine linear transformation, to solve the TMP on 8 canonical forms of curves. 

\begin{proposition}
\label{cases}
Let $k\in \RR$ and 
$\beta:=
\beta^{(2k)}=
(\beta_{i,j})_{i,j\in \ZZ_+,i+j\leq 2k}$.
Assume $\mc M(k;\beta)$ does not satisfy any nontrivial column relation between columns indexed by monomials of degree at most 2, but it 
satisfies a column relation 
$p(X,Y)=\bf{0}$, where $p\in \RR[x,y]$ is a reducible polynomial with $\deg p=3$.
If $\beta$ admits a representing measure, then there exists an invertible affine linear transformation $\phi$ of the form \eqref{alt}
such that the moment matrix 
$\phi\big(\mc M(k;\beta)\big)$
satisfies a column relation $q(x,y)=0$,
where $q$ has one of the following forms:

\renewcommand{\descriptionlabel}[1]{\hspace{\labelsep}\textit{#1}}

\begin{description}
    \item[Parallel lines type:] 
        $q(x,y)=y(a+y)(b+y)$,
        $a,b\in \RR\setminus \{0\}$,
        $a\neq b$.
        \smallskip
    \item[Circular type:]
        $q(x,y)=y(ay+x^2+y^2)$, 
        $a\in \RR\setminus\{0\}$.
        \smallskip
    \item[Parabolic type:] 
        $q(x,y)=y(x-y^2).$
        \smallskip
    \item[Hyperbolic type 1:] 
        $q(x,y)=y(1-xy)$.
        \smallskip
    \item[Hyperbolic type 2:] 
        $q(x,y)=y(x+y+axy)$,
        $a\in \RR\setminus\{0\}$.
        \smallskip
    \item[Hyperbolic type 3:] 
        $q(x,y)=y(ay+x^2-y^2)$,
        $a\in \RR$.
        \smallskip
    \item[Intersecting lines type:]         
        $q(x,y)=yx(y+1)$, 
        \smallskip
    \item[Mixed type:]  
        $q(x,y)=y(1+ay+bx^2+cy^2)$, 
        $a,b,c\in \RR$, $b\neq 0$.
\end{description}
\end{proposition}

\begin{remark} 
The name of the types of the form $q$ in 
Proposition \ref{cases} comes from the type
of the conic $\frac{q(x,y)}{y}=0$.
The conic
    $x+y+axy=0$,
    $a\in \RR\setminus \{0\}$,
is a hyperbola,
since the discriminant $a^2$ is positive. 
Similarly,  
the conic
    $ay+x^2-y^2=0$,
    $a\in \RR$,
is a hyperbola, since its discriminant is equal to 4.
Clearly, the conic
    $ay+x^2+y^2=0$,
    $a\in \RR$,
is a circle with the center 
$(0,-\frac{a}{2})$
and radius $\frac{a}{2}$.
\end{remark}

Now we prove Proposition \ref{cases}.

\begin{proof}[Proof of Proposition \ref{cases}]
    Since $p(x,y)$ is reducible, it is of the form
    $p=p_1p_2$, where 
    \begin{align*}
        p_1(x,y)
        &=a_0+a_1x+a_2y
        \quad
        \text{with } 
            a_i\in \RR,\;
            (a_1,a_2)\neq (0,0),\\
        p_2(x,y)
        &=b_0+b_1x+b_2y+b_3x^2+b_4xy+b_5y^2
        \quad
        \text{with } b_i\in \RR, \; (b_3,b_4,b_5)\neq (0,0,0).
    \end{align*}
    Without loss of generality we can assume that $a_2\neq 0$, since otherwise we 
    apply the alt $(x,y)\mapsto (y,x)$ to exchange the
    roles of $x$ and $y$.
    Since $a_2\neq 0$, the alt 
    $$
        \phi_1(x,y)=(x,a_0+a_1x+a_2y)
    $$
    is invertible and hence:
    \begin{align}
    \label{relation}
    \begin{split}
        &\text{A sequence }
        \phi_1(\beta)
        \text{ has a moment matrix }
        \phi_1\big(\mc M(k;\beta)\big)
        \text{ satisfying the column relation}\\
        &
        c_0Y+c_1X+c_2Y^2+c_3X^2Y+c_4XY^2+c_5Y^3=\mathbf{0}
        \text{ with }c_i\in\RR,\; 
        (c_3,c_4,c_5)\neq (0,0,0).
    \end{split}
    \end{align}
    We separate two cases according to the value of $c_3$.\\

    \noindent
    \textbf{Case 1: $c_3=0$.} 
    In this case \eqref{relation} is 
    equal to
    \begin{align}
    \label{relation-1}
    \begin{split}
        &\text{A sequence }
        \phi_1(\beta)
        \text{ has a moment matrix }
        \phi_1\big(\mc M(k;\beta)\big)
        \text{ satisfying the column relation}\\
        &
        c_0Y+c_1XY+c_2Y^2+c_4XY^2+c_5Y^3=\mathbf{0}
        \text{ with }c_i\in\RR,\; 
        (c_4,c_5)\neq (0,0).
    \end{split}
    \end{align}
    If $c_0=c_1=c_2=0$, then
    \eqref{relation-1} is equal to
    $c_4XY^2+c_5Y^3=\mathbf{0}$. Since by assumption $\beta$ and hence $\phi_1(\beta)$ admit  a rm, supported on 
        $$\cZ(y^2(c_4x+c_5y))=\cZ(y(c_4x+c_5y)),$$
    it follows by \cite{CF96} that 
    $c_4 XY+c_5 Y^2=\mathbf{0}$
    is a nontrivial column relation in 
    $\phi_1\big(\mc M(k;\beta)\big)$. Hence, also $\mc M(k;\beta)$
    satisfies a nontrivial column relation between columns indexed by monomials of degree at most 2,
    which is a contradiction with the assumption of the proposition. Therefore 
    $(c_0,c_1,c_2)\neq (0,0,0).$\bigskip

    \noindent 
    \textbf{Case 1.1:
    $c_0\neq 0$.}
    Dividing the relation in \eqref{relation-1}
    by $c_0$, we get:
    \begin{align}
    \label{relation2}
    \begin{split}
        &\text{A sequence }
        \phi_1(\beta)
        \text{ has a moment matrix }
        \phi_1\big(\mc M(k;\beta)\big)
        \text{ satisfying the column relation}\\
        &
        Y+
        \widetilde c_1XY+
        \widetilde c_2Y^2+
        \widetilde c_4XY^2+
        \widetilde c_5Y^3
        =\mathbf{0}
        \text{ with }\widetilde c_i\in\RR,\; 
        (\widetilde c_4,\widetilde c_5)\neq (0,0).\medskip
    \end{split}
    \end{align}

    \noindent 
    \textbf{Case 1.1.1:
    $\widetilde c_1=0$}. In this case 
    \eqref{relation2}
    is equivalent to:
    \begin{align}
    \label{relation1.1.1}
    \begin{split}
        &\text{A sequence }
        \phi_1(\beta)
        \text{ has a moment matrix }
        \phi_1\big(\mc M(k;\beta)\big)
        \text{ satisfying the column relation}\\
        &        
        Y+
        \widetilde c_2Y^2+
        \widetilde c_4XY^2+
        \widetilde c_5Y^3
        =\mathbf{0}
        \text{ with }\widetilde c_i\in\RR,\; 
        (\widetilde c_4,\widetilde c_5)\neq (0,0).\medskip 
    \end{split}
    \end{align}

    \noindent 
    \textbf{Case 1.1.1.1:
    $\widetilde c_4=0$.
    }
    In this case \eqref{relation1.1.1}
    is equivalent to
    \begin{align}
    \label{relation1.1.1.1}
    \begin{split}
        &\text{A sequence }
        \phi_1(\beta)
        \text{ has a moment matrix }
        \phi_1\big(\mc M(k;\beta)\big)
        \text{ satisfying the column relation}\\
        & 
        Y+
        \widetilde c_2Y^2+
        \widetilde c_5Y^3
        =\mathbf{0}
        \text{ with }\widetilde c_2\in \RR,\widetilde c_5\in\RR\setminus\{0\}.
    \end{split}
    \end{align}
    The quadratic equation $1+\widetilde c_2 y+ \widetilde c_5 y^2=0$
    must have two different real nonzero solutions, otherwise
        $\cZ(y(1+\widetilde c_2x+
        \widetilde c_5y))$
    is a union of two parallel lines.
    Then it follows by \cite{CF96} that
    there is a nontrivial column relation in 
    $\mc M(k;\beta)$
    between columns indexed by monomials of degree at most 2, which is a contradiction with the assumption of the proposition. 
    So we have the parallel lines type relation from the proposition.\medskip

    \noindent 
    \textbf{Case 1.1.1.2:
    $\widetilde c_4\neq 0$.}
    In this case the alt 
    $$
        \phi_2(x,y)=
        (
        -\widetilde c_2
        -
        \widetilde c_4 x
        -
        \widetilde c_5
        y,
        y
        \Big)
    $$
    is invertible and 
    applying it
    to $\phi_1(\beta)$, we obtain:
    \begin{align*}
    \begin{split}
        &\text{A sequence }
        (\phi_2\circ\phi_1)(\beta)
        \text{ has a moment matrix }
        (\phi_2\circ \phi_1)\big(\mc M(k;\beta)\big)
        \text{ satisfying}\\
        &     
        \text{the hyperbolic type 1 relation from the proposition.}\bigskip
    \end{split}
    \end{align*}

    \noindent 
    \textbf{Case 1.1.2:
    $\widetilde c_1\neq 0$.} We apply the alt 
    $$\phi_3(x,y)=(1+\widetilde c_1 x,y)$$
    to $\phi_1(\beta)$ and obtain:
    \begin{align}
    \label{relation3}
    \begin{split}
        &\text{A sequence }
        (\phi_3\circ\phi_1)(\beta)
        \text{ has a moment matrix }
        (\phi_3\circ \phi_1)\big(\mc M(k;\beta)\big)
        \text{ satisfying}\\
        &             
        \text{the column relation }
        XY+
        \widehat c_2Y^2+
        \widehat c_4XY^2+
        \widehat c_5Y^3
        =\mathbf{0}
        \text{ with }\widehat c_i\in\RR,\; 
        (\widehat c_4,\widehat c_5)\neq (0,0). \medskip
    \end{split}
    \end{align}

    \noindent 
    \textbf{Case 1.1.2.1:
    $\widehat c_4\neq 0$.}
    We apply the alt 
    $$\phi_4(x,y)=
    \Big(x-\frac{\widehat c_5}{\widehat c_4}y,y\Big)$$
    to $(\phi_3\circ \phi_1)(\beta)$ and obtain:
    \begin{align}
    \label{relation-1.1.2.1}
    \begin{split}
        &\text{A sequence }
        (\phi_4\circ\phi_3\circ\phi_1)(\beta)
        \text{ has a moment matrix }
        (\phi_4\circ\phi_3\circ \phi_1)\big(\mc M(k;\beta)\big)
        \text{ satisfying}\\
        &             
        \text{the column relation }
        XY+
        \breve c_2 Y^2+
        \widehat c_4XY^2
        =\mathbf{0}
        \text{ with }
        \breve c_2,
        \widehat c_4 \in\RR,\; 
        \widehat c_4\neq 0. \medskip
    \end{split}
    \end{align}

    \noindent 
    \textbf{Case 1.1.2.1.1:
    $\breve c_2= 0$.}
    In this case the relation in \eqref{relation-1.1.2.1} is of the form 
        \begin{equation*}
    \label{relation-1.1.2.1.1}
        XY+
        \widehat c_4XY^2
        =\mathbf{0}
        \quad
        \text{with }
        \widehat c_4 \in\RR\setminus \{0\}.
        \end{equation*}
    Applying the alt
        $$\phi_5(x,y)=(x,\widehat c_4 y)$$
    to 
    $(\phi_4\circ\phi_3\circ\phi_1)(\beta)$
    we obtain:
    \begin{align*}
    \begin{split}
        &\text{A sequence }
        (\phi_5\circ\phi_4\circ\phi_3\circ\phi_1)(\beta)
        \text{ has a moment matrix }
        (\phi_5\circ\phi_4\circ\phi_3\circ \phi_1)\big(\mc M(k;\beta)\big)
        \text{ satisfying}\\
        &             
        \text{the intersecting lines type relation from the proposition.}
        \medskip
    \end{split}
    \end{align*}
    
    \noindent 
    \textbf{Case 1.1.2.1.2:
    $\breve c_2\neq 0$.}
    We apply the alt 
    $$\phi_6(x,y)=(x,\breve{c}_2y)$$
    to      
    $(\phi_4\circ\phi_3\circ\phi_1)(\beta)$
    and obtain:
    \begin{align*}
    \begin{split}
        &\text{A sequence }
        (\phi_6\circ\phi_4\circ\phi_3\circ\phi_1)(\beta)
        \text{ has a moment matrix }
        (\phi_6\circ\phi_4\circ\phi_3\circ \phi_1)\big(\mc M(k;\beta)\big)
        \text{ satisfying}\\
        &             
        \text{the hyperbolic type 2 relation in the proposition.}\medskip
    \end{split}
    \end{align*}
    
    \noindent 
    \textbf{Case 1.1.2.2:
    $\widehat c_4= 0$.}
    In this case \eqref{relation3} is equivalent to:
    \begin{align}
    \label{relation-1.1.2.2}
    \begin{split}
        &\text{A sequence }
        (\phi_3\circ\phi_1)(\beta)
        \text{ has a moment matrix }
        (\phi_3\circ \phi_1)\big(\mc M(k;\beta)\big)
        \text{ satisfying}\\
        &             
        \text{the column relation }
        XY+
        \widehat c_2Y^2+
        \widehat c_5Y^3
        =\mathbf{0}
        \text{ with }\widehat c_2,\widehat c_5\in\RR,\; 
        \widehat c_5\neq 0. \medskip
    \end{split}
    \end{align}

    \noindent 
    \textbf{Case 1.1.2.2.1:
    $\widetilde c_2=0$.}
   Applying the alt 
        $$
        \phi_7(x,y)=(x,-\widehat c_5 y),
        $$
    to $(\phi_3\circ\phi_1)(\beta)$
    we obtain: 
    \begin{align*}
    \begin{split}
        &\text{A sequence }
        (\phi_7\circ\phi_3\circ\phi_1)(\beta)
        \text{ has a moment matrix }
        (\phi_7\circ\phi_3\circ \phi_1)\big(\mc M(k;\beta)\big)
        \text{ satisfying}\\
        &             
        \text{the parabolic type relation in the proposition.}\medskip
    \end{split}
    \end{align*}

    \noindent 
    \textbf{Case 1.1.2.2.2:
    $\widetilde c_2\neq 0$.}
    Applying the alt          
        $$\phi_8(x,y)=(x,\widehat c_2 y)$$
    to $(\phi_3\circ\phi_1)(\beta)$
    and obtain:
    \begin{align}
    \label{relation7}
    \begin{split}
        &\text{A sequence }
        (\phi_8\circ\phi_3\circ\phi_1)(\beta)
        \text{ has a moment matrix }
        (\phi_8\circ\phi_3\circ \phi_1)\big(\mc M(k;\beta)\big)
        \text{ satisfying}\\
        &             
        \text{the column relation }
        XY+
        Y^2+
        \breve c_5Y^3
        =\mathbf{0}
        \text{ with }\breve c_5\in\RR,\; 
        \widebreve c_5\neq 0.
    \end{split}
    \end{align}
Further on, the relation in
    \eqref{relation7} is equivalent to
    \begin{equation}
    \label{relation8}
        (\widebreve c_5)^{-1}
        (XY+
        Y^2)+
        Y^3
        =
        \mathbf{0}
        \quad
        \text{with }\widebreve c_5\in\RR,\; 
        \widebreve c_5\neq 0. 
    \end{equation}
    Finally, applying the alt
    $$
    \phi_9(x,y)=
    \big(
        (-\widebreve c_5)^{-1}(x+y),y
    \big)
    $$
    to $(\phi_8\circ\phi_3\circ\phi_1)(\beta)$,
    we obtain: 
    \begin{align*}
    \begin{split}
        &\text{A sequence }
        (\phi_9\circ\phi_8\circ\phi_3\circ\phi_1)(\beta)
        \text{ has a moment matrix }
        (\phi_9\circ\phi_8\circ\phi_3\circ \phi_1)\big(\mc M(k;\beta)\big)
        \\
        &             
        \text{satisfying the parabolic type relation in the proposition.}\bigskip
    \end{split}
    \end{align*}

    \noindent 
    \textbf{Case 1.2:
    $c_0= 0$.}
    In this case \eqref{relation-1} is equivalent to:
    \begin{align}
    \label{relation-1.2}
    \begin{split}
        &\text{A sequence }
        \phi_1(\beta)
        \text{ has a moment matrix }
        \phi_1\big(\mc M(k;\beta)\big)
        \text{ satisfying the column relation}\\
        &
        c_1XY+c_2Y^2+c_4XY^2+c_5Y^3=\mathbf{0}
        \text{ with }c_i\in\RR,\; 
        (c_4,c_5)\neq (0,0).
    \end{split}
    \end{align}
    Assume that $c_1=0$.
    Since by assumption $\beta$ and hence $\phi_1(\beta)$ admits a rm, supported on 
        $$\cZ(y^2(c_2+c_4x+c_5y))
        =\cZ(y(c_2+c_4x+c_5y)),$$
    it follows by \cite{CF96} that 
    $c_2Y+c_4 XY+c_5 Y^2=\mathbf{0}$
    is a nontrivial column relation in 
    $\phi_1\big(\mc M(k;\beta)\big)$. Hence, also $\mc M(k;\beta)$
    satisfies a nontrivial column relation between columns indexed by monomials of degree at most 2,
    which is a contradiction with the assumption of the proposition. Hence, $c_1\neq 0.$
    Applying the alt $(x,y)\mapsto (c_1x,y)$
    to $\phi_1(\beta)$, we obtain a sequence with the moment matrix satisfying the column relation of the form
    \eqref{relation3} and we can proceed as in the Case 1.1.2 above.
    \bigskip

    \noindent
    \textbf{Case 2: $c_3\neq 0$.}
    Applying the alt 
        $$\phi_{10}(x,y)=\big(\sqrt{|c_3|} x,y\big)$$
    to $\phi_1(\beta)$, we obtain:
    \begin{align}
    \label{relation4-v2}
    \begin{split}
        &\text{A sequence }
        (\phi_{10}\circ\phi_1)(\beta)
        \text{ has a moment matrix }
        (\phi_{10}\circ\phi_1)\big(\mc M(k;\beta)\big)
        \text{ satisfying}\\
        &
        \text{the column relation }
        c_0
        Y
        +
        \widetilde c_1 XY
        +
        c_2 Y^2
        +
        \frac{|c_3|}{c_3} X^2Y+
        \widetilde c_4 XY^2+
        c_5Y^3
        =\mathbf{0}
        \text{ with }c_i,\widetilde c_i\in\RR. \medskip
    \end{split}
    \end{align}

    \noindent
    \textbf{Case 2.1: $\widetilde c_1= 0$.}
    In this case
    \eqref{relation4-v2}
    is equivalent to:
    \begin{align}
    \label{relation5-v2}
    \begin{split}
        &\text{A sequence }
        (\phi_{10}\circ\phi_1)(\beta)
        \text{ has a moment matrix }
        (\phi_{10}\circ\phi_1)\big(\mc M(k;\beta)\big)
        \text{ satisfying}\\
        &
        \text{the column relation }
        c_0
        Y
        +
        c_2 Y^2
        +
        \frac{|c_3|}{c_3} X^2Y+
        \widetilde c_4 XY^2+
        c_5Y^3
        =\mathbf{0}
        \text{ with }c_i,\widetilde c_i\in\RR. \medskip
    \end{split}
    \end{align}

    \noindent
    \textbf{Case 2.1.1: $c_0= 0$.}
    Dividing  the relation in \eqref{relation5-v2} with $\frac{|c_3|}{c_3}$, \eqref{relation5-v2} is equivalent to:
    \begin{align}
    \label{relation5-v3-2}
    \begin{split}
        &\text{A sequence }
        (\phi_{10}\circ\phi_1)(\beta)
        \text{ has a moment matrix }
        (\phi_{10}\circ\phi_1)\big(\mc M(k;\beta)\big)
        \text{ satisfying}\\
        &
        \text{the column relation }
        \widetilde c_2 Y^2
        +
        X^2Y+
        \widehat c_4 XY^2+
        \widetilde c_5Y^3
        =\mathbf{0}
        \text{ with }\widetilde c_2,\widehat c_4, \widetilde c_5\in\RR. \medskip
    \end{split}
    \end{align}
     Applying the alt 
        $$\phi_{11}(x,y)=
        \Big(
        x+\frac{\widehat c_4}{2}y, y
        \Big)$$
    to $(\phi_{10}\circ\phi_1)(\beta)$, we obtain:
    \begin{align}
    \label{relation-151023-1509}
    \begin{split}
        &\text{A sequence }
        (\phi_{11}\circ\phi_{10}\circ\phi_1)(\beta)
        \text{ has a moment matrix }
        (\phi_{11}\circ\phi_{10}\circ\phi_1)\big(\mc M(k;\beta)\big)
        \\
        &
        \text{satisfying the column relation }
        \breve c_2 Y^2
        +
        X^2Y
        +
        \breve c_5Y^3
        =\mathbf{0}
        \text{ with }
        \breve c_2,\breve c_5\in\RR. \medskip
    \end{split}
    \end{align}
    
    \noindent
    \textbf{Case 2.1.1.1: 
    $\breve c_5=0$.}
     Since by assumption of the proposition, 
     $(\phi_{11}\circ\phi_{10}\circ\phi_1)(\beta)$ admits a rm, supported on 
        $\cZ(y(\breve c_2y+x^2))$,
    $\breve c_2$ in \eqref{relation-151023-1509} cannot be equal to 0.
    Indeed, $\breve c_2=0$
    would imply that 
    $\cZ(y(\breve c_2y+x^2))=
    \cZ(yx^2)=\cZ(yx)$
    and by \cite{CF96}, 
    $XY=\mathbf{0}$ 
    would be a nontrivial column relation in $(\phi_{11}\circ\phi_{10}\circ\phi_1)\big(\mc M(k;\beta)\big)$. Hence, also $\mc M(k;\beta)$
    would satisfy a nontrivial column relation between columns indexed by monomials of degree at most 2,
    which is a contradiction with the assumption of the proposition. Since $\breve c_2\neq 0$, after applying the
    alt 
        $$\phi_{12}(x,y)=
        (x,-\breve c_2 y)$$
    to $(\phi_{11}\circ\phi_{10}\circ\phi_1)(\beta)$, we obtain:
    \begin{align*}
    \begin{split}
        &\text{A sequence }
        (\phi_{12}\circ\phi_{11}\circ\phi_{10}\circ\phi_1)(\beta)
        \text{ has a moment matrix }
        (\phi_{12}\circ\phi_{11}\circ\phi_{10}\circ \phi_1)\big(\mc M(k;\beta)\big)
        \\
        &             
        \text{satisfying the parabolic type relation in the proposition.}\medskip
    \end{split}
    \end{align*}
    
    \noindent
    \textbf{Case 2.1.1.2: 
    $\breve c_5>0$.}
    Applying the
    alt 
        $$\phi_{13}(x,y)=
        (x,\sqrt{\breve c_5} y)$$
    to $(\phi_{11}\circ\phi_{10}\circ\phi_1)(\beta)$ we obtain:
    \begin{align*}
    \begin{split}
        &\text{A sequence }
        (\phi_{13}\circ\phi_{11}\circ\phi_{10}\circ\phi_1)(\beta)
        \text{ has a moment matrix }
        (\phi_{13}\circ\phi_{11}\circ\phi_{10}\circ \phi_1)\big(\mc M(k;\beta)\big)
        \\
        &             
        \text{satisfying the circular type relation in the proposition.}\medskip
    \end{split}
    \end{align*}
    
    \noindent
    \textbf{Case 2.1.1.3: 
    $\breve c_5<0$.}
        Applying the
    alt 
        $$\phi_{14}(x,y)=
        (x,\sqrt{-\breve c_5} y)$$
    to $(\phi_{11}\circ\phi_{10}\circ\phi_1)(\beta)$, we obtain:
    \begin{align*}
    \begin{split}
        &\text{A sequence }
        (\phi_{14}\circ\phi_{11}\circ\phi_{10}\circ\phi_1)(\beta)
        \text{ has a moment matrix }
        (\phi_{14}\circ\phi_{11}\circ\phi_{10}\circ \phi_1)\big(\mc M(k;\beta)\big)
        \\
        &             
        \text{satisfying the hyperbolic type 3 relation in the proposition.}\medskip
    \end{split}
    \end{align*} 
    
    \noindent
    \textbf{Case 2.1.2: $c_0\neq 0$.}
    Dividing the relation in
    \eqref{relation5-v2} with $c_0$, 
    \eqref{relation5-v2} is equivalent to:
    \begin{align}
    \label{relation5-v3}
    \begin{split}
        &\text{A sequence }
        (\phi_{10}\circ\phi_1)(\beta)
        \text{ has a moment matrix }
        (\phi_{10}\circ\phi_1)\big(\mc M(k;\beta)\big)
        \text{ satisfying}\\
        &
        \text{the column relation }
        Y
        +
        \widetilde c_2 Y^2
        +
        \widetilde c_3 X^2Y+
        \widehat c_4 XY^2+
        \widetilde c_5 Y^3
        =\mathbf{0}
        \text{ with }\widetilde c_i, \widehat c_4\in\RR,
        \widetilde c_3\neq 0.
    \end{split}
    \end{align} 
    Applying the
    alt 
        $$\phi_{15}(x,y)=
        \Big(
            x+\frac{\widehat c_4}{2\widetilde c_3},
            y
        \Big)
        $$
    to $(\phi_{10}\circ\phi_1)(\beta),$ we obtain:
    \begin{align*}
    \begin{split}
        &\text{A sequence }
        (\phi_{15}\circ\phi_{10}\circ\phi_1)(\beta)
        \text{ has a moment matrix }
        (\phi_{15}\circ\phi_{10}\circ \phi_1)\big(\mc M(k;\beta)\big)
        \\
        &             
        \text{satisfying the mixed type relation in the proposition.}\medskip
    \end{split}
    \end{align*}

    \noindent
    \textbf{Case 2.2: $\widetilde c_1\neq 0$.}
    Dividing the relation in \eqref{relation4-v2} with $\frac{|c_3|}{c_3}$, \eqref{relation4-v2} is equivalent to:
    \begin{align}
    \label{relation5-151023}
    \begin{split}
        &\text{A sequence }
        (\phi_{10}\circ\phi_1)(\beta)
        \text{ has a moment matrix }
        (\phi_{10}\circ\phi_1)\big(\mc M(k;\beta)\big)
        \text{ satisfying the}\\
        &
        \text{column relation }
        \widehat c_0
        Y
        +
        \widehat c_1 XY
        +
        \widehat c_2 Y^2
        +
        X^2Y+
        \widehat c_4 XY^2+
        \widehat c_5Y^3
        =\mathbf{0}
        \text{ with }\widehat c_i\in\RR,
        \widehat c_1\neq 0.
    \end{split}
    \end{align}
    Now we apply the alt 
    $$\phi_{16}(x,y)=\Big(x+\frac{\widehat c_1}{2},y\Big)$$
    to $(\phi_{10}\circ\phi_1)(\beta)$ and obtain:
    \begin{align}
    \label{relation6}
    \begin{split}
        &\text{A sequence }
        (\phi_{16}\circ\phi_{10}\circ\phi_1)(\beta)
        \text{ has a moment matrix }
        (\phi_{16}\circ\phi_{10}\circ\phi_1)\big(\mc M(k;\beta)\big)\\
        &
        \text{satisfying the column relation }
        \breve c_0Y
        +
        \breve c_2 Y^2
        +
        X^2Y+
        \breve c_4 XY^2+
        \breve c_5Y^3
        =\mathbf{0}
        \text{ with }\breve c_i\in\RR.\medskip
    \end{split}
    \end{align}

    \noindent
    \textbf{Case 2.2.1: $\breve c_0= 0$.}
    In this case the relation in \eqref{relation6}
    becomes equal to the relation in \eqref{relation5-v3-2} from the Case 2.1.1, so we can proceed as above.\\
     
    \noindent
    \textbf{Case 2.2.2: $\breve c_0\neq 0$.}
    Dividing the relation in \eqref{relation6} with $\breve c_0$, it 
    becomes equal to the relation in \eqref{relation5-v3} from the Case 2.1.2, so we can proceed as above.
\end{proof}


\section{Solving the TMP on canonical reducible cubic curves}
\label{Section-common-approach}

Let 
$\beta=\{\beta_i\}_{i\in \ZZ_+^2,|i|\leq 2k}$ 
be a sequence of degree $2k$, $k\in \NN$,
and
\begin{equation}
\label{degree-lex}
    \cC=\{\mathit{1},X,Y,X^2,XY,Y^2,\ldots,X^k,X^{k-1}Y,\ldots,Y^k\}
\end{equation}
the set of rows and columns of the moment matrix $\cM(k;\beta)$
in the degree-lexicographic order.
Let 
\begin{equation}
\label{form-of-p}
    p(x,y)=y\cdot c(x,y)\in \RR[x,y]_{\leq 3}
\end{equation}
be a polynomial of degree 3 in one of the canonical forms from Proposition \ref{cases}. Hence, $c(x,y)$ a polynomial of degree 2.
$\beta$
will have a $\cZ(p)$--rm if and only if it can be decomposed as
\begin{equation}
\label{decomposition-v2}
	\beta=\beta^{(\ell)}+\beta^{(c)},
\end{equation} 
where
\begin{align*}
    \beta^{(\ell)}
    &:=
    \{\beta_i^{(\ell)}\}_{i\in \ZZ_+^2,|i|\leq 2k}
    \quad
    \text{has a representing measure on }y=0,\\
    \beta^{(c)}    
    &:=
    \{\beta_i^{(c)}\}_{i\in \ZZ_+^2,|i|\leq 2k}
    \quad
    \text{has a representing measure on the conic }c(x,y)=0,
\end{align*}
and the sum in \eqref{decomposition-v2} is a component-wise sum.
On the level of moment matrices, \eqref{decomposition-v2}
is equivalent to 
\begin{equation}
\label{decomposition-v3}
	\cM(k;\beta)=\cM(k;\beta^{(\ell)})+\cM(k;\beta^{(c)}).
\end{equation}
Note that if $\beta$ has a $\cZ(p)$--rm,
then the matrix $\mc M(k;\beta)$ satisfies the relation 
$p(X,Y)=\mathbf{0}$ and it must be rg, i.e.,
\begin{equation}
    \label{071123-0657}
        X^iY^jp(X,Y)=\mathbf{0}
        \quad
        \text{for }i,j=0,\ldots,k-3\text{ such that }i+j\leq k-3.
\end{equation}

We write $\vec{X}^{(0,k)}:=(\mathit{1},X,\ldots,X^k)$.
Let $\cT\subseteq \cC$ be a subset, such that the columns from $\cT$
span the column space $\cC(\cM(k;\beta))$
and
\begin{align}
    \label{order-general}
    \begin{split}
    &P
    \text{ is a permutation matrix such that moment matrix }
    \widetilde{\mc M}(k;\beta):=P\mc M(k;\beta)P^T\\
    &\text{has rows and columns indexed in the order }
        \vec{X}^{(0,k)},
        \cT\setminus\vec{X}^{(0,k)},
        \cC\setminus (\vec{X}^{(0,k)}\cup \cT).
    \end{split}
    \end{align} 
In this new order of rows and columns, \eqref{decomposition-v3}
becomes equivalent to
\begin{equation}
\label{decomposition-v4}
	\widetilde\cM(k;\beta)=
 \widetilde\cM(k;\beta^{(\ell)})+
 \widetilde\cM(k;\beta^{(c)}).
\end{equation}
We write
    \begin{align}
    \label{071123-1939}
		\widetilde\cM(k;\beta)
		&=
		\kbordermatrix{
		& \vec{X}^{(0,k)}
            & \cT\setminus\vec{X}^{(0,k)}
            & \cC\setminus(\vec{X}^{(0,k)}\cup\cT) \\
			(\vec{X}^{(0,k)})^T & A_{11} & A_{12} & A_{13}\\[0.3em]
		  (\cT\setminus\vec{X}^{(0,k)})^T & (A_{12})^T & A_{22} & A_{23}\\[0.3em]
            (\cC\setminus(\vec{X}^{(0,k)}\cup\cT))^T & (A_{13})^T & (A_{23})^T & A_{33}\\
		}.
    \end{align}
By the form of the atoms, we know that 
$\widetilde\cM(k;\beta^{(\ell)})$ 
    and 
$\widetilde\cM(k;\beta^{(c)})$ 
will be of the forms
    \begin{align}
    \label{071123-0642}
    \begin{split}
		\widetilde\cM(k;\beta^{(c)})
		&=
		\kbordermatrix{
		& \vec{X}^{(0,k)}
            & \cT\setminus\vec{X}^{(0,k)}
            & \cC\setminus(\vec{X}^{(0,k)}\cup\cT) \\
			(\vec{X}^{(0,k)})^T & A & A_{12} & A_{13}\\[0.3em]
		  (\cT\setminus\vec{X}^{(0,k)})^T & (A_{12})^T & A_{22} & A_{23}\\[0.3em]
            (\cC\setminus(\vec{X}^{(0,k)}\cup\cT))^T & (A_{13})^T & (A_{23})^T & A_{33}\\
		},\\[0.5em]
		\widetilde\cM(k;\beta^{(\ell)})
		&=
		\kbordermatrix{
		& \vec{X}^{(0,k)}
            & \cT\setminus\vec{X}^{(0,k)}
            & \cC\setminus(\vec{X}^{(0,k)}\cup\cT) \\
			(\vec{X}^{(0,k)})^T & A_{11}-A & \mathbf{0} & \mathbf{0}\\[0.3em]
		  (\cT\setminus\vec{X}^{(0,k)})^T & \mathbf{0} & \mathbf{0} & \mathbf{0}\\[0.3em]
            (\cC\setminus(\vec{X}^{(0,k)}\cup\cT))^T & \mathbf{0} & \mathbf{0} & \mathbf{0}\\
		}
    \end{split}
    \end{align}
for some Hankel matrix $A\in S_{k+1}.$
    Define matrix functions 
    $\cF:S_{k+1}\to S_{\frac{(k+1)(k+2)}{2}}$ 
    and
    $\cH:S_{k+1}\to S_{k+1}$
    by
    \begin{align}
    \label{071123-1940}
    \cF(\mathbf{A})
		&=
        \begin{pmatrix} 
        \mathbf{A} & A_{12} & A_{13}\\[0.3em]
		   (A_{12})^T & A_{22} & A_{23}\\[0.3em]
             (A_{13})^T & (A_{23})^T & A_{33}
	\end{pmatrix}\quad\text{and}\quad
    \cH(\mathbf{A})=
            A_{11}-\mathbf{A}
        .
    \end{align}
    Using \eqref{071123-0642}, \eqref{decomposition-v4}
    becomes equivalent to
    \begin{equation}
        \label{decomposition-v5}
        \widetilde{\mc M}(k;\beta)
        =
        \cF(A)+\cH(A)\oplus \mathbf{0}_{\frac{k(k+1)}{2}}
    \end{equation}
    for some Hankel matrix $A\in S_{k+1}$.

 	\begin{lemma}	
		\label{071123-0646}
            Assume the notation above.
            The sequence 
		$\beta=\{\beta_i\}_{i\in \ZZ_+^2,|i|\leq 2k}$, 
            where $k\geq 3$, 
            has a $\cZ(p)$--representing measure if and only if there exist a Hankel matrix 
		$A\in S_{k+1}$,
		such that:
		\begin{enumerate}
			\item
                \label{301023-1756-pt1}
                The sequence 
				with the moment matrix $\cF(A)$ has a $\cZ(c)$--representing measure.
			\item
                \label{301023-1756-pt2}
   The sequence with the moment matrix $\cH(A)$ has a $\RR$--representing measure.
		\end{enumerate}
	\end{lemma}
 
\begin{proof}
First we prove the implication 
$(\Rightarrow)$.
If $\beta$ has a $\cZ(p)$--rm $\mu$, then $\mu$ is supported on the union of the line
	$y=0$ and the conic $c(x,y)=0$.
    Since the moment matrix, generated by the measure supported on $y=0$, can be nonzero only
	when restricted to the columns and rows indexed by ${\vec{X}}^{(0,k)}$, it follows that the
	moment matrix generated by the restriction 
        $\mu|_{\{c=0\}}$ (resp.\ $\mu|_{\{y=0\}}$) of the measure $\mu$
        to the conic $c(x,y)=0$ (resp.\ line $y=0$),
	is of the form 	$\cF(A)$ (resp.\ 
	$\cH(A)\oplus \mathbf{0}_{\frac{k(k+1)}{2}}$) for some Hankel matrix 
	$A\in S_{k+1}$.

 It remains to establish the implication 
 $(\Leftarrow)$.
	Let $\mc M^{(c)}(k)$ (resp.\ $\mc M^{(\ell)}(k)$) be the moment matrix generated by the measure $\mu_1$ (resp.\ $\mu_2$) supported on $\cZ(c)$ (resp.\ $y=0$)
	such that 
 \begin{align}
\label{071123-0717}
 P\mc M^{(c)}(k)P^T
 &=\cF(A),\quad
 P\mc M^{(\ell)}(k)P^T
 =\cH(A)\oplus \mathbf{0}_{\frac{k(k+1)}{2}},
 \end{align}
 respectively,
	where $P$ is as in 
 \eqref{order-general}.
     The equalities \eqref{071123-0717} imply 
     that $\cM(k;\beta)=\cM^{(c)}(k)+\cM^{(\ell)}(k;\beta)$.
     Since the measure $\mu_1+\mu_2$
     is supported on the curve $\cZ(c)\cup \{y=0\}=\cZ(p)$,
     the implication $(\Leftarrow)$ holds.
\end{proof}

\begin{lemma}
    \label{071123-1942}
    Assume the notation above
    and let the sequence 
	$\beta=\{\beta_i\}_{i\in \ZZ_+^2,|i|\leq 2k}$, where $k\geq 3$, 
    admit a $\cZ(p)$--representing measure.
    Let $A:=A_{\big(\beta_{0,0}^{(c)},\beta_{1,0}^{(c)},\ldots,\beta_{2k,0}^{(c)}\big)}\in S_{k+1}$ be a Hankel matrix such that $\cF(A)$ admits a $\cZ(c)$--representing measure and $\cH(A)$ admits a $\RR$--representing measure.
    Let $c(x,y)$ be of the form
        \begin{align}
        \label{081123-1004}
        \begin{split}
        &c(x,y)
        =a_{00}+a_{10}x+a_{20}x^2+a_{01}y+a_{02}y^2+a_{11}xy\quad
        \text{with }a_{ij}\in \RR\\
        &\text{and exactly one of the coefficients }a_{00},a_{10},a_{20}\text{ is nonzero}.
        \end{split}
        \end{align}
    If:
    \begin{enumerate}
        \item
        \label{071123-1942-pt1}
        $a_{00}\neq 0$, then 
            $$
            \beta_{i,0}^{(c)}
            =
            -\frac{1}{a_{00}}
            (a_{01}\beta_{i,1}+a_{02}\beta_{i,2}+a_{11}\beta_{i+1,1})
            \quad
            \text{for }i=0,\ldots,2k-2.
            $$
        \item 
        \label{071123-1942-pt2}
        $a_{10}\neq 0$, then 
            $$
            \beta_{i,0}^{(c)}
            =
            -\frac{1}{a_{10}}
            (a_{01}\beta_{i,1}+a_{02}\beta_{i,2}+a_{11}\beta_{i+1,1})
            \quad
            \text{for }i=1,\ldots,2k-1.
            $$
        \item 
        \label{071123-1942-pt3}
        $a_{20}\neq 0$, then 
            $$
            \beta_{i,0}^{(c)}
            =
            -\frac{1}{a_{20}}
            (a_{01}\beta_{i,1}+a_{02}\beta_{i,2}+a_{11}\beta_{i+1,1})
            \quad
            \text{for }i=2,\ldots,2k.
            $$
    \end{enumerate}
\end{lemma}

\begin{proof}
    By Lemma \ref{071123-0646}, $\cF(A)$ has a $\cZ(c)$--rm for
    some Hankel matrix $A\in S_{k+1}$. Hence, $\cF(A)$ satisfies the rg
    relations $X^iY^jc(X,Y)=\mathbf{0}$ for $i,j\in \ZZ_+$, $i+j\leq k-2$.
    Let us assume that $a_{00}\neq 0$ and $a_{10}=a_{20}=0$. In particular,
    $\cF(A)$ satisfies the relations
    \begin{align}
    \begin{split}
    \label{071123-1936}
        a_{00}\mathit{1}+a_{01}Y+a_{02}Y^2+a_{11}XY&=\mathbf{0},\\
        a_{00}X^{k-2}+a_{01}X^{k-2}Y+a_{02}X^{k-2}Y^2+a_{11}X^{k-1}Y&=\mathbf{0}.     
    \end{split}
    \end{align}
    Observing the rows $\mathit{1},X,\ldots,X^k$ of $\cF(A)$, the relations \eqref{071123-1936}
    imply that
    \begin{equation}
        \label{071123-1938}
            \beta_{i,0}^{(c)}
            =
            -\frac{1}{a_{00}}
            \big(a_{01}\beta_{i,1}^{(c)}+a_{02}\beta_{i,2}^{(c)}+a_{11}\beta_{i+1,1}^{(c)}\big)
            \quad
            \text{for }i=0,\ldots,2k-2.
    \end{equation}
    Using the forms of $\widetilde\cM(k;\beta)$ and $\cF(A)$ (see \eqref{071123-1939}
    and \eqref{071123-1940}), it follows that 
        $\beta_{i,1}^{(c)}=\beta_{i,1}$
    and 
        $\beta_{j,2}^{(c)}=\beta_{j,2}$
    for each $i,j$. Using this in \eqref{071123-1938} proves the statement 
    \eqref{071123-1942-pt1} of the lemma.
    The proofs of the statements \eqref{071123-1942-pt2} and \eqref{071123-1942-pt3} are analogous.
\end{proof}

Lemma \ref{071123-1942} states that for all canonical relations from Proposition \ref{cases} except for the mixed type relation, all but two entries of the Hankel matrix $A$ from Lemma \ref{071123-0646} are uniquely determined by $\beta$.
The following lemma gives the smallest candidate for $A$ in Lemma \ref{071123-0646} with respect to the usual Loewner order of matrices.

\begin{lemma}
	\label{071123-2008}
    Assume the notation above
    and let
	$\beta=\{\beta_i\}_{i\in \ZZ_+^2,|i|\leq 2k}$, where $k\geq 3$, be a sequence of degree $2k$.
    Assume that $\widetilde{\mc{M}}(k;\beta)$ is positive semidefinite and 
    satisfies the column relations \eqref{071123-0657}.
		Then:
		\begin{enumerate}	
                \item 
                    \label{071123-2008-pt0}
                    $\cF(A)\succeq 0$ for some $A\in S_{k+1}$
                    if and only if
                    $A\succeq A_{12}(A_{22})^{\dagger} (A_{12})^T$.
                \medskip
			\item
				\label{071123-2008-pt1}
					$\cF\big(A_{12}(A_{22})^{\dagger} (A_{12})^T\big)\succeq 0$ 
					and 	
					$\cH\big(A_{12}(A_{22})^{\dagger} (A_{12})^T\big) \succeq 0$.
                \medskip
			\item
				\label{071123-2008-pt2} 
					$\cF\big(A_{12}(A_{22})^{\dagger} (A_{12})^T\big)$ satisfies the column relations
                    $X^iY^jc(X,Y)=0$
                    for $i,j\in \ZZ_+$ such that $i+j\leq k-2$.
                \medskip
			\item 	
				\label{071123-2008-pt3}
                    We have that
                        \begin{align*}
                        \Rank \widetilde{\mc M}(k;\beta)
                        &=
                        \Rank A_{22}+
                        \Rank \big(A_{11}-A_{12}(A_{22})^{\dagger} (A_{12})^T\big)\\
                        &=\Rank \cF\big(A_{12}(A_{22})^{\dagger} (A_{12})^T\big)+
                        \Rank \cH\big(A_{12}(A_{22})^{\dagger} (A_{12})^T\big).
                        \end{align*}
		\end{enumerate}
\end{lemma}

\begin{proof}
By the equivalence between \eqref{pt1-281021-2128} and \eqref{pt2-281021-2128} of Theorem \ref{block-psd}
used for $(M,A)=(\widetilde{\mc M}(k;\beta),A_{11})$ 
        and 
        $(M,A)=(\big(\widetilde{\mc M}(k;\beta)\big)_{\vec{X}^{(0,k)}\cup\cT},A_{11})$,
it follows in particular that
\begin{align}
     \label{081123-0743}
     \begin{split}
        \cC\big(
        \begin{pmatrix}
            (A_{12})^T \\
            (A_{13})^T
        \end{pmatrix}
        \big)
        &\subseteq
        \cC\big(
        \begin{pmatrix}
            A_{22} & A_{23} \\[0.3em]
            (A_{23})^T & A_{33}
        \end{pmatrix}
        \big),\\[0.3em]
        \cC(A_{12}^T)&\subseteq \cC(A_{22}).
    \end{split}
\end{align}
and
\begin{equation}
     \label{081123-0731}
        \cH(A_{\min})\succeq 0,
\end{equation}
where
$$
A_{\min}:=
 \begin{pmatrix}
 A_{12} & A_{13}    
 \end{pmatrix}
 \begin{pmatrix}
     A_{22} & A_{23} \\[0.3em]
     (A_{23})^T & A_{33}
 \end{pmatrix}^{\dagger}
 \begin{pmatrix}
     (A_{12})^T \\ (A_{13})^T
 \end{pmatrix}.
$$
Using the equivalence between \eqref{pt1-281021-2128} and \eqref{pt2-281021-2128} of Theorem \ref{block-psd}
again for the pairs $(M,A)=(\cF(A),A)$ 
        and 
        $(M,A)=(\big(\cF(A)\big)_{\vec{X}^{(0,k)}\cup\cT},A)$, it follows that
 \begin{align}
 \label{081123-0729} 
 \begin{split}
 \cF(A)\succeq 0 \quad &\Leftrightarrow \quad 
 A\succeq A_{\min},\\[0.5em]
\big(\cF(A)\big)_{\vec{X}^{(0,k)}\cup\cT}\succeq 0
  \quad &\Leftrightarrow \quad 
  A\succeq 
 A_{12}
     (A_{22})^{\dagger}
     (A_{12})^T 
 =:\widetilde A_{\min}.
 \end{split}
 \end{align}
 Since $\cF(A)\succeq 0$ implies, in particular, that $\big(\cF(A)\big)_{\vec{X}^{(0,k)}\cup\cT}\succeq 0$, 
 \eqref{081123-0729} implies that 
 \begin{equation}
 \label{081123-0736}
    A_{\min}\succeq \widetilde A_{\min}.  \smallskip 
 \end{equation}

 \noindent\textbf{Claim.} $A_{\min}= \widetilde A_{\min}$.\\

 \noindent\textit{Proof of Claim.} 
 By \eqref{081123-0729} and \eqref{081123-0736}, it suffices to prove that 
 $\cF(\widetilde A_{\min})\succeq 0$.
 By definition of $\cT$ and 
 the relations $X^iY^jp(X,Y)=X^iY^{j+1}c(X,Y)=\mathbf{0}$, $i,j\in \ZZ_+,i+j\leq k-3$,
 which hold in $\widetilde{\cM}(k;\beta)$,
 it follows, in particular, that
 \begin{equation}
     \label{081123-0749}
        \cC\big(
        \begin{pmatrix}
            A_{23} \\
            A_{33}
        \end{pmatrix}
        \big)
        \subseteq
        \cC\big(
        \begin{pmatrix}
            A_{22}  \\[0.3em]
            (A_{23})^T
        \end{pmatrix}
        \big)
\end{equation}
 \eqref{081123-0743} and \eqref{081123-0749} together imply that 
  \begin{equation}
     \label{081123-0750}
        \cC\big(
        \begin{pmatrix}
            (A_{12})^T \\
            (A_{13})^T
        \end{pmatrix}
        \big)
        \subseteq
        \cC\big(
        \begin{pmatrix}
            A_{22}  \\[0.3em]
            (A_{23})^T
        \end{pmatrix}
        \big).
\end{equation}
 \eqref{081123-0743} and \eqref{081123-0750}
 can be equivalently expressed as
\begin{align}
\label{081123-0904}
\begin{split}
        \begin{pmatrix}
            A_{22}  \\[0.3em]
            (A_{23})^T
        \end{pmatrix}
        W&=
        \begin{pmatrix}
            A_{23} \\
            A_{33}
        \end{pmatrix}
        \;\text{for some matrix }W,\\[0.3em]
        \begin{pmatrix}
            A_{22}  \\[0.3em]
            (A_{23})^T
        \end{pmatrix}
        X&=
        \begin{pmatrix}
            (A_{12})^T \\
            (A_{13})^T
        \end{pmatrix}
        \;\text{for some matrix }X.
\end{split}
\end{align}
We have that
\begin{align*}
    0
    &\preceq \begin{pmatrix}
        X^T\\
        I\\
        W^T
    \end{pmatrix}
    A_{22}
    \begin{pmatrix}
        X & I & W
    \end{pmatrix}\\[0.3em]
    &=
    \begin{pmatrix}
        X^TA_{22}X & X^T A_{22} & X^TA_{22}W\\[0.3em]
        A_{22}X & A_{22} & A_{22}W\\[0.3em]
        W^TA_{22}X & W^TA_{22} & W^T A_{22} W 
    \end{pmatrix}
    \\[0.3em]
    &
    =
    \begin{pmatrix}
        A_{12}(A_{22})^{\dagger}(A_{12})^T 
            & A_{12} & A_{13}\\[0.3em]
        (A_{12})^T & A_{22} & A_{23}\\[0.3em]
        (A_{13})^T & (A_{23})^T & A_{33} 
    \end{pmatrix}
    =
    \cF(\widetilde A_{\min})
\end{align*}
where $I$ is the identity matrix of the same size as $A_{22}$
and 
we used \eqref{081123-0904} in the second equality.
This proves the Claim.\hfill\hfill$\blacksquare$\\

\noindent Using \eqref{081123-0731},
\eqref{081123-0729} and Claim,
the statements \eqref{071123-2008-pt0}
and \eqref{071123-2008-pt1} follow. 
By Theorem \ref{block-psd}.\eqref{prop-2604-1140-eq2},
used for $(M,A)=(\widetilde{\cM}(k;\beta),A_{11})$, we have that 
\begin{align}
\label{081123-0946}
\begin{split}
    \Rank \widetilde{\mc M}(k;\beta)
    &=
    \Rank 
        \begin{pmatrix}
            A_{22} & A_{23} \\[0.3em]
            (A_{23})^T & A_{33}
        \end{pmatrix}
    +
    \Rank \cH(A_{\min})\\[0.3em]
   &=
    \Rank 
        \cF(A_{\min})
    +
    \Rank \cH(A_{\min}).
\end{split}
\end{align}
By \eqref{081123-0749} and        
        $$
        B:=
        \begin{pmatrix}
            A_{22} & A_{23} \\[0.3em]
            (A_{23})^T & A_{33}
        \end{pmatrix}
        \succeq 0,
        $$
it follows by Theorem \ref{block-psd}, used for $(M,A)=(B,A_{22})$,
that $\Rank B=\Rank A_{22}$. Using this and the Claim,
\eqref{081123-0946} implies the statement \eqref{071123-2008-pt3}.

	Since $\widetilde{\mc M}(k;\beta)$ satisfies 
        the relations \eqref{071123-0657}, it follows that
	the restriction 	
        $\big(\cF(\widetilde A_{\min})\big)_{\cC\setminus \vec{X}^{(0,k)},\cC}
        $
        satisfies the column relations
        $X^iY^jc(X,Y)=\mathbf{0}$
                    for $i,j\in \ZZ_+$ such that $i+j\leq k-2$. 
	By Proposition \ref{extension-principle},
        these relations extend to 
        $\cF(\widetilde A_{\min})$,
        which proves 
        \eqref{071123-2008-pt2}.
\end{proof}

\begin{remark}
\label{general-procedure}
By Lemmas \ref{071123-0646}--\ref{071123-2008}, 
solving the $\cZ(p)$--TMP for the sequence 
    $\beta=\{\beta_i\}_{i\in \ZZ_+^2,|i|\leq 2k}$, 
where $k\geq 3$, 
with $p$ being any but the mixed type relation from Proposition \ref{cases}, 
the natural procedure is the following:
\begin{enumerate}
    \item First compute $A_{\min}:=A_{12}(A_{22})^{\dagger}A_{12}$.
        By Lemma \ref{071123-2008}.\eqref{071123-2008-pt2}, there is one entry
        of $A_{\min}$,
        which might need to be changed to obtain a Hankel structure. 
        Namely, in the notation \eqref{081123-1004},
        if:
        \begin{enumerate}
            \item $a_{00}\neq 0$, then the value of $(A_{\min})_{k,k}$ must be made equal to $(A_{\min})_{k-1,k+1}$.
            \item $a_{10}\neq 0$, then the value of $(A_{\min})_{1,k+1}$   
                must be made equal to $(A_{\min})_{2,k}$.
            \item $a_{20}\neq 0$, then the value of $(A_{\min})_{2,2}$ must be made equal to $(A_{\min})_{3,1}$.
        \end{enumerate}
        Let $\widehat A_{\min}$ be  the matrix obtained from $A_{\min}$ after performing the changes described above.
    \smallskip
    \item Study if $\cF(\widehat A_{\min})$ and $\cH(\widehat A_{\min})$ admit a $\cZ(c)$--rm 
        and a $\RR$--rm, respectively.
        If the answer is yes, $\beta$ admits a $\cZ(p)$--rm. 
        Otherwise by Lemma \ref{071123-1942},
        there are two antidiagonals of the Hankel matrix $\widehat A_{\min}$, 
        which can by varied so that the matrices $\cF(\widehat A_{\min})$ and $\cH(\widehat A_{\min})$
        will admit the corresponding measures. 
        Namely, in the notation \eqref{081123-1004},
        if:
        \begin{enumerate}
            \item $a_{00}\neq 0$, then the last two antidiagonals of $\widehat A_{\min}$ 
            can be changed.
            \item $a_{10}\neq 0$, then the left--upper and the right--lower corner of $\widehat A_{\min}$ 
            can be changed.
            \item $a_{20}\neq 0$, then the first two antidiagonals of $\widehat A_{\min}$ 
            can be changed.
        \end{enumerate}
        To solve the $\cZ(p)$--TMP for $\beta$ one needs to characterize, when it is possible to change these antidiagonals in such a way
        to obtain a matrix $\breve A_{\min}$, such that 
        $\cF(\breve A_{\min})$ and $\cH(\breve A_{\min})$ admit a $\cZ(c)$--rm 
        and a $\RR$--rm, respectively.
\end{enumerate}
\end{remark}

In Sections \ref{circular} and \ref{parabolic} we solve concretely the TMP on reducible cubic curves in the circular and parabolic type form (see the classification from Proposition \ref{cases}).
The parallel lines type form was solved in \cite{Zal22a}, while
the hyperbolic type forms will be solved in the forthcoming work \cite{YZ+}.


\section{Circular type relation: 
            $p(x,y)=y(ay+x^2+y^2)$, 
            $a\notin \RR\setminus\{0\}$.
        }
\label{circular}

In this section we solve the $\cZ(p)$--TMP for 
the sequence $\beta=\{\beta_{i,j}\}_{i,j\in \ZZ_+,i+j\leq 2k}$ 
of degree $2k$, $k\geq 3$,
where $p(x,y)=y(ay+x^2+y^2)$, $a\in \RR\setminus\{0\}$.
Assume the notation from Section \ref{Section-common-approach}.
If $\beta$ admits a $\cZ(p)$--TMP, then $\mc M(k;\beta)$ 
must satisfy the relations
\begin{equation}
    \label{221023-1844}
	aY^{2+j}X^{i}
        +
        Y^{1+j}X^{2+i}
        =
        -Y^{3+j}X^{i}\quad
        \text{for }i,j\in \ZZ_+\text{ such that }i+j\leq k-3.
\end{equation}
	In the presence of all column relations \eqref{221023-1844}, the column space $\cC(\mc M(k;\beta))$ is spanned by the columns in the set
    \begin{equation}
    \label{221023-1848}
    \cT=
        \vec{X}^{(0,k)}
        \cup 
        Y\vec{X}^{(0,k-1)}
        \cup
        Y^2\vec{X}^{(0,k-2)},
    \end{equation}
    where 
    $$
        Y^i\vec{X}^{(j,\ell)}:=(Y^iX^j,Y^iX^{j+1},\ldots,Y^iX^{\ell})
        \quad\text{with }i,j,\ell\in \ZZ_+,\; j\leq \ell,\; i+\ell\leq k.
    $$
    Let $\widetilde \cM(k;\beta)$ 
    be as in 
    \eqref{071123-0642}.
    Let 
    \begin{equation}
        \label{091123-0719}
        A_{\min}:=A_{12}(A_{22})^{\dagger} (A_{12})^T.
    \end{equation}
    As described in Remark \ref{general-procedure},
    $A_{\min}$ might need to be changed to
    $$
    \widehat A_{\min}
    =A_{\min}+\eta E_{2,2}^{(k+1)},
    $$
where
    $$
    \eta:=(A_{\min})_{1,3}-(A_{\min})_{2,2}.
    $$
    Let $\cF(\mathbf{A})$ and $\cH(\mathbf{A})$ 
    be as in 
    \eqref{071123-1940}.
Write
\begin{align}
\label{081123-1936}
\begin{split}
        \cH(\widehat A_{\min})
        &:=
            \kbordermatrix{ 
                & \mathit{1} & X& \vec{X}^{(2,k)}\\
            \mathit{1}  & \beta_{0,0}-(A_{\min})_{1,1} & \beta_{1,0}-(A_{\min})_{1,2} & (h_{12}^{(1)})^T\\[0.2em]
            X
                & \beta_{1,0}-(A_{\min})_{1,2}& \beta_{2,0}-(A_{\min})_{1,3} & (h_{12}^{(2)})^T\\[0.2em]
            (\vec{X}^{(2,k)})^T & 
            h_{12}^{(1)} &
            h_{12}^{(2)} &
            H_{22}},\\[0.5em]
H_1&:=(\cH(\widehat A_{\min}))_{\{1\}\cup\vec{X}^{(2,k)}}=
\kbordermatrix{
    & \mathit{1} & \vec{X}^{(2,k)}\\[0.2em]
    \mathit{1} & \beta_{0,0}-(A_{\min})_{1,1} & (h_{12}^{(1)})^T\\[0.2em]
    (\vec{X}^{(2,k)})^T & h_{12}^{(1)} & H_{22}},\\[0.5em]
H_2&:=(\cH(\widehat A_{\min}))_{\vec{X}^{(1,k)}}=
\kbordermatrix{
    & X & \vec{X}^{(2,k)}\\[0.2em]
    X & \beta_{2,0}-(A_{\min})_{1,3} & (h_{12}^{(2)})^T\\[0.2em]
    (\vec{X}^{(2,k)})^T & h_{12}^{(2)} & H_{22}}.
\end{split}
\end{align}
Define also the matrix function 
    \begin{equation}
        \label{301023-1930}
        \mc G:\RR^2\to S_{k+1},\qquad 
	\mc G(\mathbf{t},\mathbf{u})=
        \widehat A_{\min}
        +\mathbf{t}E_{1,1}^{(k+1)}
        +\mathbf{u}\big(E_{1,2}^{(k+1)}+E_{2,1}^{(k+1)}\big).
    \end{equation}

The solution to the cubic circular type relation TMP is the following. 

    \begin{theorem}
		\label{221023-1854}
	Let $p(x,y)=y(ay+x^2+y^2)$,
        $a\in \RR\setminus\{0\}$,
	and $\beta=(\beta_{i,j})_{i,j\in \ZZ_+,i+j\leq 2k}$, where $k\geq 3$.
	Assume also the notation above.
	Then the following statements are equivalent:
	\begin{enumerate}	
		\item\label{221023-1854-pt1} 
                $\beta$ has a $\cZ(p)$--representing measure.
                \smallskip
		\item\label{221023-1854-pt3}  
  $\widetilde{\mc{M}}(k;\beta)$ is positive semidefinite,
  the relations
	\begin{equation}\label{221023-1857-equation}
	a\beta_{i,2+j}
        +
        \beta_{2+i,1+j}
        =
        -\beta_{i,3+j}
        \quad\text{hold for every }i,j\in \ZZ_+\text{ with }i+j\leq 2k-3
	\end{equation}
	and 
	one of the following statements holds:
        \smallskip
		\begin{enumerate}
			\item
				\label{221023-1857-pt3.1}
                    $\eta=0$
                    and one of the following holds:
                    \smallskip
                    \begin{enumerate}
                        \item 
                        \label{221023-1857-pt3.1.1}
                            $\Rank (\cH(A_{\min}))_{\vec{X}^{(0,k-1)}}=k$.
                        \smallskip
                        \item 
                        \label{221023-1857-pt3.1.2} 
                            $\Rank (H_2)_{\vec{X}^{(1,k-1)}}=\Rank H_2$.
                    \end{enumerate}
                    \smallskip
                \item
			\label{221023-1857-pt3.2}
                    $\eta>0$, $H_2$ is positive semidefinite
                    and defining a real number
                    \begin{align}
                        \label{101123-1548}
                         \begin{split}
                        u_0
                        &=
                        \beta_{1,0}-(A_{\min})_{1,2}
                        -(h_{12}^{(1)})^T (H_{22})^{\dagger} h_{12}^{(2)},
                        \end{split}
                    \end{align}
                    a function
                    \begin{equation}
                        \label{101123-1550}
                        h(\mathbf{t})=
                        \sqrt{
                        (H_1/H_{22}-\mathbf{t})
                        (H_2/H_{22})
                        }
                    \end{equation}
                    and a 
                    set
                    \begin{align}
                    \label{101123-1934}
                    \begin{split}
                    \mc I
                    &=
                    \big\{(t,\sqrt{\eta t})\in \RR_+\times \RR_+\colon
                    \sqrt{\eta t}= 
                    u_{0}+h(t)\},\\
                    &\hspace{0.5cm}
                    \cup
                    \big\{(t,\sqrt{\eta t})\in \RR_+\times \RR_-\colon
                    \sqrt{\eta t}= 
                    u_{0}-h(t)\},\\
                    &\hspace{0.5cm}
                    \cup
                    \big\{(t,-\sqrt{\eta t})\in \RR_+\times \RR_+\colon
                    -\sqrt{\eta t}= 
                    u_{0}+h(t)\},\\
                    &\hspace{0.5cm}
                    \cup
                    \big\{(t,-\sqrt{\eta t})\in \RR_+\times \RR_-\colon
                    -\sqrt{\eta t}= 
                    u_{0}-h(t)\},\\
                    \end{split}
                    \end{align}
                    one of the following holds:
                    \smallskip
                    \begin{enumerate}
                    \item 
                    \label{221023-1857-pt3.2.1}
                    The set $\mc I$ has two elements and $H_2$ is positive definite.
                    \smallskip
                    \item
                    \label{221023-1857-pt3.2.2}
                    $\mc I=\{(\tilde t,\tilde u)\}$ 
                    and 
                        \begin{equation}
                        \label{111123-1803}
                        \Rank 
                        \big(
                            \big(
                            \cH(\mc G(\tilde t,\tilde u))
                            \big)_{\vec{X}^{(0,k-1)}}
                        \big)
                        =
                        \Rank \cH(\mc G(\tilde t,\tilde u)).
                        \end{equation}
                    \end{enumerate}
	\end{enumerate}
 \end{enumerate}
        \smallskip
 	Moreover, if a $\mc Z(p)$--representing measure for $\beta$ exists, then:
        \begin{itemize}
        \item There exists at most $(\Rank \widetilde{\mc{M}}(k;\beta)+1)$--atomic $\cZ(p)$--representing 
            measure.
        \item There exists a $(\Rank \widetilde{\mc{M}}(k;\beta))$--atomic $\cZ(p)$--representing measure
	   if and only if any of the following holds:
        \smallskip
        \begin{itemize}
            \item 
                $\eta=0$.
                \smallskip
            \item 
                $\eta>0$
                and 
                $\cH(A_{\min})$ is positive definite.
        \end{itemize}
        \end{itemize}
In particular, a $p$--pure sequence $\beta$ with a 
 $\cZ(p)$--representing 
            measure
admits a $(\Rank \widetilde{\mc{M}}(k;\beta))$--atomic $\cZ(p)$--representing 
            measure.
\end{theorem}

\begin{remark}
\label{proof-line+circle}
In this remark we explain the idea of the proof of Theorem \ref{221023-1854}
and the meaning of the conditions in the statement of the theorem.

By Lemmas \ref{071123-0646}--\ref{071123-1942}, the existence of a $\mc Z(p)$--rm for $\beta$ is equivalent to the existence of $t,u\in \RR$ such 
that 
    $\cF(\mc G(t,u))$ 
admits a $\cZ(ay+x^2+y^2)$--rm 
and 
    $\cH(\mc G(t,u))$ 
admits a $\RR$--rm. 
    Let 
\begin{align*}
    \mc R_1
    &=\big\{(t,u)\in \RR^2\colon \cF(\mc G(t,u))\succeq 0\big\}
    \quad\text{and}\quad
    \mc R_2
    =\big\{(t,u)\in \RR^2\colon \cH(\mc G(t,u))\succeq 0\big\}.
\end{align*}
We denote by $\partial R_i$ and $\interior{R}_i$ the topological boundary and the interior of the set $R_i$, respectively.
 By the necessary conditions for the existence of a 
    $\cZ(p)$--rm \cite{CF04,Fia95,CF96}, 
    $\widetilde{\mc M}(k;\beta)$ must be psd and the relations
    \eqref{221023-1857-equation} must hold.
Using also Theorem \ref{circle-TMP}, 
Theorem \ref{221023-1854}.\eqref{221023-1854-pt1} is equivalent to
\begin{align}
\label{251023-1603-v2}
\begin{split}
&\widetilde{\cM}(k;\beta)\succeq 0, 
\text{ the relations }
\eqref{221023-1857-equation}\text{ hold 
and }\\
&\exists (t_0,u_0)\in \mc R_1\cap\mc R_2:
\cH(\mc G(t_0,u_0))\text{ admits a }\RR\text{--rm}.
\end{split}
\end{align}
In the proof of Theorem \ref{221023-1854} we show
that
\eqref{251023-1603-v2} is equivalent to
Theorem \ref{221023-1854}.\eqref{221023-1854-pt3}:
\begin{enumerate}
\item
First we establish (see Claims 1 and 2 below)
that the form of:
\begin{itemize}
    \item 
$\mc R_1$ is one of the following:
    \begin{center}
    \begin{tabular}{lr}
    \includegraphics[width=5cm]{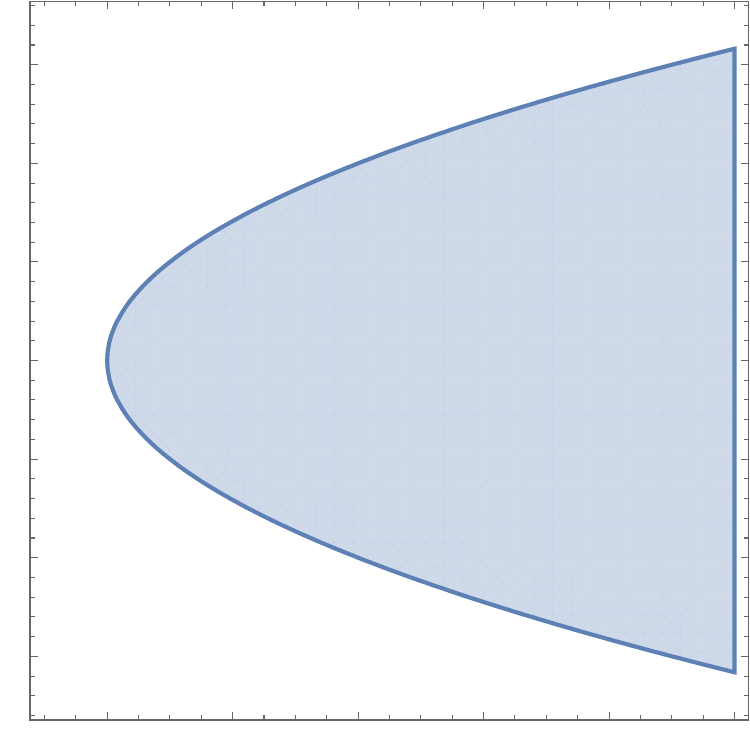}
    &
    \includegraphics[height=4cm]{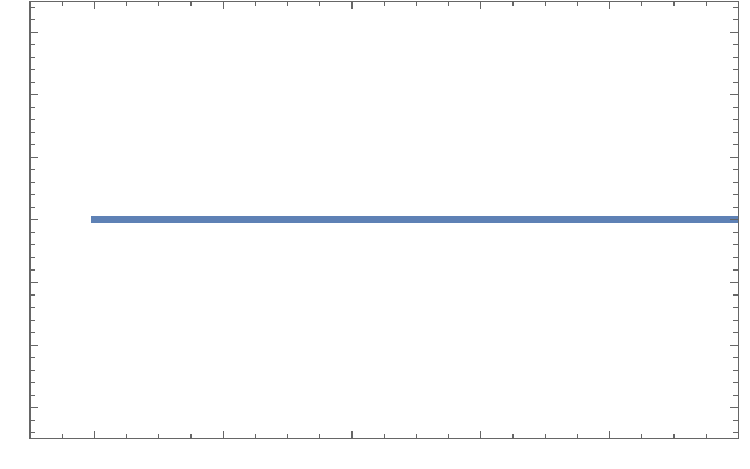}
    \end{tabular}
    \end{center}
where the left case occurs if $\eta>0$ and the right if $\eta=0$.
The case $\eta<0$ cannot occur.
\item 
$\mc R_2$ is one of the following:
    \begin{center}
    \begin{tabular}{lr}
    \includegraphics[width=5cm]{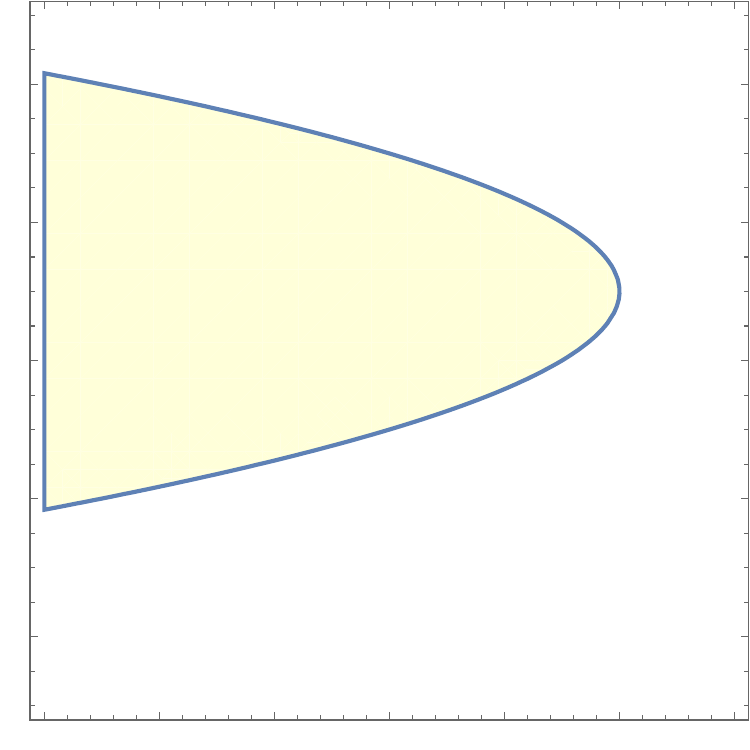}
    &
    \includegraphics[height=4cm]{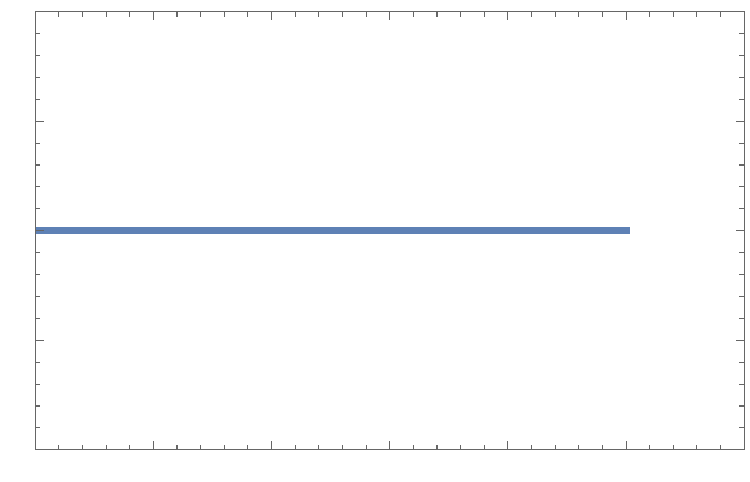}
    \end{tabular}
    \end{center}
where the left case occurs if $H_2/H_{22}>0$ and the right if $H_2/H_{22}=0$. 
\end{itemize}
\item If $\eta=0$, then  we show that
\eqref{251023-1603-v2}
is equivalent to
\begin{align*}
\begin{split}
&\widetilde{\cM}(k;\beta)\succeq 0, 
\text{ the relations }
\eqref{221023-1857-equation}\text{ hold 
and }
\cH(\mc G(0,0))\text{ admits a }\RR\text{--rm}.
\end{split}
\end{align*}
The latter statement is further equivalent to Theorem \ref{221023-1854}.\eqref{221023-1857-pt3.1}. 
\item 
If $\eta>0$,
then by the forms of $\mc R_1$ and $\mc R_2$, 
$\cI=\partial\mc R_1\cap \partial \mc R_2$ is one of the following: 
(i) $\emptyset$, 
(ii) a one-element set, 
(iii) a two-element set.  
In the case: 
\begin{itemize}
\item (i), a $\cZ(p)$--rm for $\beta$ clearly cannot exist.
\smallskip
\item (ii), then denoting $\cI=\{(\tilde t,\tilde u)\}$,
\eqref{251023-1603-v2}
is equivalent to 
\begin{align*}
\begin{split}
&\widetilde{\cM}(k;\beta)\succeq 0, 
\text{ the relations }
\eqref{221023-1857-equation}\text{ hold 
and }
\cH(\mc G(\tilde t,\tilde u))\text{ admits a }\RR\text{--rm}. 
\end{split}
\end{align*}
The latter statement is equivalent to Theorem \ref{221023-1854}.\eqref{221023-1857-pt3.2.2}.
\smallskip
\item (iii), \eqref{251023-1603-v2}
is equivalent to $H_2$ being positive definite,
which is Theorem \ref{221023-1854}.\eqref{221023-1857-pt3.2.1}.
Moreover, in this case for at least one of the points $(t,u)\in \cI$,
a $\cZ(ay+x^2+y^2)$--rm
and a $\RR$--rm exist for $\cF(\mc G(t,u))$ and
$\mc H(\mc G(t,u))$, respectively.
\end{itemize}
\end{enumerate}
\end{remark}

\begin{proof}[Proof of Theorem \ref{221023-1854}]
Let $\mc R_1, \mc R_2$ be as in Remark \ref{proof-line+circle}.
As explained in Remark \ref{proof-line+circle}, 
Theorem \ref{221023-1854}.\eqref{221023-1854-pt1} 
is equivalent to
\eqref{251023-1603-v2},
thus it remains to prove that
\eqref{251023-1603-v2} is equivalent to
Theorem \ref{221023-1854}.\eqref{221023-1854-pt3}.\\

First we establish a few claims needed in the proof.
Claim 1 (resp.\ 2) describes $\mc R_1$ (resp.\ $\mc R_2$) concretely.\\

\noindent 
\textbf{Claim 1.} 
    Assume that $\widetilde{\mc{M}}(k;\beta)\succeq 0$.
    Then
    \begin{equation}
    \label{form-of-R1}    
    \mc R_1
    =
    \left\{
    \begin{array}{rl}
    \big\{
    (t,u)\in \RR^2\colon 
    t\geq 0, 
    u\in \left[-\sqrt{\eta t},\sqrt{\eta t}\right]
    \big\},&
    \text{if }\eta\geq 0,\\[0.3em]
    \emptyset,&
    \text{if }\eta<0.
    \end{array}
    \right.
    \end{equation}
If $\eta\geq 0$, we have
\begin{align}
    \label{rank-R1}
    \Rank \cF(\mc G(t,u))=
    \left\{
    \begin{array}{rl}
            \Rank \cF(A_{\min}),&    
            \text{if }
                t=0, 
                \eta=0,\\[0.3em]
            \Rank \cF(A_{\min})+1,&    
            \text{if }
                (t>0 \text{ or }\eta>0)
                \text{ and }
                u\in \{-\sqrt{\eta t},\sqrt{\eta t}\},\\[0.3em]
            \Rank \cF(A_{\min})+2,&    
            \text{if }
                t>0,\eta>0,
                u\in \left(-\sqrt{\eta t},\sqrt{\eta t}\right),
    \end{array}
    \right.
\end{align}
where $A_{\min}$ is an in \eqref{091123-0719}.\\

\noindent\textit{Proof of Claim 1.}
Note that 
\begin{align}
\label{091123-1155}
\begin{split}
    \mc G(\mathbf{t},\mathbf{u})
    &=
        A_{\min}
        +\eta E_{2,2}^{(k+1)}
        +\mathbf{t}E_{1,1}^{(k+1)}
        +\mathbf{u}\big(E_{1,2}^{(k+1)}+E_{2,1}^{(k+1)}\big)\\
    &=
        A_{\min}+
        \begin{pmatrix}
            \mathbf{t} & \mathbf{u} \\ \mathbf{u} & \eta   
        \end{pmatrix}
        \oplus
        \mathbf{0}_{k-1}.
\end{split}
\end{align}
By Lemma \ref{071123-2008}, we have that
\begin{equation}   
\label{091123-1156}
\cF(\mc G(t,u))\succeq 0 
    \quad\Leftrightarrow\quad   
\mc G(t,u)\succeq A_{\min}
\end{equation}
Using \eqref{091123-1155}, \eqref{091123-1156} and the definition of $\mc R_1$, we have that
\begin{align}
\label{091123-1246}
    (t,u)\in \mc R_1
    \quad&\Leftrightarrow\quad 
    \begin{pmatrix}
        t & u \\ u & \eta
    \end{pmatrix}\succeq 0
    \quad\Leftrightarrow\quad 
    t\geq 0, \eta\geq 0, t\eta\geq u^2,
\end{align}
which proves \eqref{form-of-R1}. 

To prove \eqref{rank-R1} 
first note that by construction of $\cF(A_{\min})$, the columns $1$ and $X$
are 
in the span of the columns indexed by 
    $\cC\setminus \vec{X}^{(0,k)}$.
Hence, there are vectors 
\begin{equation}
\label{091123-1254}
    v_1, v_2 \in \ker \cF(A_{\min})
\end{equation}
of the forms 
$$
v_1=\begin{pmatrix}
    1 & \mathbf{0}_{1,k} & (\tilde v_1)^T
\end{pmatrix}^T\in \RR^{\frac{(k+1)(k+2)}{2}}
\quad\text{and}\quad
v_2=\begin{pmatrix}
    0 & 1 & \mathbf{0}_{1,k-2} & (\tilde v_2)^T
\end{pmatrix}^T\in \RR^{\frac{(k+1)(k+2)}{2}}.
$$
Let $r:=\Rank 
\begin{pmatrix}
    t & u \\ u & \eta
\end{pmatrix}$.
Clearly,
\begin{equation}
\label{091123-1251} 
    \Rank \cF(\mc G(t,u))\leq \Rank \cF(A_{\min})+r.
\end{equation}
We separate three cases according to $r$.
\\

\noindent\textbf{Case 1: $r=0$.} In this case $t=u=\eta=0$ and $\mc G(0,0)=A_{\min}$. In this case \eqref{rank-R1} clearly holds.\\

\noindent\textbf{Case 2: $r=1$.} In this case $t\eta=u^2$. Together with \eqref{091123-1246}, this is equivalent to
$(t>0 \text{ or }\eta>0) \text{ and }
                u\in \{-\sqrt{\eta t},\sqrt{\eta t}\}$.
By \eqref{091123-1251} and $\cF(\mc G(t,u))\succeq \cF(A_{\min})$ to
prove
\eqref{rank-R1}, it suffices to  
find $v\in \ker \cF(A_{\min})$ and $v\notin \ker \cF(\mc G(t,u))$. Note that at least one of $v_1,v_2$ from \eqref{091123-1254}
is such a vector, since
$$
(v_1)^T\cF(\mc G(t,u))v_1=t
    \quad \text{and}\quad
(v_2)^T\cF(\mc G(t,u))v_2=\eta.\smallskip
$$

\noindent\textbf{Case 3: $r=2$.} In this case $t\eta>u^2$. Together with \eqref{091123-1246}, this is equivalent to
$t>0,\eta>0,u\in (-\sqrt{\eta t},\sqrt{\eta t})$.
Note that
\begin{equation}
\label{091123-1305} 
\cF(\mc G(t,u))=
\cF\Big(\mc G\Big(\frac{u^2}{\eta},u\Big)\Big)+
\begin{pmatrix}
    t-\frac{u^2}{\eta}
\end{pmatrix}
\oplus
\mathbf{0}_{\frac{(k+1)(k+2)}{2}-1}
\succeq 
\cF\Big(\mc G\Big(\frac{u^2}{\eta},u\Big)\Big).
\end{equation}
By Case 2, we have $\Rank \cF\Big(\mc G\Big(\frac{u^2}{\eta},u\Big)\Big)=\Rank \cF(A_{\min})+1$.
By \eqref{091123-1251} and \eqref{091123-1305},
to
prove
\eqref{rank-R1}, it suffices to  
find $v\in \ker \cF\Big(\mc G\Big(\frac{u^2}{\eta},u\Big)\Big)$ and $v\notin \ker\cF(\mc G(t,u))$. 
We will check below, that $v_3$, defined by
    $$v_3=
        v_1-\frac{u}{\eta}v_2
        =
        \begin{pmatrix} 1 & -\frac{u}{\eta} & (\tilde v_3)^T\end{pmatrix}^T\in \RR^{\frac{(k+1)(k+2)}{2}},
        $$
is such a vector.
This follows by 
$$
\cF\Big(\mc G\Big(\frac{u^2}{\eta},u\Big)\Big)v_3=\cF(A_{\min})v_3+        
        \left(\begin{pmatrix}
            \frac{u^2}{\eta} & u \\ u & \eta   
        \end{pmatrix}
        \oplus
        \mathbf{0}_{\frac{(k+1)(k+2)-1}{2}}\right)v_3=\mathbf{0}_{\frac{(k+1)(k+2)}{2},1}
$$
and
$$
(v_3)^T\cF(\mc G(t,u))v_3=t-\frac{u^2}{\eta}>0.
$$

\noindent This concludes the proof of Claim 1.
\hfill$\blacksquare$\\

\noindent 
\textbf{Claim 2.} 
    Assume that $\widetilde{\mc{M}}(k;\beta)\succeq 0$.
    Let $u_{0}$, $h(\mathbf{t})$ be as in
    \eqref{101123-1548},\eqref{101123-1550} 
    and
    $$
    t_0
    =
    \beta_{0,0}-(A_{\min})_{1,1}-
                            (h_{12}^{(1)})^T (H_{22})^{\dagger} h_{12}^{(1)}.
    $$
    Then
    \begin{equation}
        \label{form-of-R2}
    \mc R_2
    =
    \left\{
    \begin{array}{rl}
    \big\{
    (t,u)\in \RR^2\colon 
    t\leq t_0, 
    u\in [u_0-h(t),u_0+h(t)]
    \big\},&\text{if }H_2\succeq 0,\\[0.3em]
    \emptyset,&\text{if }H_2\not\succeq 0.
    \end{array}
    \right.
    \end{equation}
If $H_2\succeq 0$, we have that
\begin{align}
    \label{rank-R2}
    \Rank \cH(\mc G(t,u))=
    \left\{
    \begin{array}{rl}
            \Rank H_{2},&    
            \text{for }
                t=t_0, 
                u=u_0,\\[0.2em]
            \Rank H_{22}+1,&    
            \text{for }
                t<t_0, 
                u\in \{u_0-h(t),u_0+h(t)\},\\[0.2em]
            \Rank H_{22}+2,&    
            \text{for }
                t<t_0, 
                u\in (u_0-h(t),u_0+h(t)).
    \end{array}
    \right.
\end{align}
\smallskip

\noindent\textit{Proof of Claim 2.}
Write
\begin{align*}
        H(\mathbf{t})
        &:=
        \big(
        \cH(\mc G(\mathbf{t},\mathbf{u})
        \big)_{1\cup\vec{X}^{(2,k)}}
        =
        \kbordermatrix{
        & \mathit{1} & \vec{X}^{(2,k)}\\
        \mathit{1} & \beta_{0,0}-(A_{\min})_{1,1}-\mathbf{t} & (h_{12}^{(1)})^T\\[0.3em]
        \vec{X}^{(2,k)} & h_{12}^{(1)} & H_{22}
        }
        \end{align*}
Note that 
    $H(0)=(\cH(A_{\min}))_{\{\mathit{1}\}\cup \vec{X}^{(2,k)}}$.
By Lemma \ref{071123-2008}.\eqref{071123-2008-pt1},
$\cH(A_{\min})\succeq 0$ 
and hence, 
$H(0)\succeq 0$.
By Theorem \ref{block-psd}, used for $(M,C)=(H(0),H_{22})$, it follows that $H_2\succeq 0$ and $h_{12}^{(1)}\in \cC(H_{22}).$
Again, by Theorem \ref{block-psd}, used for $(M,C)=(H(t),H_{22})$, it follows that 
$H(t)\succeq 0$ iff $t\leq t_0$.
For a fixed $t$ satisfying $t\leq t_0$, 
Lemma \ref{psd-completion}, used for $A(\mathbf{x})=\cH(\mc G(t,\mathbf{x}))$, 
together with $H(t)/H_{22}=H_1/H_{22}-t$, implies \eqref{form-of-R2}--\eqref{rank-R2} and proves Claim 2.
\hfill$\blacksquare$\\

\noindent 
\textbf{Claim 3.} 
If $\eta=0$, then
$(0,0)\in \partial\mc R_1\cap\mc R_2$.\\

\noindent\textit{Proof of Claim 3.}
By Claim 1, $\eta=0$ implies that
    $(0,0)\in\partial\mc R_1$.
By \eqref{091123-1155} and $\eta=0$, $\cH(A_{\min})=\cH(\mc G(0,0))$.
By Lemma \ref{071123-2008}.\eqref{071123-2008-pt1},
$\cH(A_{\min})\succeq 0$. Hence, $(0,0)\in \mc R_2$,
which proves Claim 3.
\hfill$\blacksquare$\\

\noindent 
\textbf{Claim 4.} 
If $\eta>0$, then:
\begin{itemize}
\item
The set $\cI$ (see \eqref{101123-1934}) has at most 2 elements. 
\smallskip
\item 
$\mc R_1\cap \mc R_2\neq \emptyset$ if and only if $\cI\neq \emptyset.$ 
\smallskip
\item 
If $\cI$ has two elements, then $H_2/H_{22}>0$. 
\smallskip
\item 
If $\cI$ has one element, which we denote by $(\tilde t,\tilde u)$, then one of the following holds:\smallskip
\begin{itemize}
    \item 
        $\mc R_1 \cap \mc R_2=\cI$.
        \smallskip
    \item
    $\partial \mc R_2=\mc R_2=\{(t,u_0)\colon t\leq t_0\}$
    and 
    $\cI\subsetneq \mc R_1\cap \mc R_2= 
    \{(t,u_0)\colon \tilde t\leq t\leq t_0\}$.
\end{itemize}
\end{itemize}
\smallskip

\noindent\textit{Proof of Claim 4.}
    Note that the set $\cI$ is equal to
    $\partial{\mc R}_1 \cap \partial{ \mc R}_2$ (see \eqref{form-of-R1} 
    and \eqref{form-of-R2}).
    Further on, $\partial{\mc R}_1$ is the union of the square root functions
    $\pm\sqrt{\eta \mathbf{t}}$, defined for
    $\mathbf{t}\in [0,\infty)$.
    Similarly, $\partial{\mc R}_2$ is 
    the union of the square root functions
    $u_0\pm\sqrt{(H_1/H_{22}-\mathbf{t})
                        (H_2/H_{22})}$, defined for
    $\mathbf{t}\in (-\infty,t_0]$.
    If $H_2/H_{22}=0$, then the latter could be a half-line
    $\{(t,u_0)\colon t\leq t_0\}$.
    If $\mc R_1\cap \mc R_2\neq \emptyset$, then geometrically it is clear that $\cI$ contains one or two elements. 
    Assume that $\cI$ contains only one element, denoted by $(\tilde t,\tilde u)$. 
    Clearly, $\cI\subseteq \mc R_1\cap \mc R_2$.
    Further on, we either have 
    $\cI=\mc R_1\cap \mc R_2$ or
    $\cI\subsetneq \mc R_1\cap \mc R_2$.
    By the forms of $\partial \mc R_1$ and $\partial \mc R_2$,
    the latter case occurs if
    $H_2/H_{22}=0$ or equivalently
    $\partial \mc R_2=\mc R_2=\{(t,u_0)\colon t\leq t_0\}$.  
    But then the whole line segment 
    $\{(t,u_0\colon \tilde t\leq t\leq t_0\}$
    lies in $\mc R_1$, which proves Claim 4.
\hfill$\blacksquare$\\

\noindent 
\textbf{Claim 5.}
Let $H_2$ (see \eqref{081123-1936}) be positive definite, $(t_1,u_1)\in \partial \mc R_2, (t_2,u_2)\in \partial \mc R_2$
and $u_1\neq u_2$.
Then at least one of
$\cH(\mc G(t_1,u_1))$ and 
$\cH(\mc G(t_2,u_2))$ 
admits a $\RR$--rm.\\

\noindent\textit{Proof of Claim 5.}
Note that $\cH(\mc G(t_i,u_i))$, $i=1,2$, is of the form
    \begin{align*}
    \cH(\mc G(t_i,u_i))
    &=
    \kbordermatrix{
    & \mathit{1} & X & \vec{X}^{(2,k-1)} & X^k\\
    \mathit{1} & \beta_{0,0}-(A_{\min})_{1,1}-t_{i} & 
    \beta_{1,0}-(A_{\min})_{1,2}-u_{i} & (\widehat{h}_{12}^{(1)})^T & \widetilde\beta_{k,0}\\[0.3em]
    X& 
    \beta_{1,0}-(A_{\min})_{1,2}-u_{i} 
    & \beta_{2,0}-(A_{\min})_{1,3}  & (\widehat{h}_{12}^{(2)})^T & \widetilde\beta_{k+1,0}\\[0.3em]
    (\vec{X}^{(2,k-1)})^T&\widehat{h}_{12}^{(1)} & \widehat{h}_{12}^{(2)} & \widehat H_2 & \widehat h_{3}\\[0.3em]
    X^k& \widetilde\beta_{k,0} & \widetilde\beta_{k+1,0} & (\widehat h_3)^T & \widetilde\beta_{2k,0}
    }.
    \end{align*}
Assume on the contrary that 
    none of $\cH(\mc G(t_1,u_1))$ and 
$\cH(\mc G(t_2,u_2))$ admits a $\RR$--rm. 
Theorem \ref{Hamburger}
implies that the column $X^k$ of $\cH(\mc G(t_i,u_i))$, $i=1,2$,
is not in the span of the other columns.
Using this fact, the facts that $\cH(\mc G(t_i,u_i))$, $i=1,2$, are not pd (by $(t_i,u_i)\in \partial \mc R_2$, $i=1,2$) and $H_2$ is pd, it follows that there is a column relation
    $
        \mathit{1}=\sum_{j=1}^{k-1} \alpha^{(i)}_j X^{j},
    $
    $\alpha_j^{(i)}\in \RR$,
    in $\cH(\mc G(t_i,u_i))$, $i=1,2$.
    Since $\cH(\mc G(t_i,u_i))\succeq 0$, $i=1,2$,  
    it follows in particular by Theorem \ref{block-psd},
    used for $(M,A)=(\cH(\mc G(t_i,u_i)),(\cH(\mc G(t_i,u_i)))_{\vec{X}^{(0,k-1)}})$, $i=1,2$,
    that
    \begin{align}
    \label{261023-1701}
        \begin{pmatrix}
            \widetilde\beta_{k,0}&
            \widetilde\beta_{k+1,0}&
            (\widehat h_3)^T
        \end{pmatrix}^T
    &\in \cC\Big(\big(
    \cH(\mc G(t_i,u_i))
    \big)_{\vec{X}^{(0,k-1)}}\Big),\quad i=1,2.
    \end{align}
    Since the first column of 
    $\cH(\mc G(t_i,u_i))\succeq 0$, $i=1,2$,
    is in the span of the others, 
    \eqref{261023-1701} is equivalent to
    \begin{align}
    \label{261023-1743}
        \begin{pmatrix}
            \widetilde\beta_{k,0}&
            \widetilde\beta_{k+1,0}&
            (\widehat h_3)^T
        \end{pmatrix}^T
    &\in \cC\Big(\big(
    \cH(\mc G(t_i,u_i))
    \big)_{\vec{X}^{(0,k-1)},\vec{X}^{(1,k-1)}}\Big),\quad i=1,2.
    \end{align}
    Since 
    $$
    \widetilde H_2
    :=
    \big(\cH(\mc G(t_i,u_i))\big)_{\vec{X}^{(1,k-1)}},\quad i=1,2, 
    $$
    is invertible as a principal submatrix of $H_2$, it follows that
    \begin{align}
    \label{261023-1746}
        \begin{pmatrix}
            \widetilde\beta_{k,0}&
            \widetilde\beta_{k+1,0}&
            (\widehat h_3)^T
        \end{pmatrix}^T
    &=\Big(\big(
    \cH(\mc G(t_i,u_i))
    \big)_{
            \vec{X}^{(0,k-1)},\vec{X}^{(1,k-1)}}\Big) v,
            \quad i=1,2.
    \end{align}
    with 
    $$
    v=
    \widetilde H_2^{-1} 
    \begin{pmatrix}
        \widetilde\beta_{k+1,0} & \widehat h_3
    \end{pmatrix}^T
    =
    \begin{pmatrix}
        v_1 & v_2 & \cdots & v_{k-1}
    \end{pmatrix}^T.
    $$
    If $v_1\neq 0$, this contradicts to \eqref{261023-1746} since 
    $u_{1}\neq u_2$.
    Hence, $v_1=0$.
    By the Hankel structure of $\cH(\mc G(t_i,u_i))$, $i=1,2$, we have that 
    $$
        \big(
            \cH(\mc G(t_i,u_i))
        \big)_{
            \vec{X}^{(0,k-2)},\vec{X}^{(2,k)}}=
        \big(
            \cH(\mc G(t_i,u_i))
        \big)_{
            \vec{X}^{(1,k-1)},\vec{X}^{(1,k-1)}},
            \quad i=1,2.
    $$
    Then \eqref{261023-1746} and $v_1=0$ imply that
    \begin{equation}
    \label{261023-1803}
        \Big(\big(
            \cH(\mc G(t_i,u_i))
        \big)_{
            \vec{X}^{(0,k-2)},\vec{X}^{(2,k)}}\Big)\widetilde v=
        \Big(\big(
            \cH(\mc G(t_i,u_i))
        \big)_{
            \vec{X}^{(1,k-1)},\vec{X}^{(1,k-1)}}\Big)\widetilde v=\mathbf{0}_{k+1,1},
    \end{equation}
    where
    $
    \widetilde v
    =
    \begin{pmatrix}
        v_2 & \cdots & v_{k-1} & -1 
    \end{pmatrix}.
    $
    Since $\big(\cH(\mc G(t_i,u_i))\big)_{
            \vec{X}^{(1,k-1)},\vec{X}^{(1,k-1)}},$
            $i=1,2$,
    is a principal submatrix of $H_{2}$,
    \eqref{261023-1803} contradicts to $H_2$ being pd.
    This proves Claim 5.\hfill $\blacksquare$\\

    Now we prove the implication 
    $\eqref{251023-1603-v2}\Rightarrow
    \text{Theorem }\ref{221023-1854}.\eqref{221023-1854-pt3}$.
    Since $(t_0,u_0)\in \mc R_1$, it follows that
    $\mc R_1\neq \emptyset.$
    By \eqref{form-of-R1}, $\eta\geq 0$.
    We separate two cases according to the value of $\eta$.\\

    \noindent
    \textbf{Case 1: $\eta=0$.}
    We separate two cases according to the invertibility of $H_2$.\\
    
    \noindent
    \textbf{Case 1.1: $H_2$ is not pd.}
    Since $H_2$ is not pd, then by Theorem \ref{Hamburger}, the last column of $\cH(\mc G(t_0,u_0))$
    is in the span of the previous ones. But then by rg, the last column of $H_2$
    is in the span of the previous ones. This is the case 
    Theorem \ref{221023-1854}.\eqref{221023-1857-pt3.1.2}.\\

    \noindent
    \textbf{Case 1.2: $H_2$ is pd.} 
    We separate two cases according to the invertibility of 
    $(\cH(A_{\min}))_{\vec{X}^{(0,k-1)}}$.\\

    \noindent
    \textbf{Case 1.2.1: $\Rank (\cH(A_{\min})_{\vec{X}^{(0,k-1)}})=k$.}
    This is the case
    Theorem \ref{221023-1854}.\eqref{221023-1857-pt3.1.1}.\\

    \noindent
    \textbf{Case 1.2.2: $\Rank (\cH(A_{\min})_{\vec{X}^{(0,k-1)}})<k$.}
    We will prove that this case cannot occur.
    It follows from the assumption in this case that 
    $\Rank \cH(A_{\min})=\Rank H_2=k$.
    Further on,
    the last column of
    $\cH(A_{\min})$
    cannot be in the span of the previous ones
    (otherwise $\Rank \cH(A_{\min})<k$).
    Hence, by Theorem \ref{Hamburger},
    $\cH(A_{\min})=\cH(\mc G(0,0))$
    does not admit a $\RR$--rm.
    Using this fact and Claim 3, 
    $(0,0)\in \partial \mc R_2$.
    If $t_0=0$, then
    $\mc R_1\cap\mc R_2=
    \{(0,0)\}$,
    which contradicts to the third condition in  
    \eqref{251023-1603-v2}.
    So $0<t_0$ must hold.
    Since $\eta=0$, Claim 1
    implies that
    $\mc R_1=
    \{(t,0)
    \colon
    t\geq 0    
    \}$
    is a horizontal half-line.
    By the form of $\partial \mc R_2$,
    which is the union of the graphs of two square root functions on the interval 
    $(-\infty,t_0]$, intersecting in 
    the point $(t_0,u_0)$
    and such that $(t_0,u_0)\in \partial R_2$, 
    it follows that   
    $\mc R_1\cap\mc R_2=
    \{(0,0)\}$. Note that by $H_2\succ 0$, we have $H_2/H_{22}>0$ and hence $h(t)\not\equiv 0$ (see \eqref{101123-1550}),
    which implies that the square root functions are indeed not just a horizontal half-line.
    As above this contradicts to the third condition in  
    \eqref{251023-1603-v2}. Hence, Case 1.2.2 cannot occur.\\

    \noindent
    \textbf{Case 2: $\eta>0$.}
    By assumptions, $(t_0,u_0)\in \mc R_1\cap \mc R_2$.
    By Claim 4, $\cI\neq \emptyset$ and $\cI$ has one or two elements.
    We separate two cases according to the number of elements in $\cI$.\\

    \noindent
    \textbf{Case 2.1: $\cI$ has two elements.}
    By Claim 4, $H_2/H_{22}>0$. 
    If $H_2$ is not pd, then
    the fact that $\cH(\mc G(t_0,u_0))$ has a $\RR$--rm, implies that
    $H_2/H_{22}=0$, which is a contradiction.
    Indeed, if $H_2/H_{22}>0$ and $H_2$ is not pd, then there is a nontrivial column relation among columns $X^2,\ldots,X^k$ in $H_2$. By Proposition \ref{extension-principle},
    the same holds for $\cH(\mc G(t_0,u_0))$. 
    Let $\sum_{i=0}^{k-2} c_i X^{i+2}=\mathbf{0}$ be the nontrivial column relation in $\cH(\mc G(t_0,u_0))$.
    But then
        $\cZ(x^2\sum_{i=0}^{k-2} c_i x^i)=\cZ(x\sum_{i=0}^{k-2} c_i x^i)$
    and it follows by \cite{CF96} that 
    $\sum_{i=0}^{k-2} c_i X^{i+1}=\mathbf{0}$
    is also a nontrivial column relation in $\cH(\mc G(t_0,u_0))$.
    In particular, $H_2/H_{22}=0$.
    Hence, $H_2$ is pd.
    This is the case
    Theorem \ref{221023-1854}.\eqref{221023-1857-pt3.2.1}.
    \\

    \noindent
    \textbf{Case 2.2: $\cI$ has one element.}
    Let us denote this element by $(\tilde t,\tilde u)$.
    By Claim 4, 
    $\cI=\mc R_1\cap \mc R_2$
    or
    $\partial \mc R_2=\mc R_2=\{(t,u_0)\colon t\leq t_0\}$
    and 
    $\cI\subsetneq \mc R_1\cap \mc R_2= 
    \{(t,u_0)\colon \tilde t\leq t\leq t_0\}$
    .
    We separate two cases according to these two possibilities.\\

    \noindent
    \textbf{Case 2.2.1: $\cI=\mc R_1\cap \mc R_2$.} In this case 
    $(t_0,u_0)=(\tilde t,\tilde u)$ and hence
    $\cH(\mc G(\tilde t,\tilde u)$ admits a $\RR$--rm. Since $(\tilde t,\tilde u)\in \partial \mc R_1$, $\cH(\mc G(\tilde t,\tilde u))$ is not pd. Hence, by Theorem \ref{Hamburger}, the statement Theorem \ref{221023-1854}.\eqref{221023-1857-pt3.2.2} holds.\\

    \noindent
    \textbf{Case 2.2.2: 
    $\partial \mc R_2=\mc R_2=\{(t,u_0)\colon t\leq t_0\}$
    and 
    $\cI\subsetneq \mc R_1\cap \mc R_2= 
    \{(t,u_0)\colon \tilde t\leq t\leq t_0\}$.}
    By \eqref{form-of-R2}, it follows that $H_2/H_{22}=0$ (see the definition
    \eqref{101123-1550} of $h(\mathbf{t})$).
    Since $H_2$ is not pd, Theorem \ref{Hamburger} used for  
    $\cH(\mc G(t_0,u_0))$, implies that the last column of $H_2$ is in the
    span of the others. Hence, the same holds by Proposition \ref{extension-principle} for $\cH(\mc G(\tilde t,\tilde u))$ and $\cH(\mc G(\tilde t,\tilde u))$ admits a $\RR$--rm by Theorem \ref{Hamburger}.
    Since $\cH(\mc G(\tilde t,\tilde u))$ is not pd, it in particular satisfies
    \eqref{111123-1803}. Hence, we are in the case Theorem \ref{221023-1854}.\eqref{221023-1857-pt3.2.2}.
    \\
    
    \noindent This concludes the proof of the implication
        $\eqref{251023-1603-v2}\Rightarrow
    \text{Theorem }\ref{221023-1854}.\eqref{221023-1854-pt3}$.\\

    Next we prove the implication 
    $
    \text{Theorem }\ref{221023-1854}.\eqref{221023-1854-pt3}\Rightarrow
    \eqref{251023-1603-v2}
    $. 
    We separate four cases according to the assumptions in $\text{Theorem }\ref{221023-1854}.\eqref{221023-1854-pt3}$.\\

    \noindent\textbf{Case 1: Theorem \ref{221023-1854}.\eqref{221023-1857-pt3.1.1} holds.}
    By Claim 3, $(0,0)\in \mc R_1\cap \mc R_2$.
    This and the assumption $\Rank (\cH(A_{\min}))_{\vec{X}^{(0,k-1)}}=k$,
    imply by Theorem \ref{Hamburger}, 
    that
    $\cH(\mc G(0,0))=\cH(A_{\min})$ admits a $\RR$--rm.
    This proves \eqref{251023-1603-v2} in case of Theorem \ref{221023-1854}.\eqref{221023-1857-pt3.1.1}.\\

    \noindent\textbf{Case 2: Theorem \ref{221023-1854}.\eqref{221023-1857-pt3.1.2}
    holds.}
    By Claim 3, $(0,0)\in \mc R_1\cap \mc R_2$.
    Since the last column of $H_2$
    is by assumption in the span of the previous ones, the same holds for 
    $\cH(\mc G(0,0))$ by Proposition \ref{extension-principle}.
    By Theorem \ref{Hamburger},
    $\cH(\mc G(0,0))$ admits a $\RR$--rm. This proves \eqref{251023-1603-v2} in case of Theorem \ref{221023-1854}.\eqref{221023-1857-pt3.1.2}.\\

    \noindent\textbf{Case 3: Theorem \ref{221023-1854}.\eqref{221023-1857-pt3.2.1}
    holds.}
    By assumption, $\cI=\partial \mc R_1\cap \partial \mc R_2=\{(t_1,u_1),(t_2,u_2)\}$.
    Since $H_2$ is pd, $\partial \mc R_{2}$ is not a half-line and hence $u_1\neq u_2$.
    By Claim 5, at least one of 
        $\cH(\mc G(t_{1},u_{1}))$
    and
        $\cH(\mc G(t_2,u_2))$
    admits a $\RR$--rm. 
    This proves \eqref{251023-1603-v2} in case of Theorem \ref{221023-1854}.\eqref{221023-1857-pt3.2.1}.
    \\

    \noindent\textbf{Case 4: Theorem \ref{221023-1854}.\eqref{221023-1857-pt3.2.2}
    holds.}
    The assumptions imply by Theorem \ref{Hamburger}, that 
    $\cH(\mc G(\tilde t,\tilde u))$
    admits a $\RR$--rm.
        This proves \eqref{251023-1603-v2} in case of Theorem \ref{221023-1854}.\eqref{221023-1857-pt3.2.2}.\\
        
        \noindent This concludes the proof of the implication
        $\text{Theorem }\ref{221023-1854}.\eqref{221023-1854-pt3}\Rightarrow\eqref{251023-1603-v2}.
    $\\
    
    Up to now we established the equivalence
    $\eqref{221023-1854-pt1}
    \Leftrightarrow
    \eqref{221023-1854-pt3}$
    in Theorem $\ref{221023-1854}$.
    It remains to prove the moreover part.    
    We observe again the proof of the implication
    $\eqref{221023-1854-pt3}
    \Rightarrow
    \eqref{251023-1603-v2}$.
    By Lemma \ref{071123-2008}.\eqref{071123-2008-pt3},
    \begin{equation}
        \label{261023-1536}
        \Rank \widetilde{\mc M}(k;\beta)
        =\Rank \cF(A_{\min})+
        \Rank \cH(A_{\min}).
    \end{equation}
    
    In the proof of the implications
    Theorem $\ref{221023-1854}.\eqref{221023-1857-pt3.1.1}\Rightarrow
    \eqref{251023-1603-v2}$
    and
    Theorem $\ref{221023-1854}.\eqref{221023-1857-pt3.1.2}\Rightarrow
    \eqref{251023-1603-v2}$
    we established that
    $\cH(\mc G(0,0))$
    has a $\RR$--rm.
    By Theorem \ref{Hamburger}, there also exists 
    a $(\Rank\cH(\mc G(0,0)))$--atomic one. By Theorem \ref{circle-TMP}, the sequence with the moment matrix
    $\cF(\mc G(0,0))$ can be represented by a
    $(\Rank\cF(\mc G(0,0)))$--atomic
    $\cZ(ay+x^2+y^2)$--rm. By \eqref{261023-1536} and $\mc G(0,0)=A_{\min}$ if $\eta=0$, in these two cases $\beta$ has a $(\Rank \widetilde{\mc M}(k;\beta))$--atomic $\cZ(p)$--rm.

    In the proof of the implication
    Theorem $\ref{221023-1854}.\eqref{221023-1857-pt3.2.1}\Rightarrow
    \eqref{251023-1603-v2}$
    we established that $\cH(\mc G(t',u'))$ has a $\RR$--rm
    for some $(t',u')\in \cI$.
    Analogously as for the point $(0,0)$ in the previous paragraph, it follows that
    $\beta$ has a 
    $(
    \Rank\cF(\mc G(t',u')) 
    +
    \Rank \cH(\mc G(t',u'))
    )$--atomic $\cZ(p)$--rm.
    Using \eqref{rank-R1}, \eqref{rank-R2} and $\Rank H_2=\Rank H_{22}+1$ (by $H_2$ being pd),
    it follows that
    \begin{equation}
    \label{231123-2303}    
    \Rank\cF(\mc G(t',u')) 
        +
        \Rank \cH(\mc G(t',u'))
        =
        \Rank \cF(A_{\min})+\Rank H_2+1.
    \end{equation}
    We separate two cases:
    \begin{itemize}
        \item If $\cH(A_{\min})$ is pd, then 
    $\Rank \cH(A_{\min})=\Rank H_{2}+1$.
    This, \eqref{261023-1536}
    and \eqref{231123-2303} imply that
    $\beta$ admits a $(\Rank \widetilde{\mc{M}}(k;\beta))$--atomic $\cZ(p)$--rm.
        \item 
        If $\cH(A_{\min})$ is not pd,
        then we must have
        $\Rank \cH(A_{\min})=\Rank H_{2}$,
    Otherwise we have
    $(\cH(A_{\min}))_{\vec{X}^{(1,k)}}/H_{22}=0$
    and hence 
    $(\cH(A_{\min}-\eta E_{2,2}^{(k+1)}))_{\vec{X}^{(1,k)}}/H_{22}<0$,
    which contradicts to 
    $\cH(A_{\min}-\eta E_{2,2}^{(k+1)})$
    being psd. Hence, in this case 
    $\beta$ has a 
    $(\Rank \widetilde{\mc{M}}(k;\beta)+1)$--atomic $\cZ(p)$--rm.
    Moreover, there cannot exist a
    $(\Rank \widetilde{\mc{M}}(k;\beta))$--atomic $\cZ(p)$--rm. Indeed, 
    since $\eta>0$, at least $\Rank \cF(A_{\min})+1$ (resp.\ $\Rank H_2$) atoms are needed to represent 
    $\cF(\mc G(t'',u''))$ (resp.\ $\cH(\mc G(t'',u''))$) for any $(t'',u'')\in \mc R_1\cap \mc R_2$ 
    (see \eqref{rank-R1} and \eqref{rank-R2}).
    Hence,
    at least    
    $\Rank\cF(A_{\min})
    +
    \Rank H_{2}+1$ atoms are needed in a  
    $\cZ(p)$--rm for any $(t'',u'')\in \mc R_1\cap \mc R_2$.
    \end{itemize}

    In the proof of the implication
    Theorem $\ref{221023-1854}.\eqref{221023-1857-pt3.2.2}\Rightarrow
    \eqref{251023-1603-v2}$
    we established that $\cH(\mc G(\tilde t,\tilde u))$
    has a $\RR$--rm.
    Analogously as for the point $(0,0)$ in two paragraphs above, it follows that
    $\beta$ has a 
    $(
    \Rank\cF(\mc G(\tilde t,\tilde u)) 
    +
    \Rank \cH(\mc G(\tilde t,\tilde u))
    )$--atomic $\cZ(p)$--rm.
    By \eqref{rank-R1} and \eqref{rank-R2}, this measure is
    $(
    \Rank\cF(A_{\min})
    +
    \Rank H_{22}+2
    )$--atomic.
    \begin{itemize}
        \item If $\cH(A_{\min})$ is pd, then 
    $\Rank \cH(A_{\min})=\Rank H_{22}+2$.
    This and \eqref{261023-1536} imply that
    $\beta$ admits a $(\Rank \widetilde{\mc{M}}(k;\beta))$--atomic $\cZ(p)$--rm.
        \item 
        If $\cH(A_{\min})$ is not pd, then we have
        $\Rank \cH(A_{\min})=\Rank H_{22}+1$,
    since otherwise the equality
    $(\cH(A_{\min}))_{\vec{X}^{(1,k)}}/H_{22}=0$
    implies
    $(\cH(A_{\min}-\eta E_{2,2}^{(k+1)}))_{\vec{X}^{(1,k)}}/H_{22}<0$,
    which contradicts to 
    $\cH(A_{\min}-\eta E_{2,2}^{(k+1)})$
    being psd. Hence, in this case 
    $\beta$ has a 
    $(\Rank \widetilde{\mc{M}}(k;\beta)+1)$--atomic $\cZ(p)$--rm.
    Moreover, there cannot exist a
    $(\Rank \widetilde{\mc{M}}(k;\beta))$--atomic $\cZ(p)$--rm in this case. Indeed, 
    $$
    (\mc R_1\cap \mc R_2)\setminus \cI
    =
    (\partial \mc R_1\cap \interior{\mc R}_2)
    \cup
    (\interior{\mc R}_1\cap \partial \mc R_2)
    \cup
    (\interior{\mc R}_1\cap \interior{\mc R}_2).
    $$
    Using \eqref{rank-R1} and \eqref{rank-R2},
    in every point from 
    $(\mc R_1\cap \mc R_2)\setminus \cI$
    at least    
    $\Rank\cF(A_{\min})
    +
    \Rank H_{22}+2$ atoms are needed in a  
    $\cZ(p)$--rm.
    \end{itemize}

    \noindent This concludes the proof of the moreover part.

    Since for a $p$--pure sequence 
    with $\widetilde{\mc M}(k;\beta))\succeq 0$, \eqref{261023-1536} implies that $\cH(A_{\min})$ is pd, it follows by the moreover part that 
    the existence of a $\cZ(p)$--rm
    implies the existence of 
    a $(\Rank \widetilde{\mc M}(k;\beta))$--atomic $\cZ(p)$--rm.
\end{proof}

The following example demonstrates the use of Theorem \ref{221023-1854}
to show that there exists a bivariate $y(-2y+x^2+y^2)$--pure sequence $\beta$ of degree 6 with a positive semidefinite $\mc M(3)$ and without a $\cZ(y(-2y+x^2+y^2))$--rm.

\begin{example}\label{ex-line-plus-circle-no-rm}
Let $\beta$ be a bivariate degree 6 sequence given by
\begin{align*}
    \beta_{00} & = 10,
    & \beta_{10} & =\frac{38}{5},
    & \beta_{01} & = \frac{39}{5},\\[0.2em]
      \beta_{20} & = \frac{602}{25},
    &  \beta_{11} &= \frac{3}{25},
     &\beta_{02} & =\frac{313}{25},\\[0.2em]
    \beta_{30} & =\frac{9152}{125},
    & \beta_{21} & =\frac{421}{125},
    & \beta_{12} & =\frac{3}{125},\\[0.2em]
     \beta_{03} & =\frac{2709}{125},
    & \beta_{40} & =\frac{172118}{625},
    & \beta_{31} & =\frac{27}{625},\\[0.2em]
     \beta_{22} & =\frac{2717}{625},
    & \beta_{13} & =\frac{3}{625},
    &     \beta_{04} & =\frac{24373}{625},\\[0.2em]
    \beta_{50} &=
        \frac{3303368}{3125},
    &\beta_{41} &=
        \frac{7789}{3125},
    &\beta_{32} &=
        \frac{27}{3125},\\[0.2em]
    \beta_{23} &=
        \frac{19381}{3125},
    &\beta_{14} &=
        \frac{3}{3125},
    &\beta_{05} &=
        \frac{224349}{3125},\\[0.2em]
    \beta_{60} &= 
        4156,
    &\beta_{51} &=
        \frac{243}{15625},
    &\beta_{42} &=
        \frac{44453}{15625},\\[0.2em]
    \beta_{33} &=
        \frac{27}{15625},
    &\beta_{24} &=
        \frac{149357}{15625},
    &\beta_{15} &=
        \frac{3}{15625},\\[0.2em]
    \beta_{06} &=
        \frac{2094133}{15625}.
\end{align*}
Assume the notation as in Theorem \ref{221023-1854}. 
  $\widetilde{\mc M}(3)$ is psd with the eigenvalues 
 $\approx 4445$, $\approx 189.2$, $\approx 16.6$, $\approx 11.9$, $\approx 3.2$, $\approx 1.22$, $\approx 0.57$, 
 $\approx 0.022$, $\approx 0.0030$, $0$
and the column relation $$-2Y^2+X^2Y+Y^3=0.$$
We have that
$$
A_{\min}
=
\begin{pmatrix}
    \frac{324330}{55873} & \frac{132789}{278915} & \frac{77}{25} & \frac{27}{125}\\[0.5em]
    \frac{132789}{278915} & \frac{4180091}{1394575} & \frac{27}{125} & \frac{1493}{625}\\[0.5em]
    \frac{77}{25} & \frac{27}{125} & \frac{1493}{625} & \frac{243}{3125}\\[0.5em]
    \frac{27}{125} & \frac{1493}{625} & \frac{243}{3125} & \frac{33437}{15625}
\end{pmatrix}
$$
and so 
$$\eta=\frac{77}{25}-\frac{4180091}{1394575}=\frac{4608}{55783}.$$
The matrix $H_{2}$ is equal to:
\begin{align*}
H_{2}
&=\begin{pmatrix}
21 & 73 & 273\\
73 & 273 & 1057\\
273 & 1057 & \frac{64904063}{15625}
\end{pmatrix}.
\end{align*}
The eigenvalues of $H_2$ are $\approx 4441.1$, $\approx 6.74$, $\approx -0.019$ and hence $H_2$ is not psd.
By Theorem \ref{221023-1854}, $\beta$ does not have a $\cZ(y(-2y+x^2+y^2))$--rm, since by \eqref{221023-1857-pt3.2} of Theorem \ref{221023-1854},
$H_2$ should be psd.
\end{example}



\section{Parabolic type relation: 
            $p(x,y)=y(x-y^2)$.
        }
\label{parabolic}

In this section we solve the $\cZ(p)$--TMP for 
the sequence $\beta=\{\beta_i\}_{i,j\in \ZZ_+,i+j\leq 2k}$ 
of degree $2k$, $k\geq 3$,
where $p(x,y)=y(x-y^2)$.
Assume the notation from Section \ref{Section-common-approach}.
If $\beta$ admits a $\cZ(p)$--TMP, then $\mc M(k;\beta)$ 
must satisfy the relations
\begin{equation}
    \label{130623-0758}
	Y^{3+j}X^{i}=Y^{1+j}X^{i+1}\quad
        \text{for }i,j\in \ZZ_+\text{ such that }i+j\leq k-3.
\end{equation}
In the presence of all column relations \eqref{130623-0758}, the column space $\cC(\mc M(k;\beta))$ is spanned by the columns in the set
    \begin{equation}
    \label{221023-1848}
    \cT=
        \vec{X}^{(0,k)}
        \cup 
        Y\vec{X}^{(0,k-1)}
        \cup
        Y^2\vec{X}^{(0,k-2)},
    \end{equation}
    where 
    $$
        Y^i\vec{X}^{(j,\ell)}:=(Y^iX^j,Y^iX^{j+1},\ldots,Y^iX^{\ell})
        \quad\text{with }i,j,\ell\in \ZZ_+,\; j\leq \ell,\; i+\ell\leq k.
    $$
    Let $\widetilde \cM(k;\beta)$ 
    be as in 
    \eqref{071123-1939}.
    Let 
    \begin{equation}
        \label{A-min-parabolic}
        A_{\min}:=A_{12}(A_{22})^{\dagger} (A_{12})^T.
    \end{equation}
    As described in Remark \ref{general-procedure},
    $A_{\min}$ might need to be changed to
    $$
    \widehat A_{\min}
    =A_{\min}+\eta \big(E_{1,k+1}^{(k+1)}+E_{k+1,1}^{(k+1)}\big),
    $$
where
    $$
    \eta:=(A_{\min})_{2,k}-(A_{\min})_{1,k+1}.
    $$
    Let $\cF(\mathbf{A})$ and $\cH(\mathbf{A})$ 
    be as in 
    \eqref{071123-1940}.
Define also the matrix function 
    \begin{equation}
        \label{141123-0705}
        \mc G:\RR^2\to S_{k+1},\qquad 
	\mc G(\mathbf{t},\mathbf{u})=
        \widehat A_{\min}
        +\mathbf{t}E_{1,1}^{(k+1)}
        +\mathbf{u}
        E_{k+1,k+1}^{(k+1)}.
    \end{equation}
Write
        \begin{align}
        \label{141123-0720}
        \begin{split}
        \cH(\widehat A_{\min})
        &=
        \kbordermatrix{ 
                & \mathit{1} & \vec{X}^{(1,k-1)} & X^k \\
            \mathit{1}   & \beta_{0,0}-(A_{\min})_{1,1} & (h_{12})^T &  \beta_{k,0}-(A_{\min})_{2,k}\\[0.2em]
            (\vec{X}^{(1,k-1)})^T
                & h_{12} & H_{22} & h_{23}\\[0.2em]
            X^k & \beta_{k,0}-(A_{\min})_{2,k} & (h_{23})^T & \beta_{2k,0}-(A_{\min})_{k+1,k+1}},\\[0.5em]
        H_1
        &:=(\cH(\widehat A_{\min}))_{\vec{X}^{(0,k-1)}}
        =
        \kbordermatrix{
            & \mathit{1} & \vec{X}^{(1,k-1)}\\
           \mathit{1} & \beta_{0,0}-(A_{\min})_{1,1} & (h_{12})^T\\[0.2em]
          (\vec{X}^{(1,k-1)})^T &  h_{12} & H_{22}
        },\\[0.5em]
        H_2
        &:=(\cH(\widehat A_{\min}))_{\vec{X}^{(1,k)}}
        =
        \kbordermatrix{
            & \vec{X}^{(1,k-1)} & X^k\\
        (\vec{X}^{(1,k-1)})^T &
            H_{22} & h_{23}\\[0.2em]
        X^k  &  (h_{23})^T & \beta_{2k,0}-(A_{\min})_{k+1,k+1}}.
        \end{split}
        \end{align}
Let us define the matrix
\begin{align*}
K
&:=
\cH(\widehat A_{\min})/H_{22}\\[0.3em]
&=
\begin{pmatrix}
                \beta_{0,0}-(A_{\min})_{1,1} &  \beta_{k,0}-(A_{\min})_{2,k}\\[0.2em]
            \beta_{k,0}-(A_{\min})_{2,k} & \beta_{2k,0}-(A_{\min})_{k+1,k+1}.
\end{pmatrix}
-
\begin{pmatrix}
    (h_{12})^T\\
    (h_{23})^T
\end{pmatrix}
(H_{22})^\dagger
\begin{pmatrix}
    h_{12} & h_{23}
\end{pmatrix}\\[0.3em]
&:=
\begin{pmatrix}
                \beta_{0,0}-(A_{\min})_{1,1}-(h_{12})^T(H_{22})^\dagger h_{12}&  \beta_{k,0}-(A_{\min})_{2,k}-(h_{12})^T(H_{22})^\dagger h_{23}\\[0.3em]
            \beta_{k,0}-(A_{\min})_{2,k}-(h_{23})^T(H_{22})^\dagger h_{12} & \beta_{2k,0}-(A_{\min})_{k+1,k+1}-(h_{12})^T(H_{22})^\dagger h_{12}
\end{pmatrix}\\[0.3em]
&:=
\begin{pmatrix}
                k_{11} &  k_{12}\\
            k_{12} & k_{22}
\end{pmatrix}.
\end{align*}
Let
$$
    \widehat \cT
    =\{
    \mathit{1},Y,X,XY,X^2,X^2Y,\ldots,X^{i},X^iY,\ldots,
    X^{k-1},X^{k-1}Y,X^k\},
$$
and
\begin{align}
    \label{order-parabolic}
    \begin{split}
    &\widehat P
    \text{ be a permutation matrix such that moment matrix }
    \widehat{\mc M}(k;\beta):=\widehat P\mc M(k;\beta)(\widehat P)^T\\
    &\text{has rows and columns indexed in the order }
        \widehat \cT,
        \cC\setminus \widehat\cT.
    \end{split}
    \end{align} 
Write

\begin{align}
\label{widehat-F}
\begin{split}
&\widehat\cF(\mathbf{t},\mathbf{u})
=
\widehat P
\cF(\mc G(\mathbf{t},\mathbf{u}))
(\widehat P)^T\\
&=
\kbordermatrix
{
& \mathit{1} & \widehat \cT\setminus \{\mathit{1},X^k\} & X^k & \cC\setminus \widehat\cT\\[0.2em]
\mathit{1} & (A_{\min})_{1,1}+\mathbf{t} & (f_{12})^T & (A_{\min})_{2,k} & (f_{14})^T\\[0.2em]
(\widehat \cT\setminus \{\mathit{1},X^k\})^T
& f_{12} & F_{22} & f_{23} & F_{24}\\[0.2em]
X^k & (A_{\min})_{2,k} & (f_{23})^T & (A_{\min})_{k+1,k+1}+\mathbf{u} & (f_{34})^T\\[0.2em]
\cC\setminus\widehat\cT & f_{14} & (F_{24})^T & f_{34} & F_{44}
}.
\end{split}
\end{align}

The solution to the cubic parabolic type relation TMP is the following. 

	\begin{theorem}
		\label{131023-0847}
	Let $p(x,y)=y(x-y^2)$
	and $\beta:=\beta^{(2k)}=(\beta_{i,j})_{i,j\in \ZZ_+,i+j\leq 2k}$, where $k\geq 3$.
	Assume also the notation above.
	Then the following statements are equivalent:
	\begin{enumerate}	
		\item\label{131023-0847-pt1} 
                $\beta$ has a $\cZ(p)$--representing measure.
                \smallskip
		\item\label{131023-0847-pt3}  
  $\widetilde{\mc{M}}(k;\beta)$ is positive semidefinite, the relations
	\begin{equation}\label{131023-0847-equation}
		\beta_{i,j+3}=\beta_{i+1,j+1} 
            \quad\text{hold for every }i,j\in \ZZ_+\text{ with }i+j\leq 2k-3,
	\end{equation}
	$\cH(\widehat A_{\min})$ is positive semidefinite,
        defining real numbers
				\begin{align}
					\label{131023-0847-eq2}
					\begin{split}
					t_1
					&=H_1/H_{22}
                        =\beta_{0,0}-(A_{\min})_{1,1}
                        -
                        (h_{12})^T
                        (H_{22})^\dagger 
                        h_{12},\\[0.2em]	
                        u_1
                        &=H_2/H_{22}
                        =
                        \beta_{2k,0}
                        -(A_{\min})_{k+1,k+1}
                        -
                        (h_{23})^T
                        (H_{22})^{\dagger}
                        h_{23},
		              \end{split}
				\end{align}
        and the property
        \begin{align}
        \label{measure-property}
            &
            (\cH(\widehat A_{\min}))_{\vec{X}^{(0,k-1)}}\succ 0 \quad\text{or}\quad 
            \Rank (\cH(\widehat A_{\min}))_{\vec{X}^{(0,k-1)}}
                    =
                    \Rank \cH(\widehat A_{\min}),
        \end{align}
	one of the following statements holds:
        \smallskip
        \begin{enumerate}
            \item 
            \label{parabolic-pt1}
                $F_{22}$ is not positive definite,
                $\eta=0$ and 
                \eqref{measure-property} holds.
            \smallskip
            \item 
            \label{parabolic-pt2}
                $F_{22}$ is positive definite,
                $H_{22}$ is not positive definite
                and one of the following holds:
                \smallskip
                \begin{enumerate}
                    \item
                    \label{parabolic-pt2-1}
                        $u_1=\eta=0$.
                    \smallskip
                    \item
                    \label{parabolic-pt2-2}
                        $u_1>0$, $t_1>0$, $t_1u_1\geq \eta^2$
                        and
                        $
                        \beta_{k,0}-(A_{\min})_{2,k}=
                        (h_{12})^T(H_{22})^{\dagger} h_{23}.
                        $
                \end{enumerate}
            \smallskip
            \item
            \label{parabolic-pt3}
                $F_{22},H_{22}$ are positive definite
                and one of the following holds:
                \smallskip
                \begin{enumerate}
                    \item
                    \label{parabolic-pt3-1}
                        $\eta=0$ and \eqref{measure-property} holds.
                        \smallskip
                    \item
                    \label{parabolic-pt3-2}
                        $\eta\neq 0$ and
                        \begin{equation}
                        \label{det-ineq}
                        (\sqrt{k_{11}k_{22}}-\sign(k_{12})k_{12})^2\geq \eta^2,
                        \end{equation}
                        where $\sign$ is the sign function
                        and $\sign(0)=0$.
                \end{enumerate}
        \end{enumerate}
        \end{enumerate}
  	Moreover, if a $\mc Z(p)$--representing measure for $\beta$ exists, then:
        \begin{itemize}
        \item There exists at most $(\Rank \widetilde{\mc{M}}(k;\beta)+1)$--atomic $\cZ(p)$--representing 
            measure.
        \item There exists a $(\Rank \widetilde{\mc{M}}(k;\beta))$--atomic $\cZ(p)$--representing measure
	   if and only if any of the following holds:
        \smallskip
        \begin{itemize}
            \item 
                $\eta=0$.
                \smallskip
            \item 
                $\Rank \cH(A_{\min})=
                \Rank H_{22}+2.
                $ 
                \smallskip
            \item 
                $\Rank \cH(A_{\min})=
                \Rank H_{22}+1
                $ 
                and
                one of the following holds:
                \smallskip
                \begin{itemize}
                    \item 
                        $H_{22}$ is not positive definite
                        and 
                        $t_1u_1=\eta^2$.
                    \smallskip
                    \item $H_{22}$ is positive definite,        
                        $k_{12}=0$ and $k_{11}k_{22}=\eta^2$.
                \end{itemize}
                \smallskip
        \end{itemize}
        \end{itemize}
        In particular, a $p$--pure sequence $\beta$ with a 
 $\cZ(p)$--representing 
            measure
admits a $(\Rank \widetilde{\mc{M}}(k;\beta))$--atomic $\cZ(p)$--representing 
            measure.
\end{theorem}

\begin{remark}
\label{proof-line+parabola}
In this remark we explain the idea of the proof of Theorem \ref{131023-0847}
and the meaning of conditions in the statement of the theorem.

By Lemmas \ref{071123-0646}--\ref{071123-1942}, the existence of a $\mc Z(p)$--rm for $\beta$ is equivalent to the existence of $t,u\in \RR$ such 
that 
    $\cF(\mc G(t,u))$ 
admits a $\cZ(x-y^2)$--rm 
and 
    $\cH(\mc G(t,u))$ 
admits a $\RR$--rm. 
    Let 
\begin{align*}
    \mc R_1
    &=\big\{(t,u)\in \RR^2\colon \cF(\mc G(t,u))\succeq 0\big\}
    \quad\text{and}\quad
    \mc R_2
    =\big\{(t,u)\in \RR^2\colon \cH(\mc G(t,u))\succeq 0\big\}.
\end{align*}
We denote by $\partial R_i$ and $\interior{R}_i$ the topological boundary and the interior of the set $R_i$, respectively.
 By the necessary conditions for the existence of a 
    $\cZ(p)$--rm \cite{CF04,Fia95,CF96}, 
    $\widetilde{\mc M}(k;\beta)$ must be psd and the relations
    \eqref{131023-0847-equation} must hold.
Then Theorem \ref{131023-0847}.\eqref{131023-0847-pt1} is equivalent to
\begin{align}
\label{151123-0822}
\begin{split}
&\widetilde{\cM}(k;\beta)\succeq 0, 
\text{ the relations }
\eqref{131023-0847-equation}\text{ hold 
and }\\
&\exists (t_0,u_0)\in \mc R_1\cap\mc R_2:
\cF(\mc G(t_0,u_0)) 
\text{ and }
\cH(\mc G(t_0,u_0))\text{ admit }\\
&\hspace{4cm}
\text{a }\cZ(x-y^2)\text{--rm and a }\RR\text{--rm, respectively.}
\end{split}
\end{align}
In the proof of Theorem \ref{131023-0847}
we show that 
\eqref{151123-0822} is equivalent to
Theorem \ref{131023-0847}.\eqref{131023-0847-pt3}:
\begin{enumerate}
\item
First we establish (see Claims 1 and 2 below)
that the form of:
\begin{itemize}
    \item 
$\mc R_1$ is one of the following:
    \begin{center}
    \begin{tabular}{lr}
    \includegraphics[width=5cm]{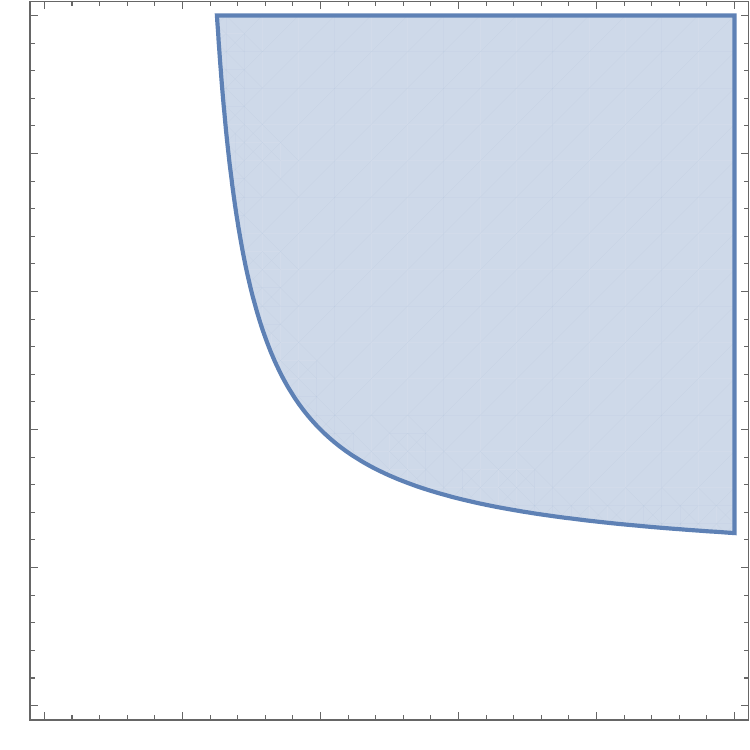}\hspace{1cm}
    &
    \includegraphics[width=5cm]{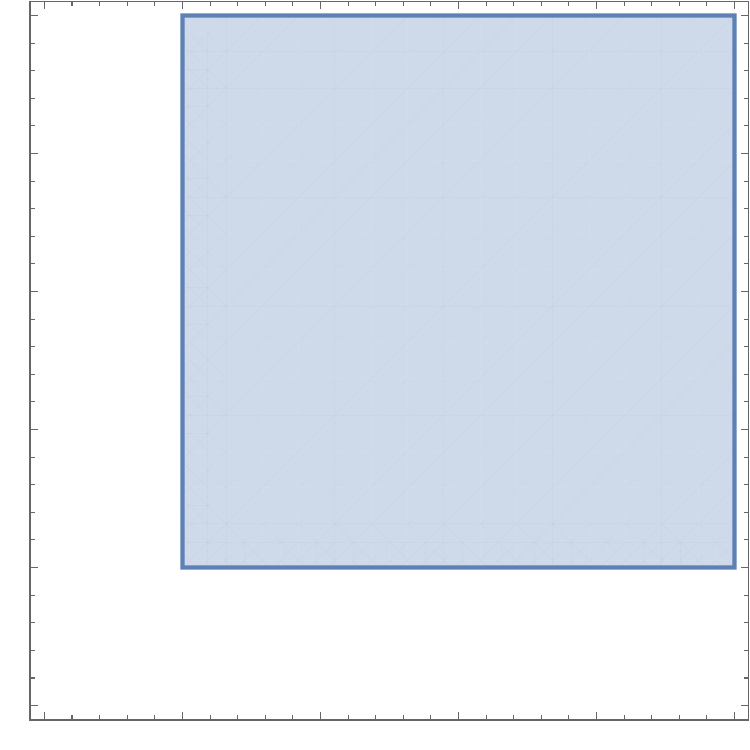}
    \end{tabular}
    \end{center}
where the left case occurs if $\eta\neq 0$ and the right if $\eta=0$.
\item 
$\mc R_2$ is one of the following:
    \begin{center}
    \begin{tabular}{lr}
    \includegraphics[width=5cm]{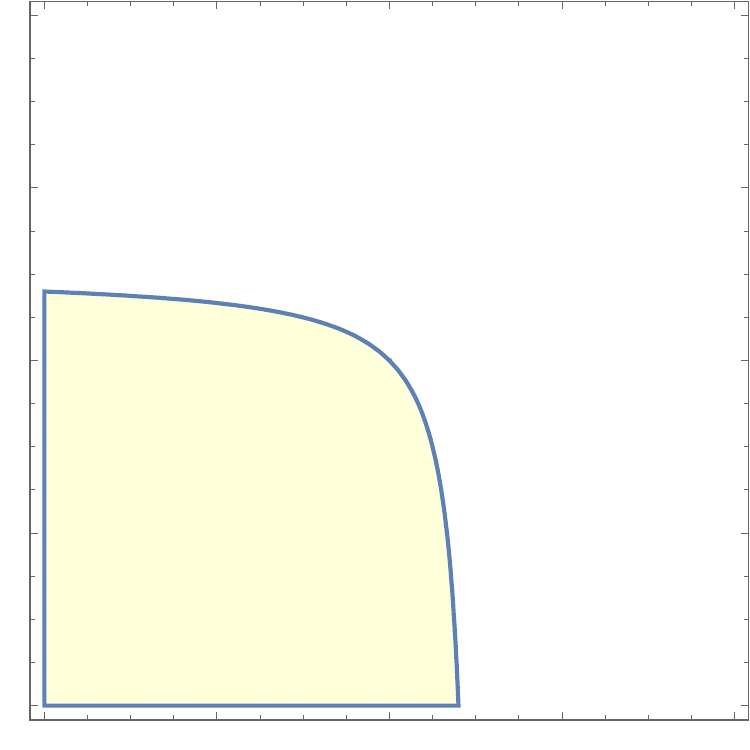}\hspace{1cm}
    &
    \includegraphics[height=5cm]{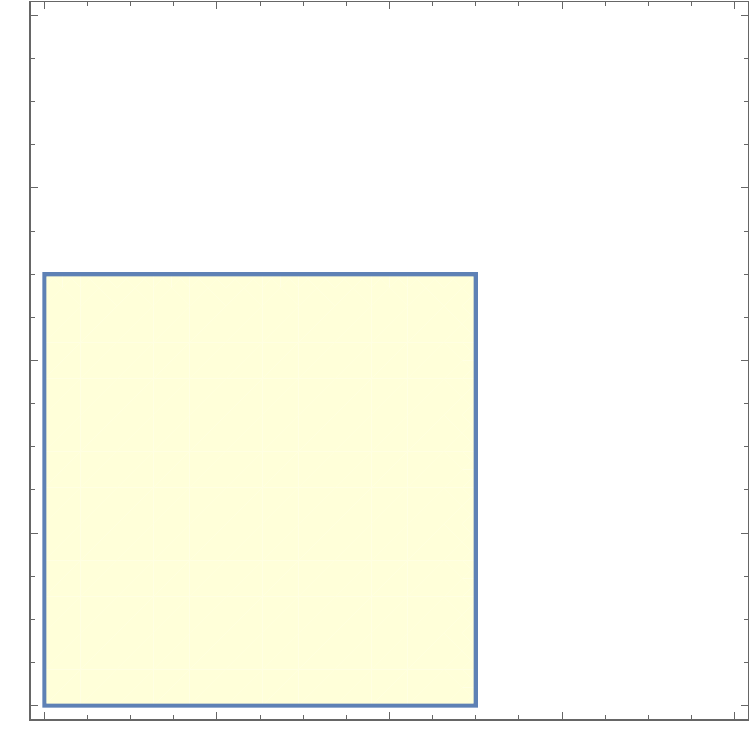}
    \end{tabular}
    \end{center}
where the left case occurs if $k_{12}\neq 0$ and the right if $k_{12}=0$. 
\end{itemize}
\item If $F_{22}$ is only positive semidefinite but not definite, then we show that
\eqref{151123-0822}
is equivalent to
\begin{align}
\label{311223-1603-v3}
\begin{split}
&\widetilde{\cM}(k;\beta)\succeq 0, 
\text{ the relations }
\eqref{131023-0847-equation}\text{ hold}, \eta=0\text{  
and }
\cH(\mc G(0,0))\text{ admits a }\RR\text{--rm}.
\end{split}
\end{align}
The latter statement is further equivalent to Theorem \ref{131023-0847}.\eqref{parabolic-pt1}. 
\item 
Assume that $F_{22}$ is positive definite and $H_{22}$ is only positive semidefinite but not definite. If:
\begin{itemize}
\item 
$u_1=0$,
then we show that \eqref{151123-0822} is equivalent to 
\eqref{311223-1603-v3}.
The latter statement is further equivalent to Theorem \ref{131023-0847}.\eqref{parabolic-pt2-1}.
\smallskip
\item 
$u_1>0$,
then we show that \eqref{151123-0822} is equivalent to 
\begin{align*}
\begin{split}
&\widetilde{\cM}(k;\beta)\succeq 0, 
\text{ the relations }
\eqref{131023-0847-equation}\text{ hold, }
\cF(\mc G(t_1,u_1)) 
\text{ and }\\
&
\cH(\mc G(t_1,u_1))\text{ admit a }
\cZ(x-y^2)\text{--rm and a }\RR\text{--rm, respectively.}
\end{split}
\end{align*}
The latter statement is further equivalent to Theorem \ref{131023-0847}.\eqref{parabolic-pt2-2}.
\smallskip
\item $u_1<0$, then \eqref{151123-0822} cannot hold.
\end{itemize}
\item
Assume that $F_{22}$ and $H_{22}$ are positive definite.
If:
\begin{itemize}
\item 
$\eta=0$,
then we show that \eqref{151123-0822} is equivalent to 
\eqref{311223-1603-v3}.
The latter statement is further equivalent to Theorem \ref{131023-0847}.\eqref{parabolic-pt3-1}.
\smallskip
\item 
$\eta\neq 0$,
then we show that \eqref{151123-0822} is equivalent to
$\mc R_1\cap \mc R_2\neq \emptyset$.
The latter statement is further equivalent to Theorem \ref{131023-0847}.\eqref{parabolic-pt3-2}.
\end{itemize}
\end{enumerate}
\end{remark}

\begin{proof}[Proof of Theorem \ref{131023-0847}]
Let $\mc R_1, \mc R_2$ be as in Remark \ref{proof-line+parabola}.
As explained in Remark \ref{proof-line+parabola}, 
Theorem \ref{131023-0847}.\eqref{131023-0847-pt1} 
is equivalent to
\eqref{151123-0822},
thus it remains to prove that
\eqref{151123-0822} is equivalent to
Theorem \ref{131023-0847}.\eqref{131023-0847-pt3}.\\

First we establish a few claims needed in the proof.
Claim 1 (resp.\ 2) describes $\mc R_1$ (resp.\ $\mc R_2$) concretely.\\

\noindent 
\textbf{Claim 1.} 
    Assume that $\widetilde{\mc{M}}(k;\beta)\succeq 0$.
    Then
    \begin{equation}
    \label{form-of-R1-parabolic}    
    \mc R_1
    =
    \big\{
    (t,u)\in \RR^2\colon 
    t\geq 0, u\geq 0, tu\geq \eta^2
    \big\}.
    \end{equation}
If $(t,u)\in \mc R_1$, we have
\begin{align}
    \label{rank-R1-parabolic}
    \Rank \cF(\mc G(t,u))=
    \left\{
    \begin{array}{rl}
            \Rank \cF(A_{\min}),&    
            \text{if }
                \eta=t=u=0, 
                \\[0.3em]
            \Rank \cF(A_{\min})+1,&    
            \text{if }
                (\eta=t=0, u>0)
                \text{ or }
                (\eta=u=0, t>0)
                \\[0.2em]
                &
                \text{ or }(\eta\neq 0, tu=\eta^2),\\[0.3em]
            \Rank \cF(A_{\min})+2,&    
            \text{if }
                tu>\eta^2.
    \end{array}
    \right.
\end{align}
where $A_{\min}$ is as in \eqref{A-min-parabolic}.\\

\noindent\textit{Proof of Claim 1.}
Note that 
\begin{align}
\label{151123-0850}
\begin{split}
    \mc G(\mathbf{t},\mathbf{u})
    &=
        A_{\min}
        +\eta \big(E_{1,k+1}^{(k+1)}+E_{k+1,1}^{(k+1)}\big)
        +\mathbf{t}E_{1,1}^{(k+1)}
        +\mathbf{u}E_{k+1,k+1}^{(k+1)}\\[0.3em]
    &=
        A_{\min}+
        \begin{pmatrix}
            \mathbf{t} & \mathbf{0}_{1,k-1} & \eta \\ 
            \mathbf{0}_{k-1,1} 
                & \mathbf{0}_{k-1} 
                & \mathbf{0}_{k-1,1} \\
            \eta & \mathbf{0}_{1,k-1} & \mathbf{u}
        \end{pmatrix}.
\end{split}
\end{align}
By Lemma \ref{071123-2008}, we have that
\begin{equation}   
\label{151123-0855}
\cF(\mc G(t,u))\succeq 0 
    \quad\Leftrightarrow\quad   
\mc G(t,u)\succeq A_{\min}
\end{equation}
Using \eqref{151123-0850}, \eqref{151123-0855} and the definition of $\mc R_1$, we have that
\begin{align}
\label{151123-0856}
    (t,u)\in \mc R_1
    \quad&\Leftrightarrow\quad 
    \begin{pmatrix}
        t & \eta \\ \eta & u
    \end{pmatrix}\succeq 0
    \quad\Leftrightarrow\quad 
    t\geq 0, u\geq 0, tu\geq \eta^2,
\end{align}
which proves \eqref{form-of-R1-parabolic}. 

To prove \eqref{rank-R1-parabolic} 
first note that by construction of $\cF(A_{\min})$, the columns $1$ and $X^k$ are in the span of the columns indexed by
    $\cC\setminus \vec{X}^{(0,k)}$.
Hence, there are vectors 
\begin{equation}
\label{151123-0901}
    v_1, v_2 \in \ker \cF(A_{\min})
\end{equation}
of the forms 
\begin{equation}
\label{candidates-v1-v2}
v_1=\begin{pmatrix}
    1 & \mathbf{0}_{1,k} & (\tilde v_1)^T
\end{pmatrix}^T\in \RR^{\frac{(k+1)(k+2)}{2}}
\quad\text{and}\quad
v_2=\begin{pmatrix}
    \mathbf{0}_{1,k} & 1 & (\tilde v_2)^T
\end{pmatrix}^T\in \RR^{\frac{(k+1)(k+2)}{2}}.
\end{equation}
Let $r:=\Rank 
\begin{pmatrix}
    t & \eta \\ \eta & u
\end{pmatrix}$.
Clearly,
\begin{equation}
\label{151123-0903} 
    \Rank \cF(\mc G(t,u))\leq \Rank \cF(A_{\min})+r.
\end{equation}
We separate three cases according to $r$.
\\

\noindent\textbf{Case 1: $r=0$.} In this case $t=u=\eta=0$ and $\mc G(0,0)=A_{\min}$. In this case \eqref{rank-R1-parabolic} clearly holds.\\

\noindent\textbf{Case 2: $r=1$.} In this case $tu=\eta^2$. Together with \eqref{151123-0856}, this is equivalent to
 $(\eta=t=0, u>0)$ or $(\eta=u=0, t>0)$ or $(\eta\neq 0, tu=\eta^2).$
By \eqref{151123-0903} and $\cF(\mc G(t,u))\succeq \cF(A_{\min})$ to
prove
\eqref{rank-R1-parabolic}, it suffices to  
find $v\in \ker \cF(A_{\min})$ and $v\notin \ker \cF(\mc G(t,u))$. Note that at least one of $v_1,v_2$ from \eqref{candidates-v1-v2}
is such a vector, since
$$
(v_1)^T\cF(\mc G(t,u))v_1=t
    \quad \text{and}\quad
(v_2)^T\cF(\mc G(t,u))v_2=u.\smallskip
$$

\noindent\textbf{Case 3: $r=2$.} In this case $tu>\eta^2$.
Note that
\begin{equation}
\label{030124-2103} 
\cF(\mc G(t,u))=
\cF\Big(\mc G\Big(\frac{\eta^2}{u},u\Big)\Big)+
\begin{pmatrix}
    t-\frac{\eta^2}{u}
\end{pmatrix}
\oplus
\mathbf{0}_{\frac{(k+1)(k+2)}{2}-1}
\succeq
\cF\Big(\mc G\Big(\frac{\eta^2}{u},u\Big)\Big).
\end{equation}
By Case 2, we have $\Rank \cF\Big(\mc G\Big(\frac{\eta^2}{u},u\Big)\Big)=\Rank \cF(A_{\min})+1$.
By \eqref{151123-0903} and \eqref{030124-2103},
to
prove
\eqref{rank-R1-parabolic}, it suffices to  
find $v\in \ker \cF\Big(\mc G\Big(\frac{\eta^2}{u},u\Big)\Big)$ and $v\notin \ker\cF(\mc G(t,u))$. 
We will check below, that $v_3$, defined by
    $$v_3=
        v_1-\frac{\eta}{u}v_2
        =
        \begin{pmatrix} 1 & \mathbf{0}_{1,k-1} & -\frac{\eta}{u} & (\tilde v_3)^T\end{pmatrix}^T\in \RR^{\frac{(k+1)(k+2)}{2}},
        $$
is such a vector.
This follows by 
$$
\cF\Big(\mc G\Big(\frac{\eta^2}{u},u\Big)\Big)v_3=\mathbf{0}_{\frac{(k+1)(k+2)}{2},1}
$$
and
$$
(v_3)^T\cF(\mc G(t,u))v_3=t-\frac{\eta^2}{u}>0.
$$

\noindent This concludes the proof of Claim 1.
\hfill$\blacksquare$\\

Note that

\begin{align*}
        \cH(\mc G(\mathbf{t},\mathbf{u}))
        &=
        \kbordermatrix{ 
                & \mathit{1} & \vec{X}^{(1,k-1)} & X^k \\
            \mathit{1}   & \beta_{0,0}-(A_{\min})_{1,1}-\mathbf{t} & (h_{12})^T &  \beta_{k,0}-(A_{\min})_{2,k}\\[0.2em]
            (\vec{X}^{(1,k-1)})^T
                & h_{12} & H_{22} & h_{23}\\[0.2em]
            X^k & \beta_{k,0}-(A_{\min})_{2,k} & (h_{23})^T & \beta_{2k,0}-(A_{\min})_{k+1,k+1}-\mathbf{u}}.
\end{align*}
Define the matrix function
\begin{align}
\label{K(G(t,u))}
\begin{split}
\mc K(\mathbf{t},\mathbf{u})
=
\cH(\mc G(\mathbf{t},\mathbf{u}))\big/H_{22}
&=
\cH(\widehat A_{\min})\big/H_{22}
-
\begin{pmatrix}
    \mathbf{t} & 0 \\ 0 & \mathbf{u}
\end{pmatrix}\\
&=
K-
\begin{pmatrix}
    \mathbf{t} & 0 \\ 0 & \mathbf{u}
\end{pmatrix}
=
\begin{pmatrix}
    k_{11}-\mathbf{t} & k_{12} \\ k_{12} & k_{22}-\mathbf{u}
\end{pmatrix}.
\end{split}
\end{align}

\noindent 
\textbf{Claim 2.} 
    Assume that $\widetilde{\mc{M}}(k;\beta)\succeq 0$.
    Then
    \begin{align}
    \label{form-of-R2-parabolic}    
    \begin{split}
    \mc R_2
    &=
    \big\{
    (t,u)\in \RR^2\colon 
    \mc K(t,u)\succeq 0
    \big\}\\
    &=\big\{
    (t,u)\in \RR^2\colon 
    t\leq k_{11}, u\leq k_{22}, (k_{11}-t)(k_{22}-u)\geq k_{12}^2
    \big\}.
    \end{split}
    \end{align}
If $(t,u)\in \mc R_2$, we have
\begin{align}
    \label{rank-R2-parabolic}
    \Rank \cH(\mc G(t,u))=
    \left\{
    \begin{array}{rl}
            \Rank H_{22},&    
            \text{if }
                k_{12}=0, t=k_{11}, u=k_{22}, 
                \\[0.3em]
            \Rank H_{22}+1,&    
            \text{if }
                (k_{11}-t)(k_{22}-u)=k_{12}^2, (t\neq k_{11}\text{ or }u\neq k_{22}),\\[0.3em]
            \Rank H_{22}+2,&    
            \text{if }
                (k_{11}-t)(k_{22}-u)>k_{12}^2.
    \end{array}
    \right.
\end{align}
where $A_{\min}$ is as in \eqref{A-min-parabolic}.\\

\noindent\textit{Proof of Claim 2.}
Permuting rows and columns of $\cH(\mc G(\mathbf{t},\mathbf{u}))$
we define
\begin{align*}
        \widetilde \cH(\mc G(\mathbf{t},\mathbf{u}))
        &=
        \kbordermatrix{ 
                & \mathit{1} & X^k & \vec{X}^{(1,k-1)} \\
            \mathit{1}   & \beta_{0,0}-(A_{\min})_{1,1}-\mathbf{t} & \beta_{k,0}-(A_{\min})_{2,k} &(h_{12})^T \\[0.2em]
            X^k & \beta_{k,0}-(A_{\min})_{2,k} & \beta_{2k,0}-(A_{\min})_{k+1,k+1}-\mathbf{u} & (h_{23})^T \\[0.2em]
            (\vec{X}^{(1,k-1)})^T
                & h_{12} & h_{23} & H_{22}}.
\end{align*}
Note that
    $$ 
    \cH(\mc G(t,u))\succeq 0 
    \quad\Leftrightarrow\quad
     \widetilde\cH(\mc G(t,u))\succeq 0
    $$
and
\begin{align}
\label{form-of-H-A-min}
        \cH(A_{\min})
        &=
        \kbordermatrix{ 
                & \mathit{1} & \vec{X}^{(1,k-1)} & X^k \\
            \mathit{1}   & \beta_{0,0}-(A_{\min})_{1,1} & (h_{12})^T &  \beta_{k,0}-(A_{\min})_{1,k+1}\\[0.2em]
            (\vec{X}^{(1,k-1)})^T
                & h_{12} & H_{22} & h_{23}\\[0.2em]
            X^k & \beta_{k,0}-(A_{\min})_{1,k+1} & (h_{23})^T & \beta_{2k,0}-(A_{\min})_{k+1,k+1}}.
\end{align}
By Lemma \ref{071123-2008}.\eqref{071123-2008-pt1}, $\cH(A_{\min})\succeq 0$.
Permuting rows and columns, this implies that
\begin{align*}
        \widetilde \cH(A_{\min})
        &=
        \kbordermatrix{ 
                & \mathit{1} & X^k & \vec{X}^{(1,k-1)} \\
            \mathit{1}   & \beta_{0,0}-(A_{\min})_{1,1} & \beta_{k,0}-(A_{\min})_{1,k+1} &(h_{12})^T \\[0.2em]
            X^k & \beta_{k,0}-(A_{\min})_{1,k+1} & \beta_{2k,0}-(A_{\min})_{k+1,k+1} & (h_{23})^T \\[0.2em]
            (\vec{X}^{(1,k-1)})^T
                & h_{12} & h_{23} & H_{22}}\succeq 0.
\end{align*}
By Theorem \ref{block-psd}, used for $(M,C)=(\widetilde \cH(A_{\min}),H_{22})$,
it follows that $H_{22}\succeq 0$ and $h_{12},h_{23}\in \cC(H_{22})$.
Let
$$
\mc L: S_2\to S_{k+1},\quad
\mc L(\mathbf{A})
=
\begin{pmatrix}
    \mathbf{A} & \left(\begin{array}{c} (h_{12})^T \\ (h_{23})^T \end{array}\right)\\
     \left(\begin{array}{cc} h_{12} & h_{23} \end{array}\right) & H_{22}
\end{pmatrix}.
$$
be a matrix function.
Using Theorem \ref{block-psd} again for 
$(M,C)=(\mc L(A),H_{22})$, it follows that 
\begin{equation}
    \label{161123-1422}
    \mc L(A)\succeq 0 
    \quad\Leftrightarrow\quad
    A\succeq 
    \begin{pmatrix}
    (h_{12})^T\\
    (h_{23})^T
\end{pmatrix}
(H_{22})^\dagger
\begin{pmatrix}
    h_{12} & h_{23}
\end{pmatrix}
\end{equation}
and
\begin{equation}
    \label{171123-1341}
    \Rank \cL(A)=\Rank H_{22}+\Rank \Big(A-\begin{pmatrix}
    (h_{12})^T\\
    (h_{23})^T
\end{pmatrix}
(H_{22})^\dagger
\begin{pmatrix}
    h_{12} & h_{23}
\end{pmatrix}\Big)
\end{equation}
Further, \eqref{161123-1422} implies that
\begin{align*}
  &\widetilde \cH(\mc G(t,u))\succeq 0
  \quad\Leftrightarrow\\[0.3em]
  &\Leftrightarrow\quad
    \begin{pmatrix}
        \beta_{0,0}-(A_{\min})_{1,1}-t & \beta_{k,0}-(A_{\min})_{2,k} \\
        \beta_{k,0}-(A_{\min})_{2,k} & \beta_{2k,0}-(A_{\min})_{k+1,k+1}-u       
    \end{pmatrix}
    -
    \begin{pmatrix}
    (h_{12})^T\\
    (h_{23})^T
\end{pmatrix}
(H_{22})^\dagger
\begin{pmatrix}
    h_{12} & h_{23}
\end{pmatrix}
    \succeq 0\\[0.3em]
    &\Leftrightarrow\quad
    \mc K(t,u)\succeq 0,
\end{align*}
where we use the definition \eqref{K(G(t,u))} of $\mc K(t,u)$ in the 
last equivalence. Moreover, 
$\Rank \widetilde{\mc H}(\mc G(t,u))=\Rank H_{22}+\Rank \mc K(t,u)$.
This proves \eqref{form-of-R2-parabolic} and \eqref{rank-R2-parabolic}.
\hfill$\blacksquare$\\

\noindent \textbf{Claim 3.} If $(t,u)\in \mc R_2\cap (\RR_+)^2$, then 
        $$tu\leq (\sqrt{k_{11}k_{22}}-\sign(k_{12})k_{12})^2=:p_{\max}.$$
    The equality is achieved if:
    \begin{itemize}
        \item $k_{12}=0$, in the point $(t,u)=(k_{11},k_{22})$.
        \smallskip
        \item $k_{12}>0$, in the point 
            $(t_-,u_-)=
            (
            k_{11}-\frac{k_{12}\sqrt{k_{11}}}{\sqrt{k_{22}}},
            k_{22}+\frac{k_{12}\sqrt{k_{22}}}{\sqrt{k_{11}}}
            )$.
        \smallskip
        \item $k_{12}<0$, in the point 
            $(t_+,u_+)=
            (
            k_{11}+\frac{k_{12}\sqrt{k_{11}}}{\sqrt{k_{22}}},
            k_{22}-\frac{k_{12}\sqrt{k_{22}}}{\sqrt{k_{11}}}
            )$.
    \end{itemize}
    Moreover, if $k_{12}\neq 0$, then for every 
    $p\in [0,p_{\max}]$
    there exists a point 
    $(\tilde t,\tilde u)\in \mc R_2\cap (\RR_+)^2$ such that 
    $\tilde t \tilde u=p$ and $(k_{11}-\tilde t)(k_{22}-\tilde u)=k_{12}^2$.
\bigskip

\noindent\textit{Proof of Claim 3.}
If $k_{12}=0$, then $(t,u)\in \mc R_2\cap (\RR_+)^2=[0,k_{11}]\times [0,k_{22}]$
and Claim 3 is clear.

Assume that $k_{12}\neq 0$. Then clearly $tu$ is maximized in some point satisfying 
$(k_{11}-t)(k_{22}-u)=k_{12}^2$.
Let $f(t):=t\big(k_{22}-\frac{k_{12}^2}{k_{11}-t}\big).$
We are searching for the maximum of $f(t)$ on the interval $[0,k_{11}]$.
The stationary points of $f$ are 
$t_{\pm}=k_{11}\pm\frac{k_{12}\sqrt{k_{11}}}{\sqrt{k_{22}}}.$ 
Then $u_\pm=k_{22}\mp\frac{k_{12}\sqrt{k_{22}}}{\sqrt{k_{11}}}$.
If $k_{12}>0$,
then $t_-\in [0,k_{11}]$ (note that $k_{11}k_{22}\geq k_{12}^2$ if 
$\mc R_2\cap (\RR_+)^2\neq \emptyset$). Further on, $t_-u_-=(\sqrt{k_{11}k_{22}}-k_{12})^2$.
Similarly, if $k_{12}<0$, then $t_{+}\in [0,k_{11}]$ and $t_+u_+=(\sqrt{k_{11}k_{22}}+k_{12})^2$. The moreover part follows by noticing that $f(0)=0$ and hence on the interval $[0,t_{\pm}]$, $f$ attains all values
between $0$ and $p_{\max}$.
\hfill$\blacksquare$\\

In the proof of Theorem \ref{131023-0847} we will need a few further observations:
\begin{itemize}
\item 
    Observe that 
    \begin{equation}
        \label{155123-1538}
        \begin{split}
        (\cH(\mc G(t,u)))_{\vec{X}^{(0,k-1)}}
        &=H_{1}-tE_{1,1}^{(k)},\\
        (\cH(\mc G(t,u)))_{\vec{X}^{(1,k-1)}}
        &=H_{22},\\
        (\cH(\mc G(t,u)))_{\vec{X}^{(1,k)}}
        &=H_{2}-uE_{k,k}^{(k)}.
        \end{split}
    \end{equation}
\item 
 We have
    \begin{equation}
    \label{upper-Schur-comp}
      (\cH(\mc G(t,u)))_{\vec{X}^{(0,k-1)}}\big/
      (\cH(\mc G(t,u)))_{\vec{X}^{(1,k-1)}}
      =
      H_1/H_{22}-t
      =
      t_1-t,
    \end{equation}
    where in the first equality we used
    \eqref{155123-1538} and in the
    second the definition of 
    $t_1$ (see \eqref{131023-0847-eq2}).
\item  We have   
    \begin{equation}
    \label{lower-Schur-comp}
      (\cH(\mc G(t,u)))_{\vec{X}^{(1,k)}}\big/
      (\cH(\mc G(t,u)))_{\vec{X}^{(1,k-1)}}
      =
      H_2/H_{22}-u
      =
      u_1-u,
    \end{equation}
    where in the first equality we used
    \eqref{155123-1538}
    and in the second the definition of 
    $u_1$ (see \eqref{131023-0847-eq2}).
\end{itemize}
\bigskip

    First we prove the implication 
    $\eqref{151123-0822}
    \Rightarrow
    \text{Theorem } \ref{131023-0847}.\eqref{131023-0847-pt3}$. 
    By the necessary conditions for the existence of a 
    $\cZ(p)$--rm \cite{CF04,Fia95,CF96}, 
    $\widetilde{\mc M}(k;\beta)$ must be psd and the relations
    \eqref{131023-0847-equation} must hold.
    By Lemma \ref{071123-2008}.\eqref{071123-2008-pt1},
    $\cF(A_{\min})\succeq 0$. 
    Hence, 
\begin{align}
\label{widehat-F-no-eta}
\begin{split}
&\widehat F
=
\widehat P
\cF(A_{\min})
(\widehat P)^T\\
&=
\kbordermatrix
{
& \mathit{1} & \widehat \cT\setminus \{\mathit{1},X^k\} & X^k & \cC\setminus \widehat\cT\\[0.2em]
\mathit{1} & (A_{\min})_{1,1} & (f_{12})^T & (A_{\min})_{1,k+1} & (f_{14})^T\\[0.2em]
(\widehat \cT\setminus \{\mathit{1},X^k\})^T
& f_{12} & F_{22} & f_{23} & F_{24}\\[0.2em]
X^k & (A_{\min})_{1,k+1} & (f_{23})^T & (A_{\min})_{k+1,k+1} & (f_{34})^T\\[0.2em]
\cC\setminus\widehat\cT & f_{14} & (F_{24})^T & f_{34} & F_{44}
}\succeq 0.,
\end{split}
\end{align}
where $\widehat P$ is as in \eqref{order-parabolic}.
In particular, $F_{22}\succeq 0$.
We separate two cases according to the invertibility of $F_{22}$.\\

    \noindent\textbf{Case 1: $F_{22}$ is not pd.}
    Let $\beta^{(c)}$ be a sequence corresponding to the moment matrix $\cF(\mc G(t_0,u_0))$.
    Let $\gamma=(\gamma_0,\ldots,\gamma_{4k})$ be a sequence defined by
    $\gamma_i=\beta^{(c)}_{\lfloor \frac{i}{2}\rfloor,i\; \text{mod}\; 2}$.
    Note that 
    $$
    \big(\widehat\cF(\mc G(t_0,u_0))\big)_{\widehat \cT\setminus \{\mathit{1},X^k\}}=
    (\widehat F)_{\widehat \cT\setminus \{\mathit{1},X^k\}}
    =F_{22}
    =A_{\widehat \gamma},
    $$
    where $\widehat\gamma=(\gamma_2,\ldots,\gamma_{4k-2})$.
    Since $F_{22}$ is not pd, it follows that there is a non-trivial column relation in $F_{22}$, which is also a column relation in $A_{\gamma}$ by Proposition \ref{extension-principle}.
    By Theorem \ref{131023-1211}, $\gamma$ has a $\RR$--rm, which implies by Theorem \ref{Hamburger}, that $A_{\gamma}$ is rg.
    Hence, the last column of $A_{\gamma}=\widehat\cF(\mc G(t_0,u_0))$ is in the span of the columns in 
    $\widehat \cT\setminus \{1,X^k\}$. It follows that
    \begin{equation}
        \label{171123-1959}
    \begin{pmatrix}
        (f_{12})^T\\
        F_{22}\\
        (f_{23})^T
    \end{pmatrix}
    (F_{22})^\dagger f_{23}=
    \begin{pmatrix}
    (A_{\min})_{2,k}\\[0.2em]
    f_{23}\\[0.2em]
    (A_{\min})_{k+1,k+1}+u_0
    \end{pmatrix}.
    \end{equation}
    On the other hand, by construction of $\widehat F$,
    the column $X^k$ is also in the span of the columns in 
    $\widehat \cT\setminus \{1,X^k\}$. Hence,
    \begin{equation}
    \label{171123-2003}
    \begin{pmatrix}
        (f_{12})^T\\
        F_{22}\\
        (f_{23})^T
    \end{pmatrix}
    (F_{22})^\dagger f_{23}=
    \begin{pmatrix}
    (A_{\min})_{1,k+1}\\[0.2em]
    f_{23}\\[0.2em]
    (A_{\min})_{k+1,k+1}
    \end{pmatrix}.
    \end{equation}
    By \eqref{171123-1959} and \eqref{171123-2003},
    it follows that $(A_{\min})_{2,k}=(A_{\min})_{1,k+1}$ or equivalently $\eta=0$, and $u_0=0$.
    Note that
    \begin{align}
    \label{measure-(0,0)-parabolic}
    \begin{split}
    \widehat\cF(\mc G(t_0,u_0))=\widehat\cF(\mc G(t_0,0))&\succeq 
    \widehat\cF(\mc G(0,0))=\cF(A_{\min}),\\
    \cH(\mc G(t_0,u_0))=\cH(\mc G(t_0,0))&\preceq 
    \cH(\mc G(0,0))=\cH(A_{\min}).
    \end{split}
    \end{align}
    Further on, 
    $\widehat\cF(A_{\min})$ has a $\cZ(x-y^2)$--rm 
    by Theorem \ref{131023-1211}
    and 
    $\cH(A_{\min})$ by Theorem \ref{Hamburger}. Indeed,
    the column $X^k$ of $\widehat\cF(A_{\min})$ is in the span of the others and since $\cH(\mc G(t_0,0))$ satisfies the conditions in Theorem \ref{Hamburger}, the same holds for $\cH(A_{\min})$. But then the property \eqref{measure-property} holds (note that $\eta=0$).
    This is the case Theorem \ref{131023-0847}.\eqref{parabolic-pt1}.
    \\
    
    \noindent\textbf{Case 2: $F_{22}$ is pd.}
    By Lemma \ref{071123-2008}.\eqref{071123-2008-pt1},
    $\cH(A_{\min})\succeq 0$ (see \eqref{form-of-H-A-min}). 
    In particular, $H_{22}\succeq 0$.
    We separate two cases according to the invertibility of
    $H_{22}$.\\

    \noindent\textbf{Case 2.1: $H_{22}$ is not pd.}
    By \eqref{lower-Schur-comp} and 
    Theorem \ref{Hamburger}, it follows that
    \begin{equation}
    \label{u1-u0-ineq}
        u_1= u_0.
    \end{equation}
    By \eqref{form-of-R1-parabolic}, 
    \begin{equation}
    \label{u0-nonneg}
        u_0\geq 0.
    \end{equation}
    We separate two cases according to the value of $u_1$.
    \\

    \noindent\textbf{Case 2.1.1: $u_1=0$.}
    By \eqref{u1-u0-ineq}, it follows 
    that $u_0=0$.
    Note that
    \begin{equation}
    \label{171123-2109}
    \big(\widehat\cF(\mc G(t_0,u_0))\big)_{
    \widehat \cT\setminus\{\mathit{1}\}
    }
    =
    \big(\widehat\cF(\mc G(t_0,0))\big)_{
    \widehat \cT\setminus\{\mathit{1}\}
    }
    =
    \big(\widehat F\big)_{\widehat \cT\setminus\{\mathit{1}\}}.
    \end{equation}
    Since in $\widehat F$ we have the column relation
    \eqref{171123-2003} by construction,
    \eqref{171123-2109} and Proposition \ref{extension-principle}
    imply that
    $$\big(\widehat\cF(\mc G(t_0,0))\big)_{
    \widehat \cT,\widehat \cT\setminus\{\mathit{1},X^k\}
    }
    (F_{22})^\dagger f_{23}
    =
    \big(\widehat\cF(\mc G(t_0,0))\big)_{
    \widehat \cT,\{X^k\}
    },
    $$
    or equivalently \eqref{171123-1959} with $u_0=0$.
    By \eqref{171123-1959} and \eqref{171123-2003},
    it follows that $(A_{\min})_{2,k}=(A_{\min})_{1,k+1}$ or equivalently $\eta=0$.
    This is the case Theorem \ref{131023-0847}.\eqref{parabolic-pt2-1}.
    \\
    
    \noindent\textbf{Case 2.1.2: $u_1>0$.}
    Since the column $X^k$ of $\cH(\mc G(t_0,u_1))$ is
    in the span of the columns in $\vec{X}^{(1,k-1)}$,
    it first follows
    by observing the first row of 
    $\cH(\mc G(t_0,u_1))$ that
    \begin{equation}
        \label{181123-0735}
        \beta_{k,0}-(A_{\min})_{2,k}=
                        (h_{12})^T(H_{22})^{\dagger} h_{23}.
    \end{equation}
    Further on,
    \begin{equation}
    \label{171123-2359}
     \cH(\mc G(t,u_1))\big/
      (\cH(\mc G(t,u_1)))_{\vec{X}^{(1,k)}}=
     (\cH(\mc G(t,u_1)))_{\vec{X}^{(0,k-1)}}\big/
      (\cH(\mc G(t,u_1)))_{\vec{X}^{(1,k-1)}}
      =
      t_1-t,
    \end{equation}
    where we used \eqref{upper-Schur-comp}
    in the second equality.
    By \eqref{171123-2359} and Theorem \ref{block-psd} used
    for $(M,C)=(\cH(\mc G(t,u_1)),(\cH(\mc G(t,u_1)))_{\vec{X}^{(1,k)}})$, it follows that
    $\cH(\mc G(t_1,u_1))\succeq 0$. 
    By Theorem \ref{Hamburger}, $\cH(\mc G(t_1,u_1))$ 
    admits a $\RR$--rm.
    Note that     
    \begin{align}
    \label{measure-(t1,u1)-parabolic}
    \begin{split}
    \widehat\cF(\mc G(t_0,u_0))=\widehat\cF(\mc G(t_0,u_1))&\preceq 
    \widehat\cF(\mc G(t_1,u_1)),
    \end{split}
    \end{align}
    where we used that $t_0\leq t_1$ by \eqref{171123-2359}.
    By Theorem \ref{131023-1211}, $(\widehat\cF(\mc G(t_1,u_1)))_{\widehat{\mc T}\setminus \{X^k\}}$
    must be pd. (Here we used that since $u_1>0$ and $F_{22}\succ 0$,
    it follows that $(\widehat\cF(\mc G(t_1,u_1)))_{\widehat{\mc T}\setminus \{1\}}\succ 0.$)
    Therefore 
    Claim 1 implies that $t_1>0$ and $t_1u_1\geq \eta^2$.
    Together with \eqref{181123-0735}, this is the case Theorem \ref{131023-0847}.\eqref{parabolic-pt2-2}.
    \\
    
    \noindent\textbf{Case 2.2: $H_{22}$ is pd.}
    We separate two cases according to the value of $\eta.$\\

    \noindent\textbf{Case 2.2.1: $\eta=0$.}
    By Lemma \ref{071123-2008}.\eqref{071123-2008-pt1},
    $\cH(A_{\min})\succeq 0$ (see \eqref{form-of-H-A-min}).
    
    If $\cH(A_{\min})$ does not admit a $\RR$--rm, it follows
    by Theorem \ref{Hamburger}, that
    $(\cH(A_{\min}))_{\vec{X}^{(0,k-1)}}$ is not pd
    and $u_1>0$. Equivalently,
    $$
    t_1=(\cH(A_{\min}))_{\vec{X}^{(0,k-1)}}
    \big/ H_{22}=0,
    $$
    which by \eqref{upper-Schur-comp} implies that $t_0=0$.
    By Theorem \ref{131023-1211}, since 
    $\widehat\cF(\mc G(t_0,u_0))
    =
    \widehat\cF(\mc G(0,u_0))
    $
    admits a $\cZ(x-y^2)$--rm, $F_{22}\succ 0$ and
    $(\widehat\cF(\mc G(0,u_0)))_{\widehat \cT\setminus \{X^k\}}$
    is not pd,
    it follows that $u_0=0$.
    But then $\cH(\mc G(t_0,u_0))=\cH(\mc G(0,0))=\cH(A_{\min})$
    does not admit a $\RR$--rm, which is a contradiction.
    
    Hence, $\cH(A_{\min})$ admits a $\RR$--rm, which is equivalent to 
    \eqref{measure-property} (using $\eta=0$).
    This is the case Theorem \ref{131023-0847}.\eqref{parabolic-pt3-1}.
    \\

    \noindent\textbf{Case 2.2.2: $\eta\neq 0$.}
    By \eqref{rank-R1-parabolic} it follows that
    $t_0u_0\geq \eta^2.$
    This fact and Claim 3 imply the second condition in 
    the case Theorem \ref{131023-0847}.\eqref{parabolic-pt3-2}.
    \\

    \noindent This concludes the proof of the implication 
    $\eqref{151123-0822}
    \Rightarrow
    \text{Theorem } \ref{131023-0847}.\eqref{131023-0847-pt3}$.\\

    Next we prove the implication 
    $\text{Theorem } \ref{131023-0847}.\eqref{131023-0847-pt3}
    \Rightarrow
    \eqref{151123-0822}
    $.
    We separate five cases according to the assumptions in 
    $\text{Theorem } \ref{131023-0847}.\eqref{131023-0847-pt3}$.\\

    \noindent\textbf{Case 1: $\text{Theorem } \ref{131023-0847}.\eqref{parabolic-pt1}$ holds.}
    By Lemma \ref{071123-2008}.\eqref{071123-2008-pt1},
    $\cF(A_{\min})\succeq 0$ and
    $\cH(A_{\min})\succeq 0$.
    Since $\eta=0$, both matrices have a moment structure. 
    Since by construction,
    the column $X^k$ of $\widehat\cF(A_{\min})$ is in the span of the others, it has a $\cZ(x-y^2)$--rm 
    by Theorem \ref{131023-1211}.
    Since $\cH(A_{\min})$ satisfies \eqref{measure-property} (using $\eta=0$), it admits a $\RR$--rm
    by Theorem \ref{Hamburger}. 
    This proves \eqref{151123-0822} in this case.
    \\

    \noindent\textbf{Case 2: $\text{Theorem } \ref{131023-0847}.\eqref{parabolic-pt2-1}$ holds.}
    By the same reasoning as in the Case 1 above,
    $\widehat\cF(A_{\min})$ has a $\cZ(x-y^2)$--rm.
    Since $u_1=0$, the column $X^k$ of $\cH(A_{\min})$
    is in the span of the other columns. 
    By Theorem \ref{Hamburger},
    $\cH(A_{\min})$ admits a $\RR$--rm.
    This proves \eqref{151123-0822} in this case.
    \\

    \noindent\textbf{Case 3: $\text{Theorem } \ref{131023-0847}.\eqref{parabolic-pt2-2}$ holds.}
    By \eqref{upper-Schur-comp}, \eqref{lower-Schur-comp} and the fourth assumption of \eqref{parabolic-pt2-2},
    it follows that $\cH(\mc G(t_1,u_1))$ is psd and the columns $1,X^k$ are in the span of the columns in $\vec{X}^{(1,k-1)}$. By Theorem \ref{Hamburger},
    $\cH(\mc G(t_1,u_1))$ admits a $\RR$--rm. 
    Since $(t_1,u_1)\in \mc R_1$ by \eqref{form-of-R1-parabolic} and the assumptions in 
    \eqref{parabolic-pt2-2}, it follows that $\widehat F(\mc G(t_1,u_1))$ is psd 
    and by construction, $\big(\widehat F(\mc G(t_1,u_1))\big)_{\widehat \cT\setminus \{X^k\}}$
    is pd. By Theorem \ref{131023-1211}, it has a $\cZ(x-y^2)$--rm.
    This proves \eqref{151123-0822} in this case.
    \\
    
    \noindent\textbf{Case 4: $\text{Theorem } \ref{131023-0847}.\eqref{parabolic-pt3-1}$ holds.}
    $\widehat\cF(A_{\min})$ has a $\cZ(x-y^2)$--rm
    and 
    $\cH(A_{\min})$ has a $\RR$--rm
    by the same reasoning as in the Case 1 above.
    This proves \eqref{151123-0822} in this case.\\

    \noindent\textbf{Case 5: $\text{Theorem } \ref{131023-0847}.\eqref{parabolic-pt3-2}$ holds.}
    We separate three cases according to the sign of $k_{12}$.
    \begin{itemize}
        \item If $k_{12}=0$, then by Claim 2, 
            $\cH(\mc G(k_{11},k_{22}))$ is psd and the column $X^k$
            is in the span of the previous ones. Since 
            $\cH(\mc G(0,0))=\cH(\widehat A_{\min})$
            is psd by assumption, it follows that $k_{11}\geq 0$
            and $k_{22}\geq 0$. Since $\eta\neq 0$ and $k_{11}k_{22}\geq \eta^2$ by \eqref{det-ineq}, it follows that $k_{11}>0$ and $k_{22}>0$.
            By Claim 1,
            $\widehat\cF(\mc G(k_{11},k_{22}))\succ 0$. By Theorem \ref{131023-1211}, it has a $\cZ(x-y^2)$--rm.
            This proves \eqref{151123-0822} in this case.
        \item 
            If $k_{12}>0$, then by Claim 3, $\cH(\mc G(t_-,u_-))$ is psd
            and $t_-u_-\geq \eta^2$.
            By construction, $\Rank \cH(\mc G(t_-,u_-))=k$
            and since $t_-<k_{11}$, it follows that
            $(\cH(\mc G(t_-,u_-)))_{\vec{X}^{(0,k-1)}}$
            is pd. 
            Hence, the column $X^k$ of $\cH(\mc G(t_-,u_-))$ is in the span of the others.
            By Theorem \ref{Hamburger}, $\cH(\mc G(t_-,u_-))$ admits a $\RR$--rm.
            By Claim 1 and $t_-u_-\geq \eta^2$,
            it follows that
            $\widehat\cF(\mc G(t_-,u_-))\succeq 0$.
            Since $t_->0$, it follows that 
            $\big(\widehat F(\mc G(t_-,u_-))\big)_{\widehat \cT\setminus \{X^k\}}$
    is pd. By Theorem \ref{131023-1211}, it has a $\cZ(x-y^2)$--rm.
            This proves \eqref{151123-0822} in this case.
        \item 
            If $k_{12}<0$, then the proof of \eqref{151123-0822} is analogous to the case     $k_{12}>0$ by replacing
            $(t_-,u_-)$ with $(t_+,u_+)$.\\
    \end{itemize}

     \noindent This concludes the proof of the implication
        $\text{Theorem } \ref{131023-0847}.\eqref{131023-0847-pt3}\Rightarrow\eqref{151123-0822}.
    $\\
    
    By now we established the equivalence
    $\eqref{131023-0847-pt1}
    \Leftrightarrow
    \eqref{131023-0847-pt3}$
    in Theorem $\ref{131023-0847}$.
    It remains to prove the moreover part.   
    We observe again the proof of the implication
    $\eqref{131023-0847-pt3}
    \Rightarrow
    \eqref{151123-0822}.$
    By Lemma \ref{071123-2008}.\eqref{071123-2008-pt3},
    \begin{equation}
        \label{181123-0849}
        \Rank \widetilde{\mc M}(k;\beta)
        =\Rank \widehat\cF(A_{\min})+
        \Rank \cH(A_{\min}).
    \end{equation}
    
    In the proofs of the implications
    Theorem $\ref{131023-0847}.\eqref{parabolic-pt1}\Rightarrow
    \eqref{151123-0822}$, 
    $\text{Theorem } \ref{131023-0847}.\eqref{parabolic-pt2-1}\Rightarrow
    \eqref{151123-0822}$
    and 
    $\text{Theorem } \ref{131023-0847}.\eqref{parabolic-pt3-1}\Rightarrow
    \eqref{151123-0822}$,
    we established that
    $\widehat\cF(A_{\min})$ and $\cH(A_{\min})$
    admit a $\cZ(x-y^2)$--rm and a $\RR$--rm, respectively.
    By Theorems \ref{Hamburger} and \ref{131023-1211}, there also exist
    a $(\Rank\widehat\cF(A_{\min}))$--atomic and a $(\Rank\cH(A_{\min}))$--atomic rm\textit{s}. 
    By \eqref{181123-0849}, $\beta$ has a $(\Rank \widetilde{\mc M}(k;\beta))$--atomic $\cZ(p)$--rm.

    Assume that
    $\text{Theorem } \ref{131023-0847}.\eqref{parabolic-pt2-2}$
    holds. We separate two cases according to the value of 
    $\eta$:
    \begin{itemize}
        \item $\eta=0$. We separate two cases according to the existence of a $\RR$--rm of $\cH(A_{\min})$:
        \smallskip
        \begin{itemize}
        \item 
        \textit{The last column of $\cH(A_{\min})$ is in the span of the previous ones.} Then as in the previous paragraph, $\widehat\cF(A_{\min})$ and $\cH(A_{\min})$
        admit a $(\Rank\widehat\cF(A_{\min}))$--atomic $\cZ(x-y^2)$--rm and a $(\Rank\cH(A_{\min}))$--atomic
        $\RR$--rm, respectively.
        Hence, $\beta$ has a $(\Rank \widetilde{\mc M}(k;\beta))$--atomic $\cZ(p)$--rm. 
        \smallskip
        \item 
        \textit{The last column of $\cH(A_{\min})$ is not in the span of the previous ones.}
        Since also $t_1>0$, it follows that 
        $\Rank \cH(A_{\min})=\Rank H_{22}+2$.
        But then 
        $\Rank \cH(\mc G(t_1,u_1))=\Rank H_{22}$
        and 
        $\Rank \widehat\cF(\mc G(t_1,u_1))=\Rank \widehat\cF(A_{\min})+2$ (see \eqref{rank-R1-parabolic}).
        This implies that 
        $\widetilde{\mc M}(\beta;k)$ admits a 
         $(\Rank \widetilde{\mc M}(k);\beta)$--atomic $\cZ(p)$--rm.
         \end{itemize}
         \smallskip
         \item $\eta\neq 0$. We separate two cases according to $\Rank\cH(A_{\min})$, which can be either 
         $\Rank H_{22}+2$ or $\Rank H_{22}+1$ (since $t_1>0$).
         \smallskip
         \begin{itemize}
             \item $\Rank \cH(A_{\min})=\Rank H_{22}+2$. 
             Then as in the second Case of the case $\eta=0$ above, in the point $(t_1,u_1)$ there is a $(\Rank \widetilde{\mc M}(k;\beta))$--atomic $\cZ(p)$--rm for $\beta$. (Note that $t_1u_1$ is automatically strictly larger than $\eta^2$,
         otherwise the measure was $(\Rank \widetilde{\mc M}(k;\beta)-1)$--atomic, which is not possible.)
            \smallskip
            \item $\Rank \cH(A_{\min})=\Rank H_{22}+1$.
            In this case we have
            \begin{align*}
            \Rank \cH(\mc G(t_1,u_1))
            +
            \Rank \widehat\cF(\mc G(t_1,u_1))
            &=
            \Rank H_{22}
            +
            \Rank \widehat\cF(\mc G(t_1,u_1))\\
            &=
            \left\{
            \begin{array}{rl}
                \Rank H_{22}+\Rank \widehat\cF(A_{\min})+1,&
                \text{if }t_1u_1=\eta^2,\\[0.3em]
                \Rank H_{22}+\Rank \widehat\cF(A_{\min})+2,&
                \text{if }t_1u_1>\eta^2,
            \end{array}
            \right.\\[0.3em]
             &=
            \left\{
            \begin{array}{rl}
                \Rank \widetilde {\mc M}(k;\beta),&
                \text{if }t_1u_1=\eta^2,\\[0.3em]
                \Rank \widetilde {\mc M}(k;\beta)+1,&
                \text{if }t_1u_1>\eta^2,
            \end{array}
            \right.
            \end{align*}
            where we used \eqref{rank-R1-parabolic}
            in the second and \eqref{181123-0849}
            in the third equality.
            Hence, $\beta$ has a $(\Rank \widetilde {\mc M}(k;\beta))$--atomic rm if $t_1u_1=\eta^2$
            and $(\Rank \widetilde {\mc M}(k;\beta)+1)$--atomic rm if $t_1u_1>\eta^2$.
            It remains to show that in the case $t_1u_1>\eta^2$,
            there does not exist a 
            $(\Rank \widetilde {\mc M}(k;\beta))$--atomic rm.
            Since $H_{22}$ is not pd and $u_1>0$, if 
            $\cH(\mc G(t',u'))$ has a $\RR$--rm, then $u'=u_1$.
            Since 
            $\eta\neq 0$, then 
            $\widehat\cF(\mc G(t',u_1))$ with a $\cZ(x-y^2)$--rm
            is at least $(\Rank \widehat\cF(A_{\min})+1)$--atomic (see \eqref{rank-R1-parabolic}).
            If $t'\neq t_1$, then $\Rank \cH(\mc G(t',u_1))=\Rank H_{22}+1$. Hence, 
            $$
            \Rank \cH(\mc G(t',u_1))+
            \Rank \widehat\cF(\mc G(t',u_1))
            \geq 
            (\Rank H_{22}+1)
            +
            (\Rank \widehat\cF(A_{\min})+1)
            =\Rank \widetilde{\mc M}(k;\beta)+1, 
            $$
            where we used \eqref{181123-0849}
            in the last equality.
         \end{itemize}
         \smallskip
         \item Assume that
    $\text{Theorem } \ref{131023-0847}.\eqref{parabolic-pt3-2}$
    holds. We separate two cases according to the value of 
    $k_{12}$.
    \begin{itemize}
        \item $k_{12}=0$. 
            We separate two cases according to $\Rank \cH(A_{\min})$,
            i.e., $\Rank \cH(A_{\min})\in \{k,k+1\}$. Note that
            $\Rank \cH(A_{\min})$ cannot be $k-1$,
            since $\eta\neq 0$ and $k_{12}=0$
            imply that $\big(\cH(A_{\min})/H_{22}\big)_{12}\neq 0$.
        \smallskip
                 \smallskip
         \begin{itemize}
             \item $\Rank \cH(A_{\min})=k+1$. 
             Then as in the second case of the case $\eta=0$ of $\text{Theorem } \ref{131023-0847}.\eqref{parabolic-pt2-2}$ above, in the point $(t_1,u_1)$ there is a $(\Rank \widetilde{\mc M}(k;\beta))$--atomic $\cZ(p)$--rm for $\beta$. (Note that $t_1u_1$ is automatically strictly larger than $\eta^2$,
         otherwise the measure was $(\Rank \widetilde{\mc M}(k;\beta)-1)$--atomic, which is not possible.) 
            \smallskip
            \item $\Rank \cH(A_{\min})=k$.
            In this case we have
            \begin{align*}
            \Rank \cH(\mc G(k_{11},k_{22}))
            +
            \Rank \cF(\mc G(k_{11},k_{22}))
            &=
            \left\{
            \begin{array}{rl}
                \Rank H_{22}+\Rank \cF(A_{\min})+1,&
                \text{if }k_{11}k_{22}=\eta^2,\\[0.3em]
                \Rank H_{22}+\Rank \cF(A_{\min})+2,&
                \text{if }k_{11}k_{22}>\eta^2,
            \end{array}
            \right.\\[0.3em]
             &=
            \left\{
            \begin{array}{rl}
                \Rank \widetilde {\mc M}(k;\beta),&
                \text{if }k_{11}k_{22}=\eta^2,\\[0.3em]
                \Rank \widetilde {\mc M}(k;\beta)+1,&
                \text{if }k_{11}k_{22}>\eta^2,
            \end{array}
            \right.
            \end{align*}
            where we used \eqref{rank-R1-parabolic}
            in the first and \eqref{181123-0849}
            in the second equality.
            Hence, $\beta$ has a $(\Rank \widetilde {\mc M}(k;\beta))$--atomic rm if $k_{11}k_{22}=\eta^2$
            and $(\Rank \widetilde {\mc M}(k;\beta)+1)$--atomic rm if $k_{11}k_{22}>\eta^2$.
            It remains to show that in the case $k_{11}k_{22}>\eta^2$,
            there does not exist a 
            $(\Rank \widetilde {\mc M}(k;\beta))$--atomic rm.
            Since $\eta\neq 0$, if $\cF(\mc G(t',u'))$ is psd, 
            it follows that $t'u'\geq \eta^2$ by \eqref{form-of-R1-parabolic}.
            But then if $\widehat\cF(\mc G(t',u'))$ also admits a $\cZ(x-y^2)$--rm,
            this rm 
            is at least $(\Rank \widehat\cF(A_{\min})+1)$--atomic (see \eqref{rank-R1-parabolic}).
            If $t'<k_{11}$ or $u'<k_{22}$, then 
            $\Rank \cH(\mc G(t',u'))\geq \Rank H_{22}+1$. Hence, 
            $$
            \Rank \cH(\mc G(t',u'))+
            \Rank \widehat\cF(\mc G(t',u'))
            \geq 
            (\Rank H_{22}+1)
            +
            (\Rank \widehat\cF(A_{\min})+1)
            =\Rank \widetilde{\mc M}(k;\beta)+1, 
            $$
            where we used \eqref{181123-0849}
            in the last equality.
         \end{itemize}
         \smallskip
        \item $k_{12}\neq 0$.
            We separate two cases according to $\Rank \cH(A_{\min})$,
            i.e.\ $\Rank \cH(A_{\min})\in \{k,k+1\}$. Note that
            $\Rank \cH(A_{\min})$ cannot be $k-1$,
            since otherwise 
            $\cH(\widehat A_{\min})/H_{22}
            =
            \begin{pmatrix}
            0 & \eta\\
            \eta & 0
            \end{pmatrix}
            $,
            which cannot be psd by $\eta\neq 0$.
             By Claim 3, there is a point $(\tilde t,\tilde u)\in \mc R_2\cap (\RR_+)^2$,
             such that $\tilde t\tilde u=\eta^2$ and 
             $(k_{11}-\tilde t)(k_{22}-\tilde u)=k_{12}^2$.
             By \eqref{rank-R1-parabolic} and \eqref{rank-R2-parabolic} we have 
            \begin{align*}
            \Rank \cH(\mc G(\tilde t,\tilde u))
            +
            \Rank \widehat\cF(\mc G(\tilde t,\tilde u))
            &=
            (\Rank H_{22}+1)
            +
            (\Rank \widehat\cF(A_{\min})+1)\\[0.3em]
             &=
            \left\{
            \begin{array}{rl}
                \Rank \widetilde {\mc M}(k;\beta),&
                \text{if }\Rank \cH(A_{\min})=k+1,\\[0.3em]
                \Rank \widetilde {\mc M}(k;\beta)+1,&
                \text{if }\Rank \cH(A_{\min})=k,
            \end{array}
            \right.
            \end{align*}
            where we used \eqref{181123-0849}
            in the second equality.
            It remains to show that in the case $\Rank \cH(A_{\min})=k$,
            there does not exist a 
            $(\Rank \widetilde {\mc M}(k;\beta))$--atomic rm.
            Since $\eta\neq 0$, if $\widehat\cF(\mc G(t',u'))$ is psd, 
            it follows that $t'u'\geq \eta^2$ by \eqref{form-of-R1-parabolic}.
            But then if $\widehat\cF(\mc G(t',u'))$ also admits a $\cZ(x-y^2)$--rm,
            this rm 
            is at least $(\Rank \widehat\cF(A_{\min})+1)$--atomic (see \eqref{rank-R1-parabolic}).
            Since $k_{12}\neq 0$, $\Rank \cH(\mc G(t',u'))\geq \Rank H_{22}+1$
            by \eqref{rank-R2-parabolic}. Hence, 
            $$
            \Rank \cH(\mc G(t',u'))+
            \Rank \widehat\cF(\mc G(t',u'))
            \geq 
            (\Rank H_{22}+1)
            +
            (\Rank \widehat\cF(A_{\min})+1)
            =\Rank \widetilde{\mc M}(k;\beta)+1, 
            $$
            where we used \eqref{181123-0849}
            in the last equality.
    \end{itemize}
    \end{itemize}
    This concludes the proof of the moreover part.

    Since for a $p$--pure sequence 
    with $\widetilde{\mc M}(k;\beta))\succeq 0$, \eqref{181123-0849} implies that $\cH(A_{\min})$ is pd, it follows by the moreover part that 
    the existence of a $\cZ(p)$--rm
    implies the existence of 
    a $(\Rank \widetilde{\mc M}(k;\beta))$--atomic $\cZ(p)$--rm.
\end{proof}

The following example demonstrates the use of Theorem \ref{131023-0847}
to show that there exists a bivariate $y(x-y^2)$--pure sequence $\beta$ of degree 6 with a positive semidefinite $\mc M(3)$ and without a $\cZ(y(x-y^2))$--rm.

\begin{example}\label{ex-line-plus-parabola-no-rm}
Let $\beta$ be a bivariate degree 6 sequence given by
\begin{align*}
    \beta_{00} & = \frac{1228153}{1372615},
    & \beta_{10} & =\frac{97}{10},
    & \beta_{01} & = \frac{21}{10},\\[0.2em]
      \beta_{20} & = \frac{2289}{10},
    &  \beta_{11} &= \frac{441}{10},
     &\beta_{02} & =\frac{91}{10},\\[0.2em]
    \beta_{30} & =\frac{67207}{10},
    & \beta_{21} & =\frac{12201}{10},
    & \beta_{12} & =\frac{455}{2},\\[0.2em]
     \beta_{03} & =\frac{441}{10},
    & \beta_{40} & =\frac{2142693}{10},
    & \beta_{31} & =\frac{376761}{10},\\[0.2em]
     \beta_{22} & =\frac{67171}{10},
    & \beta_{13} & =\frac{12201}{10},
    &     \beta_{04} & =\frac{455}{2},\\[0.2em]
    \beta_{50} &=
        \frac{71340727}{10},
    &\beta_{41} &=
        \frac{12313161}{10},
    &\beta_{32} &=
        \frac{428519}{2},\\[0.2em]
    \beta_{23} &=
        \frac{376761}{10},
    &\beta_{14} &=
        \frac{67171}{10},
    &\beta_{05} &=
        \frac{12201}{10},\\[0.2em]
    \beta_{60} &= 
        \frac{2438236509}{10},
    &\beta_{51} &=
        \frac{415998681}{10},
    &\beta_{42} &=
        \frac{71340451}{10},\\[0.2em]
    \beta_{33} &=
        \frac{12313161}{10},
    &\beta_{24} &=
        \frac{428519}{2},
    &\beta_{15} &=
        \frac{376761}{10},\\[0.2em]
    \beta_{06} &=
        \frac{67171}{10}.
\end{align*}
Assume the notation as in Theorem \ref{131023-0847}.
$\widetilde {\mc M}(3)$ is psd with the eigenvalues
 $\approx 2.51\cdot 10^8$, $\approx 47179$, $\approx 112.1$, $\approx 7.4$, 
 $\approx 1.11$, $\approx 0.1$, $\approx 0.03$, $\approx 0.0005$, $\approx 4.9\cdot 10^{-6}$, $0$,
and the column relation $Y^3=YX$.
We have that
$$
A_{\min}
=
\begin{pmatrix}
    \frac{5537}{9230} & \frac{91}{10} & \frac{455}{2} & \frac{61999553}{9230}\\[0.5em]
    \frac{91}{10} & \frac{455}{2} & \frac{67171}{10} & \frac{428519}{2}\\[0.5em]
    \frac{455}{2} & \frac{67171}{10} & \frac{428519}{2} & \frac{71340451}{10}\\[0.5em]
    \frac{61999553}{9230} & \frac{428519}{2} & \frac{71340451}{10} & \frac{450098209309}{1846}
\end{pmatrix}
$$
and so 
$$\eta=\frac{67171}{10}-\frac{61999553}{9230}=-\frac{72}{923}.$$
The matrices $F_{22}$ and $H_{22}$ are equal to:
\begin{align*}
F_{22}
&=\begin{pmatrix}
 \frac{91}{10} & \frac{441}{10} & \frac{455}{2} & \frac{12201}{10} & \frac{67171}{10} \\[0.5em]
 \frac{441}{10} & \frac{455}{2} & \frac{12201}{10} & \frac{67171}{10} & \frac{376761}{10} \\[0.5em]
 \frac{455}{2} & \frac{12201}{10} & \frac{67171}{10} & \frac{376761}{10} & \frac{428519}{2} \\[0.5em]
 \frac{12201}{10} & \frac{67171}{10} & \frac{376761}{10} & \frac{428519}{2} & \frac{12313161}{10} \\[0.5em]
 \frac{67171}{10} & \frac{376761}{10} & \frac{428519}{2} & \frac{12313161}{10} & \frac{71340451}{10} 
\end{pmatrix},\qquad
H_{22}
=\begin{pmatrix}
 \frac{7}{5} & \frac{18}{5} \\[0.5em]
 \frac{18}{5} & \frac{49}{5}
\end{pmatrix}.
\end{align*}
They are both pd with the eigenvalues $\approx 7.3\cdot 10^6$, $\approx 1987.6$, $\approx 5.6$, $\approx 0.099$, $\approx 0.0013$ and $\approx 11.1$, $\approx 0.068$, respectively.
The matrix $K$ is equal to
$$
K=
\begin{pmatrix}
    k_{11} & k_{12} \\ 
    k_{12} & k_{22}
\end{pmatrix}
=
\begin{pmatrix}
 \frac{6050329}{48143098510} & \frac{3}{95} \\[0.2em]
 \frac{3}{95} & \frac{4941414}{87685} 
\end{pmatrix}
$$
and thus
\begin{equation}
\label{020123-0913}
    (\sqrt{k_{11}k_{12}}-k_{12})^2-\eta^2=-0.0033<0.
\end{equation}
By Theorem \ref{131023-0847}, $\beta$ does not have a $\cZ(y(x-y^2))$--rm, since by \eqref{parabolic-pt3-2} of Theorem \ref{131023-0847},
\eqref{020123-0913} should be positive.
\end{example}



\begin{thebibliography}{Arv666}

\bibitem[Alb69]
{Alb69}
	A.\ Albert, 
		\textit{Conditions for positive and nonnegative definiteness in terms of pseudoinverses},
			SIAM J. Appl. Math. \textbf{17} (1969), 434--440.

\bibitem[Akh65]
{Akh65}
	N.I.\ Akhiezer.
		\textit{The classical moment problem and some related questions in analysis},
			New York: Hafner Publishing Co., 1965.

\bibitem[AK62]
{AhK62}
	N.I.\ Akhiezer, M.\ Krein.
		\textit{Some questions in the theory of moments},
			Transl.\ Math.\ Monographs 2. Providence: American Math.\ Soc., 1962.

\bibitem[BZ21]
{BZ21}
	A.\ Bhardwaj, A.\ Zalar. 
		\textit{The tracial moment problem on quadratic varieties}, 
			J.\ Math.\ Anal.\ Appl.\  \textbf{498} (2021).
	Available from:
				\url{https://doi.org/10.1016/j.jmaa.2021.124936}.

\bibitem[BW11]
{BW11}
	M.\ Bakonyi, H.J.\ Woerdeman, 
		\textit{Matrix Completions, Moments, and Sums of Hermitian Squares}, 
			Princeton University Press, Princeton, 2011.

\bibitem[Ble15]
{Ble15}
	G.\ Blekherman.
 		\textit{Positive Gorenstein ideals},
			Proc.\ Amer.\ Math.\ Soc.\ \textbf{143} (2015) 69--86.
	Available from:
				\url{https://doi.org/10.1090/S0002-9939-2014-12253-2}.
%
\bibitem[BF20]
{BF20}
	G.\ Blekherman, L.\ Fialkow.
		\textit{The core variety and representing measures in the truncated moment problem},
			Journal of Operator Theory \textbf{84} (2020) 185--209.

\bibitem[CGIK+]
{CGIK+}
	R.\ Curto, M.\ Ghasemi, M.\ Infusino, S.\ Kuhlmann.
	\textit{The truncated moment problems for unital commutative $\RR$-algebras}, 
		Journal of Operathor Theory, to appear.
	Available from:
			\url{https://arxiv.org/pdf/2009.05115.pdf}. 

\bibitem[CHM74]
{CHM74}
	D.\ Carlson, E.\ Haynsworth, T.\ Markham,
		\textit{A generalization of the Schur complement by means of the Moore-Penrose inverse},
			SIAM J.\ Appl.\ Math..\ \textbf{26(1)} (1974) 169--175.

\bibitem[CH69]
{CH69}
	D.\ Crabtree, E.\ Haynsworth, 
		\textit{An identity for the Schur complement of a matrix}. 
			Proc.\ Am.\ Math.\ Soc.\ \textbf{22} (1969) 364--366.

\bibitem[CF91]
{CF91}
    R.\ Curto, L.\ Fialkow, 
    \textit{Recursiveness, positivity, and truncated moment problems},
    Houston J.\ Math.\ 
    \textbf{17} (1991) 603--635.
    
\bibitem[CF96]
{CF96}
	 R.\ Curto, L.\ Fialkow,
\textit{Solution of the truncated complex moment problem
	 	for flat data},
   Mem.\ Amer.\ Math.\ Soc.\ \textbf{119} (1996).
   
\bibitem[CF02]
{CF02}
	 R.\ Curto, L.\ Fialkow,
\textit{Solution of the singular quartic moment problem},
		J.\ Operator Theory \textbf{48} (2002) 315--354.

\bibitem[CF04]
{CF04}
	R. Curto, L. Fialkow, \textit{Solution of the truncated parabolic moment problem},
		Integral Equations Operator Theory \textbf{50}  (2004), 169--196.

\bibitem[CF05]
{CF05}
	R. Curto, L. Fialkow, \textit{Solution of the truncated hyperbolic moment problem},
		Integral Equations Operator Theory \textbf{52} (2005) 181--218.  

\bibitem[CF05b]
{CF05b}
	R. Curto, L. Fialkow, \textit{Truncated $K$-moment problems in several variables},
		J.\ Operator Theory \textbf{54} (2005) 189--226.

\bibitem[CF08]
{CF08}
    R.\ Curto, L.\ Fialkow,
    \textit{An analogue of the Riesz-Haviland theorem for the truncated moment 
    problem}, 
    J.\ Funct.\ Anal.\  
    \textbf{225} (2008) 2709--2731.

\bibitem[CF13]
{CF13}
    R.\ Curto, L.\ Fialkow,
    \textit{Recursively determined representing measures for bivariate truncated moment sequences},
    J.\ Operator Theory 
    \textbf{70(2)} (2013) 401--436.

\bibitem[CFM08]
{CFM08}
	R. Curto, L. Fialkow, H. M. M\"oller, \textit{The extremal truncated moment problem}, 
		Integral Equations Operator Theory \textbf{60(2)} (2008) 177-200. 

\bibitem[CY14]
{CY14}
	R. Curto, S. Yoo, \textit{Cubic column relations in the truncated moment problems}, J.\ Funct.\ Anal.\ \textbf{266(3)} (2014) 1611--1626. 

\bibitem[CY15]
{CY15}
	R. Curto, S. Yoo, \textit{Non-extremal sextic moment problems}, J.  Funct. Anal. \textbf{269(3)} (2015) 758--780. 

\bibitem[CY16]
{CY16}
	R. Curto, S. Yoo, \textit{Concrete solution to the nonsingular quartic binary moment problem}, 
		Proc. Amer. Math. Soc. \textbf{144} (2016) 249--258.

\bibitem[Dan92]
{Dan92}
	J. Dancis, \textit{Positive semidefinite completions of partial hermitian matrices},
		Linear Algebra Appl. \textbf{175} (1992) 97--114.

\bibitem[DS18]
{DS18}
 	P.J.\ di Dio, K.\ Schm\"udgen. 
	\textit{The multidimensional truncated Moment Problem:Atoms, Determinacy, and Core Variety},
	J.\ Funct.\ Anal. \textbf{274} (2018) 3124--3148.
	Available from:
	\url{https://doi.org/10.1016/j.jfa.2017.11.013}.

\bibitem[Fia95]
{Fia95}
	L. Fialkow, \textit{Positivity, extensions and the truncated complex moment problem}, Contemporary Math. \textbf{185} (1995), 133--150.

\bibitem[Fia11]
{Fia11}
    L.\ Fialkow,
    \textit{Solution of the truncated moment problem with variety $y=x^3$},
    Trans.\ Amer.\ Math.\ Soc.\ 
    \textbf{363} (2011) 3133--3165. 

\bibitem[Fia15]
{Fia15}
	L. Fialkow, \textit{The truncated moment problem on parallel lines}, 
	In: Theta Foundation International Book Series of Mathematical Texts \textbf{20} (2015), 99--116.

\bibitem[Fia17]
{Fia17}
	L.\ Fialkow. 
		\textit{The core variety of a multisequence in the truncated moment problem},
		J.\ Math.\ Anal.\ Appl.\ \textbf{456} (2017) 946--969. 
	Available from:
	\url{https://doi.org/10.1016/j.jmaa.2017.07.041}.

\bibitem[FN10]
{FN10}
	L. Fialkow, J. Nie,  
		\textit{Positivity of Riesz functionals and solutions of quadratic and quartic
		moment problems}, J. Funct. An. \textbf{258} (2010), 328--356.

\bibitem[GJSW84]
{GJSW84}
	R. Grone, C. R. Johnson, E. M. S\'a, H. Wolkowicz, \textit{Positive definite completions of partial hermitian matrices},
	Linear Algebra Appl. \textbf{58} (1984), 109--124. 

\bibitem[Kim14]
{Kim14}
	D.P.\ Kimsey.
	\textit{The cubic complex moment problem},
		 Integral Equations Operator Theory \textbf{80} (2014) 353-–378.
	Available from:
			\url{https://doi.org/10.1007/s00020-014-2183-4}.
%
\bibitem[Kim21]
{Kim21}
	D.P.\ Kimsey.
	\textit{On a minimal solution for the indefinite truncated multidimensional moment problem},
		J.\ Math.\ Anal.\ Appl.\ \textbf{500} (2021)
	Available from:
			\url{https://doi.org/10.1016/j.jmaa.2021.125091}.

\bibitem[KN77]
{KN77}
	K.G.\ Krein, A.A.\ Nudelman.
		\textit{The Markov moment problem and extremal problems},
			Translations of Mathematical Monographs, Amer.\ Math.\ Soc.; 1977.

\bibitem[Lau05]
{Lau05}
	M. Laurent, 
		\textit{Revising two theorems of Curto and Fialkow on moment matrices},
			Proc. Amer. Math. Soc. \textbf{133} (2005), 2965--2976. 

\bibitem[Nie14]
{Nie14}
	J.\ Nie. 
		\textit{The $\mathcal{A}$-truncated $K$-moment problem},
		Found.\ Comput.\ Math. \textbf{14} (2014) 1243--1276. 

%
\bibitem[Sch17]
{Sch17}
	K.\ Schm\"udgen.
	 \textit{The moment problem},
		Graduate Texts in Mathematics 277. Cham: Springer; 2017.

\bibitem[Wol]
{Wol}
	Wolfram Research, Inc., Mathematica, Version 12.0, Wolfram Research, Inc., Champaign,
		IL, 2020.

\bibitem[Yoo17a]
{Yoo17a}
	S.\ Yoo, \textit{Sextic moment problems on 3 parallel lines}, Bull. Korean Math. Soc. \textbf{54} (2017), 299--318.

\bibitem[Yoo17b]
{Yoo17b}
	S.\ Yoo, \textit{Sextic moment problems with a reducible cubic column relation}, Integral Equations Operator Theory \textbf{88} (2017), 45--63.

\bibitem[YZ+]
{YZ+}
	S.\ Yoo, A.\ Zalar, 
 \textit{The truncated moment problem on reducible cubic curves II: 
Hyperbolic type relations}, in preparation.

\bibitem[Zal21]
{Zal21}
	A.\ Zalar. 
	\textit{The truncated Hamburger moment problem with gaps in the index set},
	Integral Equations Operator Theory \textbf{93} (2021) 36 pp. 
	Available from:
	\url{https://doi.org/10.1007/s00020-021-02628-6}.
 
\bibitem[Zal22a]
{Zal22a}
	A.\ Zalar. 
	\textit{The truncated moment problem on the union of parallel lines},
	Linear Algebra and its Applications \textbf{649} (2022) 186--239. 
	Available from:
 \url{https://doi.org/10.1016/j.laa.2022.05.008}.

 \bibitem[Zal22b]
{Zal22b}
	A.\ Zalar. 
	\textit{The strong truncated Hamburger moment problem with and without gaps},
	J. Math. Anal. Appl. \textbf{516} (2022) 21pp.
	Available from:
	\url{https://doi.org/10.1016/j.jmaa.2022.126563}.

 \bibitem[Zal23]
{Zal23}
	A.\ Zalar. 
	\textit{The truncated moment problem on curves $y=q(x)$ and $yx^\ell=1$},
	Linear and Multilinear Algebra (2023) 45pp.
	Available from:
	\url{https://doi.org/10.1080/03081087.2023.2212316}.

\bibitem[Zha05]
{Zha05}
F. Zhang, \textit{The Schur Complement and Its Applications}, 
	Springer-Verlag, New York, 2005.
 
\end{thebibliography}
\end{document}